\documentclass[3p,times]{elsarticle}

\usepackage{silence}
\usepackage{amsmath}

\usepackage{amssymb}
\usepackage{amsthm}
\usepackage{mathabx}

\usepackage{stmaryrd}
\usepackage{array}
\usepackage{tikz}
\usepackage{algorithmic}

\usepackage{enumitem}

\usepackage{url}
\usepackage{xcolor} 
\usepackage{hyperref}
\hypersetup{
    colorlinks=true,
    linkcolor=blue,
    filecolor=magenta,      
    urlcolor=cyan,
    }
\usepackage{cleveref}
\usepackage[figuresright]{rotating}
\usetikzlibrary{decorations.pathreplacing}

\newtheorem{theorem}{Theorem}[section]
\newtheorem{lemma}[theorem]{Lemma}
\newtheorem{corollary}[theorem]{Corollary}
\newtheorem{definition}[theorem]{Definition}
\newtheorem{proposition}[theorem]{Proposition}

\newtheorem{remark}[theorem]{Remark}

\newcommand{\mmodel}{{\bm M}}
\newcommand{\mframe}{{\bm F}}
\newcommand{\sub}[1]{\langle#1\rangle}
\newcommand{\llrr}[1]{\llbracket #1 \rrbracket}
\newcommand{\MSL}{\mathsf{MSL}}
\newcommand{\MISL}{\mathsf{MISL}}
\newcommand{\MSO}{\mathsf{MSO}}
\newcommand{\ML}{\mathsf{ML}}
\newcommand{\IMR}{\mathsf{IMR}}
\newcommand{\PDL}{\mathsf{PDL}}
\newcommand{\DEL}{\mathsf{DEL}}
\newcommand{\PAL}{\mathsf{PAL}}
\newcommand{\MIC}{\mathsf{MIC}}

\newcommand{\dep}{{\mathsf{dep}}}

\newcommand{\bis}{\mathrel{\mathchoice%
{\raisebox{.3ex}{$\,
  \underline{\makebox[.7em]{$\leftrightarrow$}}\,$}}%
{\raisebox{.3ex}{$\,
  \underline{\makebox[.7em]{$\leftrightarrow$}}\,$}}%
{\raisebox{.2ex}{$\,
  \underline{\makebox[.5em]{\scriptsize$\leftrightarrow$}}\,$}}%
{\raisebox{.2ex}{$\,
  \underline{\makebox[.5em]{\scriptsize$\leftrightarrow$}}\,$}}}}

\usepackage{bm}

\begin{document}

\begin{frontmatter}

\title{A modal approach towards substitutions}

\author[label1]{Yaxin Tu}
\author[label2]{Sujata Ghosh}
\author[label3,label4]{Fenrong Liu}
\author[label5,label6]{Dazhu Li}
\address[label1]{Department of Computer Science, Princeton University, Princeton, USA}
\address[label2]{Indian Statistical Institute, Chennai, India}
\address[label3]{The Tsinghua-UvA JRC for Logic, Department of Philosophy, Tsinghua University, Beijing, China}

\address[label4]{Institute for Logic, Language and Computation, University of Amsterdam, Amsterdam, The Netherlands}
\address[label5]{Institute of Philosophy, Chinese Academy of Sciences, Beijing, China}
\address[label6]{Department of Philosophy, University of Chinese Academy of Sciences, Beijing, China}

\begin{abstract}
Substitutions play a crucial role in a wide range of contexts, from analyzing the dynamics of social opinions and conducting mathematical computations to engaging in game-theoretical analysis. For many situations, considering one-step substitutions is often adequate. Yet, for more complex cases, iterative substitutions become indispensable. In this article, our primary focus is to study logical frameworks that model both single-step and iterative substitutions. We explore a number of properties of these logics, including their expressive strength, Hilbert-style proof systems, and satisfiability problems. Additionally, we establish connections between our proposed frameworks and relevant existing ones in the literature. For instance, we precisely delineate the relationship between single-step substitutions and the standard syntactic replacements commonly found in many classical logics. Moreover, special emphasis is placed on iterative substitutions. In this context, we compare our proposed framework with existing ones involving iterative reasoning, thereby highlighting the advantages of our proposal.

\end{abstract}

\begin{keyword}
Substitutions \sep replacements \sep iterative reasoning \sep modal logic \sep satisfiability problem 
\end{keyword}

\end{frontmatter}

\section{Substitutions as a ubiquitous mechanism}\label{sec:intro}

Many phenomena have substitutions as a core mechanism. Examples are abound in various fields, ranging from everyday social interactions to theoretical computations, including mathematical calculations, game-theoretical analysis as well as logic operations. In this article, instead of studying the concrete manifestations of substitutions in specific fields, we aim to explore logical frameworks to reason about substitutions themselves. This not only improves our understanding of a fundamental technique utilized across various fields, but also provides us with a uniform tool for these investigations. To set the groundwork for this discussion, let us first clarify our approach to substitutions.

There can be different kinds of substitutions. For an illustration, let us take the logical operation $\varphi[\psi/p]$ as an example, which is common in a number of classical frameworks, including propositional logic, modal logic and first-order logic. With this operation, a new formula is obtained by replacing all occurrences of $p$ in $\varphi$ with $\psi$. An important feature of the operation is that it does not affect the semantic extension of a fixed formula: for instance, given a valuation function $V$ and a propositional letter $q$, the extension $V(q)$ of $q$ is never changed by the operation. So, the operation $\varphi[\psi/p]$ is essentially a syntactic update. However, depending on specific applications, it is equally natural to work with substitutions on a semantic level. It enables us to change the semantic extensions of formulas, especially when we need to encode dynamic information with fixed formulas.\footnote{In what follows, to distinguish between these two approaches, we will often call the syntactic substitutions `{\em replacements}' and the semantic ones `{\em substitutions}'.}  To illustrate this and the ubiquity of substitutions, let us now present some concrete examples from different areas mentioned above.

\vspace{2mm}

{\bf {Scenario 1: single-step substitutions in belief diffusion}}\; Consider the social community  depicted in  \Cref{fig:example1}, where three agents, $a$, $b$ and $c$,  are friends of each other. Friendships  are represented by directed arrows. Also, we use atomic propositions  to annotate beliefs or opinions: for instance, in Stage 1 of \Cref{fig:example1}, agents $a$ and $b$ hold the belief $p$, while $c$ does not. As usual, modalities $\Box$ and  $\Diamond$ are used to describe the properties of friends.\footnote{For instance, agent $c$ has the property $\Box p$, i.e., {\em all $c's$ friends} have the belief $p$, while agent $a$ has the property $\Diamond p$, meaning that she has {\em at least one friend} with the belief $p$.} As suggested in e.g., \cite{liu_logical_2014,baltag-diffusion}, beliefs of agents can be affected by their friends (due to conformity or peer pressure, for instance).  Consider the following policy on the updates of beliefs:

\vspace{1.5mm}

\noindent{\em An agent  holds the belief $p$ in the next stage if, and only if, 
 (i) all her friends hold the belief $p$ at the current stage or (ii) she believes $p$ and some of her friends also have the same belief.}

\vspace{1.5mm} 

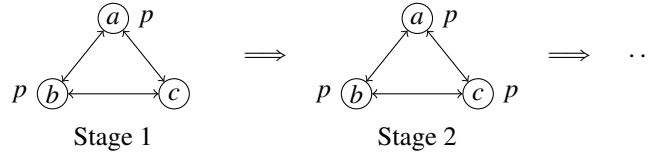
\begin{figure}[htbp]
    \centering
\begin{tikzpicture}
\node(s)[circle,draw,inner sep=0pt,minimum size=4mm,label=right:{$p$}] at (1,1) {$a$};
\node(u)[circle,draw,inner sep=0pt,minimum size=4mm,label=left:{$p$}] at (.2,0){$b$};
\node(v)[circle,draw,inner sep=0pt,minimum size=4mm] at (1.8,0){$c$};
\draw[<->](s) to  (u);
\draw[<->] (s) to  (v);
\draw[<->] (u) to  (v);

\node(v)[] at (3,0.5){$\Longrightarrow$};

\node(a)[circle,draw,inner sep=0pt,minimum size=4mm,label=right:{$p$}] at (5,1) {$a$};
\node(b)[circle,draw,inner sep=0pt,minimum size=4mm,label=left:{$p$}] at (4.2,0){$b$};
\node(c)[circle,draw,inner sep=0pt,minimum size=4mm,,label=right:{$p$}] at (5.8,0){$c$};

\node(v3)[] at (7,0.5){$\Longrightarrow$};

\node(v4)[] at (8,0.5){$\dots$};

\draw[<->](a) to  (b);
\draw[<->] (a) to  (c);
\draw[<->] (b) to  (c);
\node(1)[] at (1,-.6){Stage 1};
\node(2)[] at (5,-.6){Stage 2};
\end{tikzpicture}
    \caption{The belief status in a community that stabilizes.}
    \label{fig:example1}
\end{figure}

\noindent Formally, the mechanism of the policy has an essence of substitutions, which can be  written specifically as follows:
\[
p:= \Box p \lor (\Diamond p\land p)
\]
This implies that {\em the agents who will believe $p$  in the next stage are exactly those who currently has the property $\Box p\lor (\Diamond p\land p)$}. With the formal description, one can see  that the set of agents who believe $p$ at Stage 2 would become $\{a,b,c\}$,  evolving from the initial set of $\{a,b\}$. Afterwards, the belief status among the agents stablizes, which is exactly the transformation required  by the update policy.\footnote{Generally, the belief status of agents in a community may not always stabilize. To illustrate this, consider a community comprising two agents who are friends with each other, where one holds the belief $p$ and the other does not. Under the given update policy, their belief statuses will enter a loop, continually alternating rather than reaching a stable state. \cite{Benthem2015OscillationsLA} develops formal tools for the phenomena of oscillations. The substitution operations laid out above are important to understand these dynamics.}  So, the diffusion of beliefs in this way is essentially an embodiment of the abstract process of substitutions.

\vspace{2mm}

Given the ubiquity of phenomena involving substitutions, the single-step updates designed above might appear somewhat limited. Indeed, many situations call for {\em iterative} reasoning. Let us consider the following case to illustrate this point:

\vspace{2mm}

 {\bf  Scenario 2: iterative substitutions in backward induction}\; Consider a simple $2$-player board game, denoted as $\mathcal{G}= (\mathcal{B},s_0)$, involving {\em player  $0$} and {\em player  $1$}. $\mathcal{B} = (W,R)$ is a finite board, consisting of finitely many nodes $W$ and a binary relation $R$ among nodes, and $s_0\in W$ is a node indicating the current position of a {\em token} `$t$'. In each {\em round}, players move the token alternately, following the arrows in $R$, with player $0$ acting first. Moreover, a player {\em wins} whenever the opponent cannot proceed with a move. When the game runs infinitely, player $0$ wins.
\begin{figure}[htbp]
    \centering
\begin{tikzpicture}
\node(a1)[circle,draw,inner sep=0pt,minimum size=4mm] at (0,1.2) {$a$};
\node(b1)[circle,draw,inner sep=0pt,minimum size=4mm] at (1.2,1.2){$b$};
\node(c1)[circle,draw,inner sep=0pt,minimum size=4mm] at (0,0){$c$};
\node(d1)[circle,draw,inner sep=0pt,minimum size=4mm] at (1.2,0){$d$};
\draw[->](a1) to  (b1);
\draw[->](a1) to  (c1);
\draw[->](c1) to  (b1);
\draw[->](b1) to  (d1);
\end{tikzpicture}
    \caption{A game board $\mathcal{B}_1$}
    \label{fig:game_example_intro}
\end{figure}
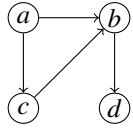

\noindent Let us consider a concrete example with the game board $\mathcal{B}_1$ depicted in \Cref{fig:game_example_intro}. It is simple to see that when the token $t$ is originally placed at $a$, player $0$ has a winning strategy: whenever it is her turn, she just needs to move the token `downwards'; while when $t$ is originally placed at $c$, player $1$ wins. More generally, on a finite $\mathcal{B}=(W,R)$, we can conduct backward induction to determine {\em winning positions} for players: a node $s$ is a winning position for a player if, and only if, when the initial position of the token $t$ is $s$, the player can win. For instance, for player $1$, the procedure goes as follows:\footnote{By the {\em Gale-Stewart Theorem} \cite{gale-theorem}, a node on a  board is either a winning position of player $0$ or a winning position of player $1$. So, the complement of the set of winning positions of player $1$ is the set of winning positions of player $0$.}

\vspace{2mm}

\noindent{\em Starting from the set $S_0$ of dead ends on the board, determine the set $S_1$ such that along arrows in $R$, any state $s_1$ of $S_1$ can only reach states $s'$  that  have access to some state in $S_0$; then based on $S_0\cup S_1$, determine the set $S_2$ such that $S_2$ is to $S_0\cup S_1$  as $S_1$ is to $S_0$; and repeat this procedure up to some finite times.}\footnote{Here $2^{|W|}$ is a proper upper bound for the repetitions in the backward induction process ($|W|$ is {\em the cardinal number} of $W$), since after $2^{|W|}$ times we will  repeat some previous step.}

\vspace{2mm}

\noindent  Again, this shares an obvious substitution core, but now it is more complex than that for Scenario 1 and has an iterative form as follows:  
\[
p:=\Box\bot; (p:= p\lor \Box\Diamond p)^*
\]
reading that {\em first set $p$ to be the dead ends and then iteratively re-define the $p$-states (of the next stage) as the states with the property $p\lor \Box\Diamond p$ (at the current stage) for finitely many times}. Such a form of iterated substitutions is first introduced in \cite{Johan-fixedpoint2021}, and for an illustration of the match between this formal definition and the backward induction procedure, see \Cref{fig:game_example}, and we defer the precise proof to  \Cref{sec:expressing-games}.\footnote{One may observe that the backward induction described can also be formalized by the modal $\mu$-calculus (see \Cref{sec:expressing-games}), and we will show that our proposal with the substitution operators are more capable to reason about the details of the game (\Cref{sec:relation-with-other-logics}).}

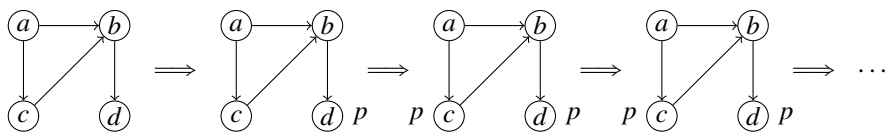
\begin{figure}[htbp]
    \centering
   \begin{tikzpicture}
\node(a1)[circle,draw,inner sep=0pt,minimum size=4mm] at (0,1.2) {$a$};
\node(b1)[circle,draw,inner sep=0pt,minimum size=4mm] at (1.2,1.2){$b$};
\node(c1)[circle,draw,inner sep=0pt,minimum size=4mm] at (0,0){$c$};
\node(d1)[circle,draw,inner sep=0pt,minimum size=4mm] at (1.2,0){$d$};
\draw[->](a1) to  (b1);
\draw[->](a1) to  (c1);
\draw[->](c1) to  (b1);
\draw[->](b1) to  (d1);

\node(v1)[] at (2,0.6){$\Longrightarrow$};

\node(a2)[circle,draw,inner sep=0pt,minimum size=4mm] at (2.8,1.2) {$a$};
\node(b2)[circle,draw,inner sep=0pt,minimum size=4mm] at (4,1.2){$b$};
\node(c2)[circle,draw,inner sep=0pt,minimum size=4mm] at (2.8,0){$c$};
\node(d2)[circle,draw,inner sep=0pt,minimum size=4mm,label=right:{$p$}] at (4,0){$d$};
\draw[->](a2) to  (b2);
\draw[->](a2) to  (c2);
\draw[->](c2) to  (b2);
\draw[->](b2) to  (d2);

\node(v2)[] at (4.8,0.6){$\Longrightarrow$};

\node(a3)[circle,draw,inner sep=0pt,minimum size=4mm] at (5.6,1.2) {$a$};
\node(b3)[circle,draw,inner sep=0pt,minimum size=4mm] at (6.8,1.2){$b$};
\node(c3)[circle,draw,inner sep=0pt,minimum size=4mm,label=left:{$p$}] at (5.6,0){$c$};
\node(d3)[circle,draw,inner sep=0pt,minimum size=4mm,label=right:{$p$}] at (6.8,0){$d$};
\draw[->](a3) to  (b3);
\draw[->](a3) to  (c3);
\draw[->](c3) to  (b3);
\draw[->](b3) to  (d3);

\node(v3)[] at (7.6,0.6){$\Longrightarrow$};

\node(a4)[circle,draw,inner sep=0pt,minimum size=4mm] at (8.4,1.2) {$a$};
\node(b4)[circle,draw,inner sep=0pt,minimum size=4mm] at (9.6,1.2){$b$};
\node(c4)[circle,draw,inner sep=0pt,minimum size=4mm,label=left:{$p$}] at (8.4,0){$c$};
\node(d4)[circle,draw,inner sep=0pt,minimum size=4mm,label=right:{$p$}] at (9.6,0){$d$};
\draw[->](a4) to  (b4);
\draw[->](a4) to  (c4);
\draw[->](c4) to  (b4);
\draw[->](b4) to  (d4);

\node(v4)[] at (10.4,0.6){$\Longrightarrow$};

\node(v5)[] at (11.2,0.6){$\dots$};
\end{tikzpicture}
    \caption{With the description $p:=\Box\bot; (p:= p\lor \Box\Diamond p)^*$, we have the transformation sequence 
$\emptyset\Rightarrow\{d\}\Rightarrow\{c,d\}\Rightarrow\{c,d\}\Rightarrow\dots$ for the $p$-states, 
which converges to $\{c,d\}$. This is exactly the backward induction reasoning of the winning positions of player $1$. }
    \label{fig:game_example}
\end{figure}

\vspace{2mm}

Although the above scenarios have been studied in different settings across existing literature, adopting a perspective centered on substitutions offers a unified approach to these diverse cases.
To understand the properties of the substitution cores in these examples, we are going to explore logical frameworks with both styles of substitutions: the single-step and the iterative computations. These will be integrated into standard modal logic as additional modalities. Beyond the widespread application of substitutions, as illustrated in the examples, they also play a crucial role in numerous other logical systems. The remainder of this section will discuss related work and provide a concise summary of the contributions made by this article.

\vspace{2mm}

{\bf Related work}\;   Here are several lines of research that are closely relevant to our work:

\vspace{1.5mm}

\noindent{\em Local fact change logic}\; Generalizing the idea of propositional control in \cite{propositional-control}, \cite{Declan} studies  the logic of local fact change that extends the standard modal logic with an operator $\bigcirc\varphi$ stating that after replacing the propositional letters true at the current state with some set of propositional letters, $\varphi$ is true at that state. As established in  \cite{Declan}, the logic has an undecidable satisfiability problem. Notably, there are several distinctions between this approach and our current work. A key difference is that the updates in this logic are implicit; they are not explicitly described by an update operator. Additionally, the set of propositional letters involved in an update could be infinite, and in general, the updates may be undefinable using a formula within the logic itself.

\vspace{1.5mm}

\noindent{\em Substitutions as syntactic modalities}\; To avoid making substitutions a notion  of the meta-language, \cite{tangwen} treats one-step substitutions as modalities and adds them to propositional logic, modal logic and first-order logic, respectively. For the resulting logics, \cite{tangwen} develops complete proof systems.  This method contrasts with our  framework, as the substitution modalities in \cite{tangwen} are syntactic; they function to alter the formulas themselves rather than to update the models.\footnote{See \cite{substitution-modality} for a proceedings version of \cite{tangwen}, which uses one-step substitutions as modalities in the setting of first-order logic.}

\vspace{1.5mm}

\noindent{\em Dynamic-epistemic logics}\; Another important tradition  is dynamic-epistemic logic $\DEL$   (e.g., \cite{hans-del,johan-del-book}) that removes states from models (e.g., the public announcement logic $\PAL$ \cite{pal}) or changes models with product updates  \cite{baltag-BMS}, to capture epistemic (or doxastic) updates induced by informative changes or factual changes. Moreover, \cite{IMR} studies a logic of iterative modal relativization that involves iterative generalizations of public announcement operators in $\PAL$ and \cite{term-modal-logic} explores epistemic logic with a predicate language that contains operators to update values of terms. 

 It is important to emphasize that the notion of substitution has been explored in $\DEL$, at both implicit level (using substitution as a way to define models) and the usual explicit level (incorporating expressions for substitution into logical languages). For instance, \cite{baltag-BMS} considers an action for `flip' that changes the truth set of a propositional variable in a model to its complement. More generally, \cite{Eijck,johan-communicaiton-change,hans-barteld-2008} use substitution as a method to capture how an event changes the valuation for propositional variables in models describing the same. They also discuss the simultaneous version of substitution that changes the valuation for two or more propositional variables at the same time. More specifically, \cite{public-assignment} adds explicit operators for single-step substitutions, called `public assignments', to $\PAL$  and discusses operators for simultaneous single-step substitution. A sound and complete Hilbert-style proof system for the logic extending $\PAL$ with the simultaneous single-step substitution operators is given in \cite{public-assignment-journal}, and \cite{kooi-substitution-formulas} explores its further extension with various operators for common knowledge (with models in which knowledge is represented by arbitrary relations). Finally, we would like to point out that \cite{public-assignment-journal} and  \cite{public-assignment} provide further observations that are relevant to some of our results (cf. Proposition \ref{prop:derivable-rule-1}, Proposition \ref{prop:derivale-rule-2}, Theorem \ref{thm:sub_equal_replacement_MSL}, Theorem \ref{thm:sub_equal_replacement_MISL}). In the course of this paper, we will compare the relevant concepts/results to indicate the detailed connections and show how we improve or complement the results in \cite{public-assignment,public-assignment-journal} (see footnote \ref{footnote1} and footnote \ref{footnote2}).

\vspace{1.5mm}

\noindent  Needless to say, in addition to the work mentioned above, there are quite a few milestone contributions that have close connections to our proposed frameworks, including the {\em modal $\mu$-calculus} \cite{10.1007/3-540-12896-4_370,Yde-mu-calculus}, {\em the infinitary modal logic} \cite{infinity-modal-logic-axiomatization}, the {\em propositional dynamic logic} \cite{open-minds} and {\em iterative model relativization} \cite{IMR}. These logics also involve iterated and/or infinite computations. For a better understanding of our proposal, we would explore the precise connections between these frameworks and ours in the later sections.

\vspace{2mm}

Our main goal in the article is to explore the framework, {\em Modal Iterative Substitution Logic $\MISL$}, with both single-step substitutions and their iterative generalizations. As a starting point, we will begin with a comparatively simpler setting, {\em Modal Substitution Logic $\MSL$}. The rest of the article is organized as follows:
\begin{itemize}
    \item  \Cref{sec:MSL} lays out the basics of $\MSL$, whose language extends the standard modal language with single-step substitution operators $\sub{p:=\psi}$. The resulting logic enjoys a number of desired properties. With the techniques of recursion axioms developed in dynamic-epistemic logic that enable us to reduce $\MSL$-formulas to standard modal logic formulas, we offer a complete proof system for $\MSL$. Based on this axiomatization, we also establish a decidability result for the satisfiability problem of the logic.
    
\item  Although our substitution operators are different from ordinary replacements $\varphi[\psi/p]$ on many levels, they are also closely related to each other, and a witness to this is given by the axiomatization in \Cref{sec:MSL}. In \Cref{sec:compare-two-styles-of-substitution} we determine the precise condition under which they are modally equivalent. 

\item After having an understanding of the basic framework, in \Cref{misl} we move on to the more intricate case $\MISL$. \Cref{sec:validities-misl} discusses various validities concerning the generalizations of the schematic principles of $\MSL$.\,
\Cref{sec:expressing-games} analyzes the applications of the resulting logic to our motivating examples and other relevant game-theoretic notions in the literature.

\item  Next, motivated by the observation on the applications of $\MISL$, we place our logic at a broader setting and study the formal connections between $\MISL$ and its cousins with the flavor of iterative computations in \Cref{sec:relation-with-other-logics}, which also demonstrates the merits of our proposal.

\item In \Cref{sec:properties-misl} we systematically investigate various properties of  $\MISL$, involving its expressive power and computational behavior.

\vspace{.5mm}

 W.r.t. expressiveness, although $\MISL$ is much more powerful than the standard modal logic,  \Cref{sec:bisimulation} shows that it is still invariant under the notion of standard bisimulation for modal logic. 

\vspace{.5mm}

On the other hand, the iterative substitution operators increase the computational complexity  of $\MISL$ drastically. \Cref{sec:undecidable} shows that the logic does not have the finite model property, and develops suitable upper bound and lower bound for the satisfiability problem of $\MISL$, which illustrate that  the logic is $\Sigma^1_1$-complete (hence, undecidable). Moreover, \Cref{sec:undecidable-tree} shows that the satisfiability problem of the logic is undecidable even when we confine ourselves to very simple classes of models (e.g., finite tree models).

\item Finally, \Cref{sec:conclusion} concludes the article with several directions that deserve to be explored in future.
\end{itemize}

\section{A basic logic for single-step substitutions: axiomatization and decidability}\label{sec:MSL}

Let us now present a formal proposal to explore the nature of substitutions. This section  focuses on a comparatively simple framework,  {\em Modal Substitution Logic} ($\MSL$), that only contains substitution operators for single-step updates. For the logic, we develop a complete Hilbert-style proof system and prove its decidability.

First of all, let us define  the language of $\MSL$, which extends {\em the language $\mathcal{L}_\ML$ of the standard modal logic $\ML$}  in the following manner:
\begin{definition}[Language $\mathcal{L}_\MSL$]
    Let ${\bf P}$ be a countable set of propositional letters. The {\em language $\mathcal{L}_\MSL$ of $\MSL$} is given by the following grammar: 
\[
\varphi::=p\mid\neg\varphi\mid\varphi\land\varphi\mid\Diamond\varphi\mid\sub{p:=\varphi}\varphi
\]
where $p\in\bf P$. Abbreviations $\top$, $\bot$,  $\lor$ and $\rightarrow$ are as usual, so are the box operators  $\Box$ and $[p:=\varphi]$ for $\Diamond$ and $\sub{p:=\varphi}$  respectively. 
\end{definition}

The readings of basic modal formulas are as usual, and the formula $\sub{p:=\psi}\varphi$ states that {\em after we replace the truth set of $p$ with the truth set of $\psi$, $\varphi$ is the case.} The truth set of a formula in a model consists of the states where the formula is true. This can be made precise after we introduce the semantics, and for now, let us first define some syntactic notions. Given a formula $\sub{p:=\psi}\varphi$, the propositional letter $p$ and formula $\varphi$ are the \emph{pivot} and the {\em scope} of the substitution operator $\sub{p:=\psi}$, respectively. Similarly for the box version $[p:=\psi]$ in a formula   $[p:=\psi]\varphi$. Also, we sometimes write $\sub{p:=\psi_1;q:=\psi_2}\varphi$  for $\sub{p:=\psi_1}\sub{q:=\psi_2}\varphi$, and $[p:=\psi_1;q:=\psi_2]\varphi$ for $[p:=\psi_1][q:=\psi_2]\varphi$. The notion of {\em subformulas} for $\MSL$, written as $\mathsf{Sub}(\varphi)$, extends that for the standard modal logic   with the following: 
\[
\mathsf{Sub}(\sub{p:=\psi}\varphi):=\mathsf{Sub}(\psi)\cup \mathsf{Sub}(\varphi)\cup\{\sub{p:=\psi}\varphi\}.
\]

As in the case for $\ML$, the formulas of $\mathcal{L}_\MSL$ are evaluated in {\em Kripke models} $\mmodel = (W, R, V)$, where $W$ 
is a non-empty set of states, $R\subseteq W\times W$ is a binary relation, and $V:{\bf P} \rightarrow 2^W$ 
is a valuation function. As usual, for any $w\in W$,   $(\mmodel,w)$ is a {\em pointed model}, and for simplicity, we often write $\mmodel,w$. Moreover, {\em frames} $\mframe=(W,R)$ are pairs without valuations. We also  write $w\in\mmodel$ or $w\in\mframe$ when $w\in W$.

\begin{definition}[Semantics]\label{definition:msl-model}
Let $\mmodel = (W, R, V)$ be a model and $w\in W$. {\em Truth of $\MSL$-formulas $\varphi$ at $(\mmodel,w)$}, written as $\mmodel,w\vDash_{\MSL} \varphi$, is given recursively as follows: 
\begin{center}
\begin{tabular}{r@{\quad $\Leftrightarrow$\quad}l}
$\mmodel,w\vDash_{\MSL} p$& $w\in V(p)$\\
$\mmodel,w\vDash_{\MSL} \neg\varphi$& $\mmodel,w\not\vDash_{\MSL} \varphi$\\
$\mmodel,w\vDash_{\MSL} \varphi_1\land\varphi_2$& $\mmodel,w\vDash_{\MSL}\varphi_1$ and $\mmodel,w\vDash_{\MSL} \varphi_2$\\
$\mmodel,w\vDash_{\MSL} \Diamond\varphi$& $\mmodel,v\vDash_{\MSL} \varphi$  for some $ v\in W $ s.t. $Rwv$ \\
$\mmodel,w\vDash_{\MSL} \sub{p:=\psi}\varphi$& $\mmodel|_{p:=\psi}, w\vDash_{\MSL} \varphi$
\end{tabular}
\end{center}
where $\mmodel|_{p:=\psi} = (W,R,V|_{p:=\psi})$ is obtained through model transformation from $\mmodel$ and may disagree with $\mmodel$ only on the valuation of $p$, where $V|_{p:=\psi}(p)=\{v\in W:\mmodel,v\vDash_{\MSL} \psi\}$.
\end{definition}



For any  model $\mmodel$ and formula $\varphi$, the {\em truth set  of $\varphi$ in $\mmodel$} is defined as $\llrr{\varphi}^{\mmodel}:=\{w\in\mmodel: \mmodel,w\vDash_{\MSL}\varphi\}$, and we often omit the superscript for the model $\mmodel$ when it is clear from the context. Also, notions such as {\em satisfiability}, {\em validity} and {\em logical consequence} are defined as usual \cite{blackburn2004modal}. We note that substitutions are, to some extent, unidirectional: we cannot trace back to the status of valuations before the substitution after the substitution is carried out. For instance, $\sub{p:=\Box q}\Box q\to p$ is not a validity, and for a counterexample, see \Cref{fig:counterexample}.

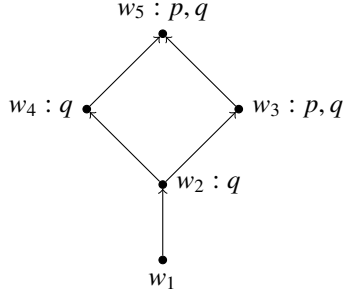
\begin{figure}[htbp]
    \centering
   \begin{tikzpicture}
\node(w1)[circle,draw,inner sep=0pt,minimum size=1mm,fill=black] at (0,0) [label=below:$w_1$]{};

\node(w2)[circle,draw,inner sep=0pt,minimum size=1mm,fill=black] at (0,1) [label=right:{$w_2: q$}]  {};

\node(w3)[circle,draw,inner sep=0pt,minimum size=1mm,fill=black] at (1,2) [label=right:{$w_3:p, q$}]   {};

\node(w4)[circle,draw,inner sep=0pt,minimum size=1mm,fill=black] at (-1,2) [label=left:{$w_4:q$}] {};

\node(w5)[circle,draw,inner sep=0pt,minimum size=1mm,fill=black] at (0,3) [label=above:{$w_5:p,q$}] {};

\draw[->](w1) to (w2);
\draw[->](w2) to (w3);
\draw[->](w2) to (w4);

\draw[->](w3) to (w5);
\draw[->](w4) to (w5);
\end{tikzpicture}
    \caption{A model $\mmodel$ in which $\sub{p:=\Box q}\Box q\to p$ is false at $w_1$. In the model, formula $\Box q$ is true everywhere. So, after the substitution, $p$ is true everywhere, and in particular, $\Box q$ is still true at $w_1$ in the new model. However, as we can see, $w_1$ does not have the property $p$ in  $\mmodel$.}
    \label{fig:counterexample}
\end{figure}
 
\vspace{2mm}

 {\bf Conventions}\; Many logical systems will be involved in the article.  For each logical system, we will use its name as a subscript of $\vDash$ and $\mathcal{L}$, respectively,  to express the satisfaction relation of the logic and its language: for instance, we wrote $\vDash_{\MSL}$ for the satisfaction relation of $\MSL$ and $\mathcal{L}_{\MSL}$ for its language.  

\vspace{2mm}

 $\MSL$ has many desired properties, among which the following is an important one:

\vspace{2mm}

\noindent{\bf Proof system} \; \Cref{tab:proof-system-msl} presents a calculus ${\bf MSL}$ for $\MSL$. It is a direct extension of proof system ${\bf ML}$ for $\ML$, which  consists of $\mathsf{(A1)}$-$\mathsf{(A3)}$, $\mathsf{(Dual)}$, $\mathsf{(K_{\Box})}$,  $\mathsf{(MP)}$ and $\mathsf{(Nec_{\Box})}$. The axiom $(\mathrm{K}_{[~]})$ shows that we can distribute an arbitrary substitution operator over an implication. Also, similar to the case for dynamic epistemic logic ($\DEL$) \cite{hans-del,johan-del-book,pal,baltag-BMS}, we have the recursion axioms that enable us to reduce $\MSL$-formulas to equivalent $\ML$-formulas:  $(\mathsf{R1})$ and $(\mathsf{R2})$ show how we can eliminate substitution operators when their scopes are propositional letters, $(\mathsf{R3})$ indicates that substitution operators are self-dual, $(\mathsf{R4})$ shows that we can distribute the operators over conjunction, and finally $(\mathsf{R5})$ ensures that the substitution operators can be pushed inside the scope of a standard modality $\Box$. In the same convention, we will write $\vdash_{\bf MSL}\varphi$ if $\varphi$ is provable in ${\bf MSL}$, and write $\vdash_{\bf ML}\varphi$ if $\varphi$ is provable in ${\bf ML}$. 

\begin{table} 
    \centering
\begin{tabular}{ll}
\cline{1-2}
\multicolumn{2}{c}{ Proof  system $\bf{MSL}$}\\
\cline{1-2}
\multicolumn{2}{l}{\;\;Basic axioms:} \\
\cline{1-2}
\quad $(\mathsf{A1})$ &\quad  $\varphi\to(\psi\to \varphi)$  \vspace{1mm}\\
\quad $(\mathsf{A2})$ &\quad  $(\varphi\to(\psi\to \chi))\to((\varphi\to \psi)\to(\varphi\to \chi))$    \vspace{1mm}\\
\quad $(\mathsf{A3})$ &\quad  $(\neg \varphi\to\neg \psi)\to(\psi\to \varphi)$  \vspace{1mm}\\
\quad $(\mathsf{Dual})$ &\quad  $\Diamond\varphi\leftrightarrow \neg\Box \neg\varphi$  \vspace{1mm} \\
\quad $(\mathsf{K}_{\Box})$ &\quad  $\Box(\varphi\to \psi)\to(\Box \varphi\to\Box \psi)$  \vspace{1mm} \\
\quad $(\mathsf{K}_{[~]})$ &\quad  $[p:=\chi](\varphi\to \psi)\to([p:=\chi] \varphi\to[p:=\chi] \psi)$\vspace{1mm} \\
\cline{1-2}
\multicolumn{2}{l}{\;\;Recursion axioms:} \\
\cline{1-2}
\quad $(\mathsf{R1})$ &\quad  $[p:=\chi]q\leftrightarrow q$,\quad given that $p$ is not $q$ \vspace{1mm} \\
\quad $(\mathsf{R2})$ &\quad  $[p:=\chi]p\leftrightarrow \chi$ \vspace{1mm} \\
\quad $(\mathsf{R3})$ &\quad  $ [p:=\chi]\neg\varphi\leftrightarrow \neg [p:=\chi]\varphi$ \vspace{1mm} \\
\quad $(\mathsf{R4})$ &\quad  $[p:=\chi](\varphi\land\psi) \leftrightarrow [p:=\chi]\varphi\land[p:=\chi]\psi$  \vspace{1mm} \\
\quad $(\mathsf{R5})$ &\quad $[p:=\chi]\Box\varphi \leftrightarrow \Box[p:=\chi]\varphi$\vspace{1mm} \\
\cline{1-2}
\multicolumn{2}{l}{\;\;Inference rules:} \\
\cline{1-2}
\quad $(\mathsf{MP})$ &\quad From $\varphi$ and $\varphi\to\psi$, infer $\psi$   \vspace{1mm}\\
\quad $(\mathsf{Nec}_{\Box})$ &\quad From $\varphi$, infer $\Box\varphi$   \vspace{1mm}\\
\quad  $(\mathsf{Nec}_{[~]})$ &\quad From $\varphi$, infer $[p:=\chi]\varphi$ \vspace{1mm}\\
\cline{1-2}
\end{tabular}    
    \caption{A proof  system $\bf{MSL}$ for $\MSL$}
    \label{tab:proof-system-msl}
\end{table}

To deepen the understanding of how this calculus works, let us explore some technical properties of the device. 
 \begin{proposition}\label{prop:derivable-rule-1}
The following rule is derivable:
\begin{center}
   From $\chi_1\leftrightarrow \chi_2$, infer $[p:=\varphi]\chi_1\leftrightarrow[p:=\varphi]\chi_2$.  
\end{center}
\end{proposition}
\begin{proof}
It suffices to show that   $\vdash_{\bf MSL}\chi_1\to\chi_2$ implies $\vdash_{\bf MSL}[p:=\varphi]\chi_1\to[p:=\varphi]\chi_2$. The proof is routine, and it goes as follows:
 \begin{center}
     \begin{tabular}{lll}
      $(a)$    & $\chi_1\to\chi_2$ & Assumption \\
      $(b)$    & $[p:=\varphi](\chi_1\to \chi_2)$ &  $(\mathsf{Nec}_{[~]})$,  $(a)$\\
      $(c)$ & $[p:=\varphi](\chi_1\to \chi_2)\to ([p:=\varphi]\chi_1\to [p:=\varphi]\chi_2)$ &  $(\mathrm{K}_{[~]})$\\
      $(d)$ & $ [p:=\varphi]\chi_1\to [p:=\varphi]\chi_2$ & $(\mathsf{MP})$,  $(b)$, $(c)$ \\
     \end{tabular}
 \end{center}
 This completes the proof.
\end{proof}

Next, as in the case for $\DEL$, with the help of the recursion axioms $(\mathsf{R1})$-$(\mathsf{R5})$, we can show the following:
\begin{proposition}\label{prop:MSL=ML-syntax}
For any $\varphi\in\mathcal{L}_{\MSL}$, there is a standard modal formula  $\varphi_\ML\in\mathcal{L}_{\ML}$ such that  $\vdash_{\bf MSL}\varphi\leftrightarrow \varphi_\ML$.
\end{proposition}
 
\begin{proof}
Let $\varphi$ be an $\MSL$-formula. Starting from an innermost occurrence of a substitution operator, we can push it inside Boolean connectives and the standard modality $\Box$, which can finally give us a formula in which the substitution operator is eliminated. Repeating this procedure,  we can finally get a standard modal formula $\varphi_{\ML}$ that is provably equivalent to the original $\varphi$, for which we may need to use propositional logic and the rule in Proposition \ref{prop:derivable-rule-1}.
\end{proof}

Then, let us note the following:
\begin{proposition}\label{prop:derivale-rule-2}
   The following rule is derivable:
\begin{center}
    From $\chi_1\leftrightarrow \chi_2$, infer  $[p:=\chi_1]\varphi\leftrightarrow[p:=\chi_2]\varphi$.
\end{center} 
\end{proposition}

\begin{proof}
 It is by induction on $\varphi$.

 \vspace{1.5mm}

 (1). Formula $\varphi$ is a propositional variable $q$. There are two different cases. 
 
 First, let us assume that $q$ is distinct from $p$.  By the axiom $(\mathsf{R1})$, $[p:=\chi_1]\varphi\leftrightarrow q$ and $[p:=\chi_2]\varphi\leftrightarrow q$. Now it is easy to see that $[p:=\chi_1]\varphi\leftrightarrow [p:=\chi_2]\varphi$.

 Next, assume that $q$ is exactly $p$. By the axiom $(\mathsf{R2})$, $[p:=\chi_1]\varphi\leftrightarrow \chi_1$ and $[p:=\chi_2]\varphi\leftrightarrow \chi_2$. By assumption, we have $\chi_1\leftrightarrow\chi_2$, which then gives us $[p:=\chi_1]\varphi\leftrightarrow [p:=\chi_2]\varphi$, as needed.

 \vspace{1.5mm}

 (2). Formula $\varphi$ is $\neg\psi$. By induction hypothesis, $[p:=\chi_1]\psi\leftrightarrow [p:=\chi_2]\psi$. So, $\neg [p:=\chi_1]\psi\leftrightarrow \neg [p:=\chi_2]\psi$. It follows from the axiom $(\mathsf{R3})$ that $\neg [p:=\chi_1]\psi\leftrightarrow  [p:=\chi_1]\neg \psi$ and $\neg [p:=\chi_2]\psi\leftrightarrow  [p:=\chi_2]\neg \psi$. By propositional logic, $[p:=\chi_1]\varphi\leftrightarrow [p:=\chi_2]\varphi$.

 \vspace{1.5mm}

 (3). Formula $\varphi$ is $\varphi_1\land\varphi_2$. By induction hypothesis, for each $i\in\{1,2\}$, we have $[p:=\chi_1]\varphi_i\leftrightarrow [p:=\chi_2]\varphi_i$. Using propositional logic, we can obtain  $[p:=\chi_1]\varphi_1\land [p:=\chi_1]\varphi_2 \leftrightarrow [p:=\chi_2]\varphi_1\land [p:=\chi_2]\varphi_2$. It follows from  $(\mathsf{R4})$ that  $[p:=\chi_1]\varphi_1\land [p:=\chi_1]\varphi_2\leftrightarrow [p:=\chi_1]\varphi$ and that $[p:=\chi_2]\varphi_1\land [p:=\chi_2]\varphi_2\leftrightarrow [p:=\chi_2]\varphi$. Consequently, $[p:=\chi_1]\varphi\leftrightarrow [p:=\chi_2]\varphi$, as desired.

 \vspace{1.5mm}

 (4). Formula $\varphi$ is $\Box\psi$.  By induction hypothesis, $[p:=\chi_1]\psi\leftrightarrow [p:=\chi_2]\psi$. Then, using propositional logic and the rule $(\mathsf{Nec}_{\Box})$, we can get $\Box([p:=\chi_1]\psi\leftrightarrow [p:=\chi_2]\psi)$. By $(\mathsf{K}_{\Box})$,   $(\mathsf{MP})$ and propositional logic, we have $\Box[p:=\chi_1]\psi\leftrightarrow \Box[p:=\chi_2]\psi$. It follows from axiom $(\mathsf{R5})$ that $\Box[p:=\chi_1]\psi\leftrightarrow [p:=\chi_1]\Box\psi$ and that  $\Box[p:=\chi_2]\psi\leftrightarrow [p:=\chi_2]\Box\psi$. So, $[p:=\chi_1]\varphi\leftrightarrow [p:=\chi_2]\varphi$.

 \vspace{1.5mm}

 (5). Formula $\varphi$ is $[q:=\psi_1]\psi_2$. Now, it follows from Proposition \ref{prop:MSL=ML-syntax} that there is a  $\theta\in\mathcal{L}_{\ML}$ with $[q:=\psi_1]\psi_2\leftrightarrow \theta$. For this $\theta$, based on the arguments in (1)-(4), it holds that $[p:=\chi_1]\theta\leftrightarrow [p:=\chi_2]\theta$.  We know from Proposition \ref{prop:derivable-rule-1} that $[p:=\chi_1]\theta\leftrightarrow [p:=\chi_1]\varphi$ and that $[p:=\chi_2]\theta\leftrightarrow [p:=\chi_2]\varphi$. Then, $[p:=\chi_1]\varphi\leftrightarrow [p:=\chi_2]\varphi$, as needed.
\end{proof}

 The inference rules given in Proposition \ref{prop:derivable-rule-1} and Proposition \ref{prop:derivale-rule-2} have there own uses, even though not explicitly mentioned here.\footnote{\label{footnote1}As mentioned in Section \ref{sec:intro}, \cite{public-assignment-journal} extends $\PAL$ with substitution operators, and it proves that the two inference rules preserve {\em validity} of formulas. Our results can be seen as their syntactic counterparts.} Now, it is easy to check the soundness of the proof system:
\begin{proposition}[Soundness of ${\bf MSL}$]\label{prop:soundness-msl}
  The proof system ${\bf MSL}$ is sound for $\MSL$.  
\end{proposition}

With the soundness, we can show the following:
 \begin{proposition}\label{prop:MSL=ML-semantics}
For any $\varphi\in\mathcal{L}_{\MSL}$, there is a  formula  $\varphi_\ML\in\mathcal{L}_{\ML}$ such that for any $(\mmodel,s)$,  it holds that 
 \begin{center}
      $\mmodel,s\vDash_{\MSL}  \varphi$ \; iff\;  $\mmodel,s\vDash_{\ML} \varphi_\ML$. 
 \end{center}
\end{proposition}

\begin{proof}
As indicated by Proposition \ref{prop:MSL=ML-syntax}, for any $\mathcal{L}_{\MSL}$, there is a  standard modal formula $\varphi_{\ML}$ that is provably equivalent to the original $\varphi$. Now, by the soundness of the calculus ${\bf MSL}$,  formulas $\varphi$ and $\varphi_{\ML}$ are equivalent, as required. 
\end{proof}

 Consequently, although  $\MSL$ is a direct extension of $\ML$, the two logics are as expressive as each other. With the result above, we can show the completeness of the proof system:
 
\begin{theorem}[Completeness of ${\bf MSL}$]\label{theorem:completeness-msl}
    The proof system  ${\bf MSL}$ is complete for $\MSL$. 
\end{theorem}

\begin{proof}
Consider an $\MSL$-formula $\varphi$ such that $\not\vdash_{\bf MSL}\varphi$. With the help of the recursion axioms, we know that there is some $\varphi_{\ML}\in\mathcal{L}_{\ML}$ that is provably equivalent to $\varphi$. So, $\not \vdash_{\bf MSL} {\varphi_{\ML}}$. Since ${\bf MSL}$ is a direct extension of ${\bf ML}$, we have $\not \vdash_{\bf ML} {\varphi_{\ML}}$. As ${\bf ML}$ is a complete calculus for $\ML$, we have $\not\vDash_{\ML} {\varphi_{\ML}}$. Now, using \Cref{prop:MSL=ML-semantics}, it holds that $\not \vDash_{\MSL} \varphi$. So, the calculus ${\bf MSL}$ is complete w.r.t. the class of models.
\end{proof}

Thus, ${\bf MSL}$ is a desired proof system for $\MSL$.\footnote{For more discussion on the axiomatization of $\DEL$-style logics, we refer to \cite{yanjing-PAL}.} The reasoning in the proof for \Cref{prop:MSL=ML-semantics} also illustrates that we can encode the satisfiability problem for $\MSL$ with that for $\ML$, and since the latter is decidable \cite{open-minds}, we have the following as a corollary:
\begin{corollary}[Decidability of $\MSL$]\label{coro:decidability-msl}
The satisfiability problem for $\MSL$ is decidable. 
\end{corollary} 
 
So far, we have seen that $\MSL$ enjoys many desired properties. The techniques used to achieve them are inspired by those employed in $\DEL$. Also, we showed that with respect to expressiveness, the logic is equivalent to $\ML$, and so is equivalent to many landmark paradigms of $\DEL$ (e.g., the public announcement logic $\PAL$). But it is worth noting that we do not necessarily need $\ML$ as a bridge to link them: as we will see, there is a natural and direct relation between our substitution operators and announcement operators of $\PAL$ (\Cref{def:IMR_to_MISL}). These suggest that our proposal could also be promising for modelling situations involving knowledge updates, but we leave this discussion for future.


\section{Substitution operators v.s. ordinary replacements}\label{sec:compare-two-styles-of-substitution}

In \Cref{sec:intro} we have clarified a conceptual difference between our notion of substitutions and the ordinary  replacements $\varphi[\psi/p]$. On the other hand, it is also crucial to acknowledge the close relationship between these two approaches. An illustration for such connections is our recursion axioms in the proof system ${\bf MSL}$: the replacements also have similar valid schemata \cite{tangwen}.  To have a better understanding of the substitution operators, in this part we explore the precise relation between the two styles of substitutions. Along the way, we introduce many notions, and it is instructive to recognize the many applications of those concepts: for instance, they help in understanding the nature of the iterative setting $\MISL$ (cf. \Cref{sec:validities-misl,sec:undecidable}).

Given a formula $\sub{p:=\psi}\varphi$, let us first note that the pivot  $p$  would be `bounded' by the operator, in that, very roughly speaking, the operator $\sub{p:=\psi}$ would update the truth set of $p$ with the states where $\psi$ is true. To make this clear, we define the notions of  {\em free variables} and {\em bound variables} in a formula: 
\begin{definition}[Free variables and bound variables in $\mathcal{L}_\MSL$]\label{def:free-bound-msl}
For any $\varphi\in\mathcal{L}_\MSL$, an occurrence of a propositional variable $p$ is a \emph{bound occurrence} if it is a pivot or appears in the scope of a substitution operator whose pivot is $p$. An occurrence of a propositional variable is a \emph{free occurrence} if it is not bound. 
    A propositional variable is called a \emph{bound variable} if it has a bound occurrence. It is called a \emph{free variable} if it has a free occurrence. We write $bv(\varphi)$ for {\em the set of bound variables} of  $\varphi$, and $fv(\varphi)$ for {\em the set of free variables} of the formula. 
\end{definition}

With the definition, a propositional variable in a formula can be both free and bound: for instance, in formula $\sub{p:=\Box p}p$, the second occurrence of $p$ is free, while the first and the third occurrences of the proposition are bound. Usually, it would be  easier to consider the formulas in which no propositional letter is  both free and bound:

\begin{definition}[Clean $\mathcal{L}_\MSL$-formulas]
    A formula $\varphi\in \mathcal{L}_\MSL$ is \emph{clean} if $fv(\varphi)\cap bv(\varphi)=\emptyset$. 
\end{definition}

Fortunately, for many purposes it is enough to consider only the clean formulas, since we can turn an arbitrary formula to a clean one without changing its truth value. For instance, one can check that $\sub{p:=\Box p}p$ is equivalent to  $\sub{q:=\Box p} q$. More generally, we have the following  renaming rule for bound variables: 
\begin{equation}\label{eqn:MSL_bound_variable_renaming}\tag{Renaming$_{\MSL}$}
\sub{p:=\psi}\varphi\leftrightarrow\sub{q:=\psi}(\varphi[q/p]), ~\textit{where $q$ is a fresh variable}.
\end{equation}
It allows us to replace bound variables  with fresh ones so that we can make a formula clean. Also, $\varphi[q/p]$ used in the rule is the formula obtained by replacing every {\em free} occurrence of the variable $p$ in $\varphi$ with an occurrence of $q$, which is formally given by the following:

\begin{definition}[Replacements of variables in $\mathcal{L}_\MSL$]\label{def:variable_replacement_MSL}
For any $\varphi,\psi\in\mathcal{L}_\MSL$, formula $\varphi[\psi/p]$ is given by the following: 
   \begin{align*}
       q[\psi/p]&:=\begin{cases}q & q\text{ is distinct from }p\\ \psi & q\text{ is }p\end{cases}\\
       (\neg\varphi)[\psi/p]&:=\neg(\varphi[\psi/p])\\
        (\varphi_1\land\varphi_2)[\psi/p]&:=\varphi_1[\psi/p]\land\varphi_2[\psi/p]\\
        (\Diamond\varphi)[\psi/p]&:=\Diamond(\varphi[\psi/p])\\
        (\sub{q:=\chi}\varphi)[\psi/p]&:=\begin{cases}
            \sub{q:=\chi[\psi/p]}(\varphi[\psi/p]) & q\text{ is distinct from }p\\
            \sub{q:=\chi [\psi/p]}\varphi & q\text{ is }p
        \end{cases}
   \end{align*}
\end{definition}


We have the following:

\begin{proposition}\label{prop:clean_formula}
 The  principle  (\ref{eqn:MSL_bound_variable_renaming}) is valid. As a consequence,   every $\MSL$-formula is equivalent to a clean formula. 
\end{proposition}

 \begin{proof}
 We show that (\ref{eqn:MSL_bound_variable_renaming}) is valid. The second part of the proposition directly follows by recursively applying the renaming principle to a formula, as each application of the renaming principle strictly reduces the intersection size of free and bound variables.

 Let $\sub{p:=\psi}\varphi$ be a formula and $q\in{\bf P}$ be a propositional variable not occurring in $\sub{p:=\psi}\varphi$.  Let $\mframe=(W,R)$ be a frame. To achieve the goal, it is enough to show that  for any valuations $V_1,V_2$ s.t. $V_1(p)=V_2(q)$ and for any $r\in\bf{P}$ distinct from $p,q$, $V_1(r)=V_2(r)$, it holds that
    \begin{center}
          $\mmodel_1,w\vDash_{\MSL} \varphi$\quad  iff \quad     $\mmodel_2,w\vDash_{\MSL} \varphi[q/p]$, 
    \end{center}
where $\mmodel_1=(W,R,V_1)$ and  $\mmodel_2=(W,R,V_2)$.

We prove this by induction on $\varphi$. The cases for Boolean connectives $\neg,\land$ hold directly by  induction  hypothesis and Definition \ref{def:variable_replacement_MSL}, and we merely consider other situations.

\vspace{1.5mm}

 (1). Formula $\varphi$ is a propositional variable $r$. There are two different cases.

\vspace{1.5mm}

 (1.1). Formula $r$ is $p$. Then, $\varphi[q/p]$ is $q$. By definition, $w\in V_1(p)$ iff $w\in V_2(q)$.

\vspace{1.5mm}

 (1.2). Formula $r$ is distinct from $p$. Then, $\varphi[q/p]$ is $r$. Again, by definition, $w\in V_1(r)$ iff $w\in V(r)$ iff  $w\in V_2(r)$ (recall that $q$ does not occur in $\sub{p:=\psi}\varphi$, $r$ must be different from $q$).

\vspace{1.5mm}

 (2). Formula $\varphi$ is $\Diamond\chi$. By Definition \ref{def:variable_replacement_MSL}, $\varphi[q/p]$ is $\Diamond(\chi[q/p])$. By induction hypothesis, $\llrr{\chi}^{\mmodel_1}=\llrr{\chi[q/p]}^{\mmodel_2}$. With the truth condition for $\Diamond$, one can easily see that  $\mmodel_1,w\vDash_{\MSL} \varphi$ iff   $\mmodel_2,w\vDash_{\MSL} \varphi[q/p]$, as needed.

\vspace{1.5mm}

 (3). We move to the case that $\varphi$ is $\sub{r:=\theta}\chi$. There are two different situations. 

\vspace{1.5mm}

 (3.1). Formula $r$ is exactly $p$. Then, $\varphi[q/p]$ is $\sub{p:=\theta[q/p]}\chi$. By induction hypothesis, $\llrr{\theta}^{\mmodel_1}=\llrr{\theta[q/p]}^{\mmodel_2}$. From this,  it follows that  $V_1|_{p:=\theta}(p)=V_2|_{p:=\theta[q/p]}(p)$. Now,   $\mmodel_1|_{p:=\theta}$ and $\mmodel_2|_{p:=\theta[q/p]}$  may only differ on the truth set of $q$. Since $q$ does not occur in $\chi$, it holds that
\begin{center}
    $\mmodel_1|_{p:=\theta}, w\vDash_{\MSL}\chi$ \quad   iff  \quad     $\mmodel_2|_{p:=\theta[q/p]}, w\vDash_{\MSL}\chi$, 
\end{center}
as desired.

\vspace{1.5mm}

 (3.2). Formula $r$ is different from $p$. Now, $\varphi[q/p]$ is $\sub{r:=\theta[q/p]}(\chi[q/p])$. As the case above, by induction hypothesis,  $V_1|_{r:=\theta}(r)=V_2|_{r:=\theta[q/p]}(r)$. For these two new valuations,  $V_1|_{r:=\theta}(p)=V_1(p)=V_2(q)=V_2|_{r:=\theta[q/p]}(q)$ and for any other $u\in{\bf P}$, $V_1|_{r:=\theta}(u)=V_2|_{r:=\theta[q/p]}(u)$. So, we can still apply induction hypothesis to the setting involving $\mmodel_1|_{r:=\theta}$ and $\mmodel_2|_{r:=\theta[q/p]}$, which gives us the following:
\begin{center}
    $\mmodel_1|_{r:=\theta},w\vDash_{\MSL}\chi$ \quad   iff \quad   $\mmodel_2|_{r:=\theta[q/p]},w\vDash_{\MSL}\chi[q/p]$,
\end{center}
which completes the proof. 
\end{proof}

\begin{remark}\label{remark:clean-MSL}
    Generally the cleanness of formulas may {\em not} be preserved with respect to its subformulas.\footnote{For instance,  the formula $\sub{p:=q}\sub{p:=p}p$ is clean,  but its subformula $\sub{p:=p}p$ is not.} But we can use  (\ref{eqn:MSL_bound_variable_renaming}) to obtain clean formulas such that their cleanness is preserved under their subformulas, e.g., by making the pivots of all substitution operators distinct from each other.
    With this understanding, in what follows, by clean $\MSL$-formulas we actually mean that they are clean $\MSL$-formulas such that the cleanness is preserved  with respect to subformulas.
\end{remark}

With the help of these notions, we can establish the following principle connecting $\sub{p:=\psi}\varphi$ and $\varphi[\psi/p]$:

\begin{theorem}\label{thm:sub_equal_replacement_MSL}
    $\sub{p:=\psi}\varphi\leftrightarrow \varphi[\psi/p]$ is a validity of $\MSL$,  given that $\sub{p:=\psi}\varphi$ is clean.
\end{theorem}

We will prove this in the remaining part of this section.  Before introducing the details, let us note that when formula $\sub{p:=\psi}\varphi$ is not clean, the equivalence may fail. For instance, given three different propositional letters $p,q,r$,  consider formula $\sub{p:=q;q:=r}(p\land q)$ that is not clean: the formula is equivalent to $q\land r$, while  formula $(\sub{q:=r}(p\land q))[q/p]$, i.e., $\sub{q:=r}(q\land q)$, is equivalent to $r$.\footnote{\label{footnote2}In \cite{public-assignment}, $\PAL$ is extended with single-step substitution operators. The work finds that generally the pattern  $\sub{p:=\psi}\varphi\leftrightarrow \varphi[\psi/p]$ does not hold. Although the result looks the same as our finding, we would like to emphasize that {\em the meaning of $\varphi[\psi/p]$ in \cite{public-assignment} is very different from ours}.  \cite{public-assignment} does not distinguish between bound variables and free variables, and does not have the notion of clean formulas. In that work, $\varphi[\psi/p]$ denotes the formula obtained by replacing {\em all} occurrences of $p$ in $\varphi$ with $\psi$. As explained in \cite{public-assignment}, the failure of   $\sub{p:=\psi}\varphi\leftrightarrow \varphi[\psi/p]$ there is caused by the interaction between substitution operators and public announcement operators, but one can see from the above counterexample that the equivalence does not hold even though public announcement operators are absent in our language. To make the equivalence hold, \cite{public-assignment} imposes a precondition that the truth of $\psi$ should be preserved under restriction to arbitrary submodels, and states that ``a syntactic characterization of the rather semantic notion of `preservation' is not available". But as Theorem \ref{thm:sub_equal_replacement_MSL} shows, we have a syntactic restriction to make the principle to hold, which is general enough (Proposition \ref{prop:clean_formula}) and can also be transferred to the more complicated setting (Theorem \ref{thm:sub_equal_replacement_MISL}).}

To prove Theorem \ref{thm:sub_equal_replacement_MSL}, let us  introduce the following order on clean formulas:

 \begin{definition}[Order on clean $\mathcal{L}_\MSL$-formulas]\label{def:order_MSL}
We define a binary relation $\ll$ on  clean $\mathcal{L}_\MSL$-formulas as follows:
 \begin{center}
 $\varphi\ll\neg\varphi$,  \quad    $\varphi\ll\varphi\land\psi$, \quad   $\varphi\ll\psi\land\varphi$, \quad      $\varphi\ll\Diamond\varphi$, \quad      $\varphi[\psi/p]\ll\sub{p:=\psi}\varphi$.    
 \end{center}
  We use $\gg$ for the converse of $\ll$.
\end{definition}

This order is well-defined on clean formulas, since whenever the formula to the right-hand side of the order symbol $\ll$ is clean, then the formula on the left-hand side must be clean as well. Consequently, any descending $\ll$-chains spanned from a clean formula consist exclusively of clean formulas. In particular, replacements would not break cleanness, due to the following:
\begin{align*}
  fv(\varphi[\psi/p]) &=(fv(\varphi)\setminus\{p\})\cup fv(\psi) = fv(\sub{p:=\psi}\varphi)\\  bv(\varphi[\psi/p]) &= bv(\varphi)\cup bv(\psi)\subseteq bv(\varphi)\cup bv(\psi) \cup \{p\}=bv(\sub{p:=\psi}\varphi)
\end{align*}

 With respect to the order $\ll$, we have the following result:

\begin{proposition}\label{prop:replacement_preseve_ordering_MSL}
    For any clean $\mathcal{L}_{\MSL}$-formulas $\sub{p:=\psi}\varphi$ and $\sub{p:=\psi}\chi$,   if $\varphi\gg\chi$ then $\varphi[\psi/p]\gg \chi[\psi/p]$.
\end{proposition}

\begin{proof}
Since $\sub{p:=\psi}\varphi$ and $\sub{p:=\psi}\chi$ are clean, all formulas $\varphi$, $\psi$, $\chi$, $\varphi[\psi/p]$ and $\chi[\psi/p]$ are clean. Now, assume that  $\varphi\gg\chi$. Let us consider different cases for $\varphi\gg\chi$.

\vspace{1.5mm}

(1). Formula $\varphi$ is $\neg \chi$. Then, $\varphi[\psi/p]$ is $\neg (\chi[\psi/p])$. Hence $\varphi[\psi/p]\gg\chi[\psi/p]$.

\vspace{1.5mm}

(2). Formula  $\varphi$ is $\chi\land\chi'$. Then, $\varphi[\psi/p]$ is $\chi[\psi/p]\land \chi'[\psi/p]$. So it holds that  $\varphi[\psi/p]\gg \chi[\psi/p]$.

\vspace{1.5mm}

 (3). Formula $\varphi$ is $\Diamond\chi$. Immediately, $\varphi[\psi/p]$ is $\Diamond(\chi[\psi/p])$. Hence $\varphi[\psi/p]\gg\chi[\psi/p]$. 

 \vspace{1.5mm}

  (4). Formula $\varphi$ is $\sub{q:=\psi'}\varphi'$. Then, from the definition of $\gg$ and the assumption that $\varphi\gg\chi$, it follows that $\chi$ is  $\varphi'[\psi'/q]$. Now, we  consider two different situations.

 \vspace{1.5mm}

(4.1).  Proposition $q$ is different from $p$.  Then, $\varphi[\psi/p]$ is the formula $\sub{q:=\psi'[\psi/p]}(\varphi'[\psi/p])$. Also, formula $\chi[\psi/p]$ is $(\varphi'[\psi'/q])[\psi/p]$, which  is obtained by first replacing the free occurrences of $q$ in $\varphi'$ with $\psi'$ and then replacing the free occurrences of $p$ in $\varphi'[\psi'/q]$ with $\psi$. The free occurrences of $p$ in $\varphi'[\psi'/q]$ are those in $\varphi'$ and $\psi'$. The replacement process of $(\varphi'[\psi'/q])[\psi/p]$ is equivalent to first replacing the free occurrences of $p$ in $\varphi'$ with $\psi$, and then replacing the free occurrences of $q$ in $\varphi'[\psi/p]$ with $\psi'[\psi/p]$. Since $q\not\in fv(\psi)$, the free occurrences of $q$ in $\varphi'$ are the same as the free occurrences of $q$ in $\varphi'[\psi/p]$. Hence, $\chi[\psi/p]$ is $(\varphi'[\psi/p])[(\psi'[\psi/p])/q]$ and $\varphi[\psi/p]\gg\chi[\psi/p]$.

 \vspace{1.5mm}

(4.2).  Proposition $q$ is $p$. Then, $\varphi[\psi/p]$ is $\sub{p:=\psi'[\psi/p]}\varphi'$, and $\chi[\psi/p]$ is $(\varphi'[\psi'/p])[\psi/p]$. Following the same argument in (4.1), we can see that $\chi[\psi/p]$ is the same as $\varphi'[(\psi'[\psi/p])/p]$. Therefore, $\varphi[\psi/p]\gg\chi[\psi/p]$. 
 \vspace{1.5mm}
    
To sum up, $\varphi[\psi/p]\gg\chi[\psi/p]$ always holds. 
\end{proof}

 So, given a finite chain, e.g.,
\begin{center}
    $\varphi\gg\varphi_1\gg\varphi_2\gg\dots\gg q_{\varphi}$,  
\end{center}
the above result gives us the following for free:
\begin{center}
  $\varphi[\psi/p]\gg\varphi_1[\psi/p]\gg\varphi_2[\psi/p]\gg\dots\gg q_{\varphi}[\psi/p]$  
\end{center}
If $q_{\varphi}$ is not $p$, then the end of the above chain is exactly $q_{\varphi}$, which cannot have any $\gg$-successor. But if  $q_{\varphi}$ is  $p$, then the end of the above chain is $\psi$. When the latter is the case, given a finite $\gg$-chain $\varepsilon$ staring from $\psi$ and ending at a propositional variable $r$, by adding $\varepsilon$ to the chain above, we can obtain a finite  $\gg$-chain starting from $\varphi[\psi/p]$ and ending at $r$, which cannot have further $\gg$-successor. With this, we can establish the following: 
\begin{lemma}\label{lem:wellfounded_ordering_MSL}
The order  $\ll$ defined in \Cref{def:order_MSL} is well-founded.
\end{lemma}
 
\begin{proof} 
We need to prove that any clean $\mathcal{L}_{\MSL}$-formula spans  finite descending $\ll$-chains. It goes by induction on formulas. The base cases for propositional variables $p$ are straightforward. The first four clauses in \Cref{def:order_MSL} only allow a formula to descend to a strictly shorter formula. We now proceed to consider for $\sub{p:=\psi}\varphi$.

\vspace{1.5mm}

Let $\sub{p:=\psi}\varphi$ be a clean $\mathcal{L}_{\MSL}$-formula. Then, both $\psi$ and $\varphi$ are clean (\Cref{remark:clean-MSL}). By induction hypothesis, $\varphi$ and $\psi$ only span finite descending $\ll$-chains. Assume that the following is a finite descending chain starting from $\varphi$:
\[
\varphi\gg\varphi_1\gg\varphi_2\gg\dots\gg q
\]
where $q\in{\bf P}$. Then, by \Cref{prop:replacement_preseve_ordering_MSL},  $\varphi[\psi/p]$ spans a descending $\ll$-chain with the following as an initial segment:
\[
\varphi[\psi/p]\gg\varphi_1[\psi/p]\gg\varphi_2[\psi/p]\gg\dots\gg q[\psi/p].
\]
With the previous analysis, we know that $\varphi[\psi/p]$ can only span finite $\gg$-chains that end at propositional variables and so have no further $\gg$-successor. The proof is completed.\qedhere
    
\end{proof}

We then define {\em the depth of a clean formula $\dep(\varphi)$} to be the length of the longest descending $\ll$-chains spanned by $\varphi$. The details are as follows:

\begin{definition}[Depth of a clean $\mathcal{L}_{\MSL}$-formula]\label{def:formula_depth_MSL}
For a clean $\varphi\in\mathcal{L}_{\MSL}$, its depth $\dep(\varphi)$  is given by the following:
\begin{center}
    $\dep(p):=0$\quad  $\dep(\neg\varphi):=\dep(\varphi)+1$\quad    
        $\dep(\varphi_1\land\varphi_2):=\max\{\dep(\varphi_1),\dep(\varphi_2)\}+1$\vspace{1mm}\\
$\dep(\Diamond\varphi):=\dep(\varphi)+1$\quad    $\dep(\sub{p:=\psi}\varphi):=\dep(\varphi[\psi/p])+1$ 
\end{center}
\end{definition}

From the  proof for \Cref{lem:wellfounded_ordering_MSL}, we can see that $\dep(\sub{p:=\psi}\varphi)$ is upper-bounded by $\dep(\psi)+\dep(\varphi)+1$. 
For any clean $\varphi\in\mathcal{L}_{\MSL}$, its depth is always finite.  
Now we are ready to prove \Cref{thm:sub_equal_replacement_MSL}. 

\begin{proof}[Proof of \Cref{thm:sub_equal_replacement_MSL}]
It goes by induction on $\dep(\sub{p:=\psi}\varphi)$. The base case is that $\dep(\sub{p:=\psi}\varphi)=1$, which means that $\varphi[\psi/p]$ is a propositional variable. Clearly, $\sub{p:=\psi}\varphi\leftrightarrow\varphi[\psi/p]$. Now, suppose the statement is true for all $\sub{p:=\psi}\varphi$ such that $\dep(\sub{p:=\psi}\varphi)<n$, where $n>1$. We are going to show that it is also true for $\sub{p:=\psi}\varphi$ such that $\dep(\sub{p:=\psi}\varphi)=n$. To do so, let us consider the form of $\varphi$. 

    \vspace{1.5mm}

    (1). Formula $\varphi$ is a propositional variable. When $\varphi$ is $p$,  $\sub{p:=\psi}\varphi$ is equivalent to $\psi$ (by the axiom $\mathsf{(R1)}$), and $\varphi[\psi/p]$ is exactly $\psi$. When $\varphi$ is different from $p$,  $\sub{p:=\psi}\varphi$ is equivalent to $\varphi$ (by the axiom $\mathsf{(R2)}$), and $\varphi[\psi/p]$ is exactly $\varphi$. Therefore $\sub{p:=\psi}\varphi\leftrightarrow\varphi[\psi/p]$. 

      \vspace{1.5mm}

     (2). Formula $\varphi$ is  $\neg\chi$. By the axiom $\mathsf{(R3)}$,  $\sub{p:=\psi}\neg\chi\leftrightarrow \neg\sub{p:=\psi}\chi$. Then, 
     \begin{center}
$\dep(\sub{p:=\psi}\chi)=\dep(\chi[\psi/p])+1 =\dep(\neg\chi[\psi/p]) =\dep(\sub{p:=\psi}\varphi)-1 $   
     \end{center}
     By induction hypothesis, it holds that $\sub{p:=\psi}\chi\leftrightarrow \chi[\psi/p]$. Now, we have $\sub{p:=\psi}\varphi\leftrightarrow  \varphi[\psi/p]$.

      \vspace{1.5mm}

    (3). Formula $\varphi$ is  $\chi_1\land\chi_2$. By the axiom $\mathsf{(R4)}$, $\sub{p:=\psi}(\chi_1\land\chi_2)\leftrightarrow \sub{p:=\psi}\chi_1\land\sub{p:=\psi}\chi_2$. Also, for each $i\in\{1,2\}$, it holds that 
\begin{center}
   $ \dep(\sub{p:=\psi}\chi_i)=\dep(\chi_i[\psi/p])+1< \dep(\varphi[\psi/p])+1=\dep(\sub{p:=\psi}\varphi)$
\end{center}
    Then, by induction hypothesis, it holds immediately that  $\sub{p:=\psi}\chi_i \leftrightarrow \chi_i[\psi/p]$. Also, since $ \chi_1[\psi/p]\land\chi_2[\psi/p]$ is exactly $\varphi[\psi/p]$, we can obtain $\sub{p:=\psi}\varphi \leftrightarrow \varphi[\psi/p]$.

  \vspace{1.5mm}

(4). Formula $\varphi$ is  $\Diamond\chi$. It follows from  $\mathsf{(R5)}$ that  $\sub{p:=\psi}\Diamond\chi\leftrightarrow \Diamond\sub{p:=\psi}\chi$. Moreover,  
\begin{center}
$\dep(\sub{p:=\psi}\chi)=\dep(\Diamond\chi[\psi/p])=\dep(\sub{p:=\psi}\varphi)-1 $
\end{center}
By induction hypothesis, it holds that $\sub{p:=\psi}\chi\leftrightarrow \chi[\psi/p]$. Now, we have $ \sub{p:=\psi}\varphi\leftrightarrow \varphi[\psi/p]$.
 
 \vspace{1.5mm}

(5). Formula $\varphi$ is  $\sub{q:=\chi_1}\chi_2$. Since $\sub{p:=\psi}\varphi$ is clean, we can  assume that $\varphi$ is also clean (\Cref{remark:clean-MSL}). Now, it is enough to prove the following equivalences:
    \begin{align*}
        \sub{p:=\psi}\sub{q:=\chi_1}\chi_2 &\leftrightarrow \sub{p:=\psi}(\chi_2[\chi_1/q])\\
        &\leftrightarrow (\chi_2[\chi_1/q])[\psi/p]\\
        &\leftrightarrow (\sub{q:=\chi_1}\chi_2)[\psi/p]
    \end{align*}
One can verify that $\dep(\sub{q:=\chi_1}\chi_2)<\dep(\sub{p:=\psi}\sub{q:=\chi_1}\chi_2)=n$. Therefore, $\sub{q:=\chi_1}\chi_2\leftrightarrow \chi_2[\chi_1/q]$. Hence, the first and the last equivalences hold. To prove the second equivalence, it suffices to show that 
    \begin{center}
     $\dep(\sub{p:=\psi}(\chi_2[\chi_1/q]))<n$.   
    \end{center} 
    By Definition \ref{def:order_MSL}, it holds that $\varphi\gg \chi_2[\chi_1/q]$. Then, it follows from  \Cref{prop:replacement_preseve_ordering_MSL} that $\varphi[\psi/p]\gg (\chi_2[\chi_1/q])[\psi/p]$. Hence, we have the following:
\begin{center}
 $n = \dep(\sub{p:=\psi}\varphi)
        = \dep(\varphi[\psi/p])+1
        = \dep((\chi_2[\chi_1/q])[\psi/p])+2
        = \dep(\sub{p:=\psi}(\chi_2[\chi_1/q]))+1$
\end{center}
So, $\dep(\sub{p:=\psi}(\chi_2[\chi_1/q]))<n$,
    and the second equivalence holds. The proof is completed.\qedhere
\end{proof}

Now we have shown the main result of this section, which clarifies the  relationship between $\sub{p:=\psi}\varphi$ and  $\varphi[\psi/p]$. It is important to emphasize that the said comparison has the language $\mathcal{L}_{\MSL}$ as a common basis, in that the replacement operation is defined for $\mathcal{L}_{\MSL}$ (\Cref{def:variable_replacement_MSL}). It would also be meaningful to compare $\mathcal{L}_{\MSL}$ with the language extending $\mathcal{L}_{\ML}$ with the replacements (but not containing substitution operators) \cite{tangwen}.

To end this section, let us note that \Cref{thm:sub_equal_replacement_MSL} holds broader significance beyond its immediate context. For instance, it offers an alternative approach to prove \Cref{prop:MSL=ML-semantics} by induction on the depth of formulas, which we will leave as an exercise to the reader. Also, many of the notions that were used to prove the result are crucial, and we will revisit and generalize them in later sections.  We are now ready to move to the more powerful setting containing the iterative generalizations of the single-step substitutions.


\section{A modal logic for iterative substitutions}\label{misl}

This part explores the augmented {\em Modal Iterative Substitution Logic} ($\MISL$) with iterative substitution operators. The language of the logic is proposed in  \cite{Johan-fixedpoint2021}, and we will proceed to explore its formal properties. As we will see, the iterative computations make the logic much more complicated than $\MSL$. This section  will present various validities to show how the iterative operators work and formally show that the backward induction introduced in Scenario 2 of \Cref{sec:intro} can be characterized by $\MISL$. In later sections, we will look into the properties of $\MISL$, involving expressiveness and computational behavior, and study its connections with some relevant frameworks. 

The  {\em language $\mathcal{L}_{\MISL}$ for $\MISL$} extends $\mathcal{L}_{\MSL}$ with iterative substitution operators
\[
\sub{(p:=\psi)^*}\varphi
\]
expressing that {\em there is some $n\in\mathbb{N}$ such that after we iteratively substitute the truth set of $p$ with that of $\psi$ for $n$ times, $\varphi$ is the case}. Precisely, the truth condition is given by the following:
\begin{center}
  $\mmodel,s\vDash_{\MISL} \sub{(p:=\psi)^*}\varphi \;\Leftrightarrow\; \mmodel|_{(p:=\psi)^n}, s\vDash_{\MISL} \varphi$  {\it for some}  $n\in\mathbb{N}$,   
\end{center}
where $\mmodel|_{(p:=\psi)^n} = (W,R,V|_{(p:=\psi)^n})$ disagrees with  $\mmodel$ only on the valuation of $p$, and $V|_{(p:=\psi)^n}(p)$ is defined in the following inductive manner:
\begin{center}
 $V|_{(p:=\psi)^0}(p) := V(p)$ \vspace{1mm} \\
$V|_{(p:=\psi)^{n+1}}(p):= \{s\in W:\mmodel|_{(p:=\psi)^{n}},s\vDash_{\MISL} \psi\}$
\end{center}
So, for any natural number $n\in\mathbb{N}$, $\mmodel|_{(p:=\psi)^n}$ is obtained by $n$ model transformations starting from $\mmodel$. 

We use $[(p:=\psi)^*]$ as the dual of $\sub{(p:=\psi)^*}$, and so its truth condition is as follows:
\begin{center}
  $\mmodel,s\vDash_{\MISL} [(p:=\psi)^*]\varphi\; \Leftrightarrow\;   \mmodel|_{(p:=\psi)^n}, s\vDash_{\MISL} \varphi$ {\it for all}  $n\in\mathbb{N}$.
\end{center}
Unlike the case of $\sub{p:=\psi}\varphi$, iterative substitution operators are not self-dual, i.e.,   $[(p:=\psi)^*]\varphi\leftrightarrow \sub{(p:=\psi)^*}\varphi$ fails in general.

To close the part, let us introduce some syntactic notions.  To distinguish from iterative substitution operators, those $\sub{p:=\psi}$ and $[p:=\psi]$ are termed as {\em simple substitution operators}. Moreover, both the iterative and the simple ones  are called {\em substitution operators}. The notion of {\em subformulas for $\MISL$} extends that for $\MSL$ with the following: 
\[
\mathsf{Sub}(\sub{(p:=\psi)^*}\varphi)=\mathsf{Sub}(\psi)\cup \mathsf{Sub}(\varphi)\cup \{\sub{(p:=\psi)^*}\varphi\}.
\]

 For a formula $\sub{(p:=\psi)^*}\varphi$,  $p$ is the \emph{pivot} of the iterative substitution operator, but in contrast to the case of simple substitution operators, we define the \emph{scope} of operator $\sub{(p:=\psi)^*}$ as the whole formula $\sub{(p:=\psi)^*}\varphi$. Finally, let us introduce the following   notion of `depth of iteration' $\mathsf{DOI}(\varphi)$:
\begin{center}
 $\mathsf{DOI}(p)=0$,\quad    $\mathsf{DOI}(\neg\varphi)=\mathsf{DOI}(\varphi)$, \quad $\mathsf{DOI}(\varphi\land\psi)=\max\{\mathsf{DOI}(\varphi),\mathsf{DOI}(\psi)\}$, \quad $\mathsf{DOI}(\Diamond\varphi)=\mathsf{DOI}(\varphi)$, \vspace{1mm} \\
 $\mathsf{DOI}(\sub{p:=\psi}\varphi)=\max\{\mathsf{DOI}(\varphi),\mathsf{DOI}(\psi)\}$,\quad  $\mathsf{DOI}(\sub{(p:=\psi)^*}\varphi)=\max\{\mathsf{DOI}(\varphi),\mathsf{DOI}(\psi)+1\}$, 
\end{center} 
This corresponds to the number of layers of iterative operators involved in a formula. In the next part, we explore more validities of the logic.

\subsection{Validities of \texorpdfstring{$\MISL$}{}}\label{sec:validities-misl}

Since  $\MISL$ is a direct extension of  $\MSL$, all validities of the latter still hold in the new setting, including the recursion axioms for the simple substitution operators. Also, as we shall see, the generalizations of  the bound variable renaming rule  (\ref{eqn:MSL_bound_variable_renaming}) and  \Cref{thm:sub_equal_replacement_MSL} for $\MISL$  are also true. Moreover, the following schemata are valid: 
\begin{proposition}\label{prop:msl_validities_in_misl}
The following formulas involving iterative substitution operators are valid:
    \begin{equation*}
    \begin{aligned}
    \sub{(p:=\psi)^*}q&\leftrightarrow q,\quad\text{ if }q\text{ is not }p.\\
 \sub{(p:=\psi)^*}p& \leftrightarrow p\lor \sub{p:=\psi;(p:=\psi)^*}p\\
 [(p:=\psi)^*](\varphi\to [(p:=\psi)]\varphi)& \leftrightarrow (\varphi\to [(p:=\psi)^*]\varphi)\\
       \sub{(p:=\psi)^*}\neg\varphi&\leftrightarrow \neg[(p:=\psi)^*]\varphi \\
      \sub{(p:=\psi)^*}(\varphi_1\land\varphi_2) & \rightarrow \sub{(p:=\psi)^*}\varphi_1\land\sub{(p:=\psi)^*}\varphi_2\\
    \sub{(p:=\psi)^*}\Diamond\varphi &\leftrightarrow \Diamond\sub{(p:=\psi)^*}\varphi
        \end{aligned}
    \end{equation*}
\end{proposition}

 Note that the second last formula is only one-directional. Some patterns above are similar to that of the recursion axioms in the proof system {\bf MSL}, whereas the others have typical forms for iterative operations like those in propositional dynamic logic $\PDL$ \cite{pdl-bood}. Let us now take a closer look at the iterative substitution operators. 

As stated, an iterative substitution operator $\sub{(p:=\psi)^*}$ leads to a  sequence of transformations of the truth set of  $p$, starting from the original $V(p)$ given by a model. Generally speaking, the proposition $p$ in $\sub{(p:=\psi)^*}$ functions as both the starting point of the transformation sequence  and the pivot of each single-step transformation  $\sub{p:=\psi}$. Intuitively, when $\psi$ contains $p$, the occurrences of $p$ in $\psi$ in the first $\sub{p:=\psi}$ of the sequence are free, while those in the later stages are bound. To make this clear, let us first introduce the following: 
\begin{definition}[Normal $\mathcal{L}_{\MISL}$-formulas]\label{def:normal_formula}
An $\MISL$-formula is  \emph{normal} if every iterative substitution $(p:=\psi)^*$ follows immediately after a single-step initialization $p:=\psi_0$.
\end{definition}

For instance, $\sub{p:=\bot;(p:=\Box p)^*}p$ is a normal formula, while $\sub{p:=\bot;q:=\top;(p:=\Box p)^*}p$ is not. Given a normal formula, its subformulas might not be normal.\footnote{For instance, formula $\sub{p:=\neg p;(p:=\Box p)^*}p$ is normal, but $\sub{(p:=\Box p)^*}p$ is not.}  In a normal formula, each iterative substitution is a part of $p:=\psi_0;(p:=\psi)^*$. Now, we can generalize the notions of {\em free occurrences}, {\em free variables}, {\em bound occurrences} and  {\em bound variables} for $\MSL$ (Definition \ref{def:free-bound-msl}) to {\em normal formulas}.  For instance, in formula $\sub{p:=p\lor q; (p:=\Box p)^*} p$, only the second occurrence of $p$ is free, while its other occurrences are bound.  Moreover, an $\MISL$-formula is {\em clean} if it is a normal formula not containing any propositional letter that is both free and bound.

Just as the case in $\MSL$, we can  turn a formula of $\mathcal{L}_{\MISL}$ to a clean one. For instance, $\sub{(p:=\Box p)^*}p$ is equivalent to $\sub{(q:= p);(q:=\Box q)^*}q$. Given a formula of $\mathcal{L}_{\MISL}$, to obtain a desired clean formula, we can first apply the following normalization rule 
\begin{equation}\label{eqn:MISL_normalization}\tag{Normalization}
    \sub{(p:=\psi)^*}\varphi\leftrightarrow \sub{p:=p; (p:=\psi)^*}\varphi
\end{equation}
that transforms any $\MISL$-formula to a normal one, and then apply the two bound variable renaming rules below to transform every normal $\MISL$-formula to a clean formula: 
\begin{equation}\label{eqn:MISL_bound_variable_renaming}
\tag{Renaming$_{\MISL}$}
    \begin{split}
        \sub{p:=\psi}\varphi&\leftrightarrow \sub{q:=\psi}\varphi[q/p], \\
        \sub{p:=\psi_0;(p:=\psi_1)^*}\varphi&\leftrightarrow \sub{q:=\psi_0;(q:=\psi_1[q/p])^*}\varphi[q/p]
    \end{split}
\end{equation}
where the main connective of $\varphi$ in $ \sub{p:=\psi}\varphi$ in the first equivalence is not an iterative substitution operator,\footnote{Precisely, since $\sub{p:=\psi}\varphi$ is a normal formula, the main connective of $\varphi$ is either $\neg, \land, \Diamond$, a simple substitution operator, or an iterative substitution operator with pivot $p$. But the last case is dealt with by the second renaming rule, and this is why we require  that the main connective of $\varphi$ in $\sub{p:=\psi}\varphi$ is {\em not} an iterative operator.} the propositional variable $q$ is a fresh variable in both the principles, and formula $\varphi[q/p]$ is obtained by replacing every free occurrence of the variable $p$ in $\varphi$ with an occurrence of $q$, which is defined in the following manner: 
\begin{definition}[Replacement of variables in a normal formula]\label{def:variable_replacement_MISL}
    For any normal formulas $\varphi,\psi\in\mathcal{L}_\MISL$,  $\varphi[\psi/p]$ is defined by extending the clauses for propositional variables, $\neg,\land,\Diamond$ in   \Cref{def:variable_replacement_MSL} with the following: 
            \begin{center}
              $(\sub{q:=\chi}\varphi)[\psi/p]:=\begin{cases}
            \sub{q:=\chi[\psi/p]}(\varphi[\psi/p]) & q\text{ is different from }p\\
            \sub{q:=\chi [\psi/p]}\varphi & q\text{ is }p
        \end{cases}$  
        
        \vspace{2mm}

$(\sub{q:=\chi_0;(q:=\chi_1)^*}\varphi)[\psi/p]:=\begin{cases}
            \sub{q:=\chi_0[\psi/p];(q:=\chi_1[\psi/p])^*}(\varphi[\psi/p]) & q\text{ is different from }p\\
            \sub{q:=\chi_0[\psi/p];(q:=\chi_1)^*}\varphi & q\text{ is }p
        \end{cases}$
        \end{center}
where  the main connective of $\varphi$ in the first clause is not an iterative substitution operator. 
\end{definition}

With the notions, we now introduce the following:

\begin{proposition}\label{prop:MISL_clean_equivalence}
The two principles of (\ref{eqn:MISL_bound_variable_renaming}) are valid. As a direct consequence, every normal formula in $\MISL$ is equivalent to a clean formula. 
\end{proposition}

\begin{proof}
 We show that the two principles of (\ref{eqn:MISL_bound_variable_renaming}) are valid. 
The proof for  the validity of the first principle, $\sub{p:=\psi}\varphi\leftrightarrow \sub{q:=\psi}\varphi[q/p]$,  is by induction on $\varphi$, and  
the argument is the same as that in the proof for Proposition \ref{prop:clean_formula}. 
We now move to showing the validity of the second principle.  
 
 Let $\sub{p:=\psi_0;(p:=\psi_1)^*}\psi_2$  be a formula and $q\in{\bf P}$ be a propositional variable not occurring in it. Let $\mframe=(W,R)$ be a frame. We are going to first claim and prove the following statement:

\vspace{1.5mm}

\noindent \textbf{Claim: }\textit{For any valuations $V_1,V_2$ such that $V_1(p)=V_2(q)$ and $V_1(r)=V_2(r)$ for any $r\in\bf{P}$ distinct from $p$ and $q$, let $\mmodel_1=(W,R,V_1)$, $\mmodel_2=(W,R,V_2)$ be two arbitrary models with the same frame and valuations $V_1, V_2$ respectively. Then for any normal formula $\varphi\in\mathcal{L}_{\MISL}$, $$\mmodel_1,w\vDash_{\MISL}  \varphi \text{ iff }   \mmodel_2,w\vDash_{\MISL}  \varphi[q/p].$$}

 We prove this statement by an outer induction on $\sf DOI(\varphi)$ and an inner induction on the length of $\varphi$, where the length of a formula means the number of symbols occurring in it. 

\vspace{1.5mm}

 (1). We begin with the basic case that $\mathsf{DOI}(\varphi)=0$, i.e., $\varphi$ does not contain any iterative substitution operator. Again, this can be proved with the same argument as that in the proof for Proposition \ref{prop:clean_formula}.

\vspace{1.5mm}

 (2). For the induction step, let $1\le m\in\mathbb{N}$ and suppose that the claim is true for (i) every formula $\psi$ with $\mathsf{DOI(\psi)}<m$   and (ii) every formula $\chi$ such that  $\mathsf{DOI(\chi)}=m$ and the length is smaller than $k$. We proceed to consider the case that $\mathsf{DOI}(\varphi)=m$ and the length of $\varphi$ is $k$. For this, we consider the form of $\varphi$. 

 Since $m>1$, $\varphi$ is not a propositional variable. The cases for Boolean connectives $\neg,\land$ and modality $\Diamond$ hold directly by  induction hypothesis. Now consider the case where $\varphi$ is $\sub{r:=\psi}\chi$. 
The reasoning for the case where the main connective of $\chi$ is not an iterative substitution operator, is similar to that in the proof for Proposition \ref{prop:clean_formula}. We need only consider the case where the main connective of $\chi$ is an iterative substitution operator, and since $\varphi$ is a normal formula, $\varphi$ must be $\sub{r:=\psi_0;(r:=\psi_1)^*}\chi$. We then split into two different cases according to whether $r$ is $p$.

\vspace{1.5mm}

 (2.1). Propositional variable $r$ is $p$. By definition, $\varphi[q/p]$ is $\sub{p:=\psi_0[q/p];(p:=\psi_1)^*}\chi$. Note that either $(i)$ $\mathsf{DOI}(\psi_0)=m$ and  the length of $\psi_0$ is smaller than $k$  or $(ii)$ $\mathsf{DOI}(\psi_0)<m$. So,  by induction hypothesis, it holds that
\begin{center}
$\mmodel_1,w\vDash_{\MISL}\psi_0$\quad iff \quad $\mmodel_2,w\vDash_{\MISL}\psi_0[q/p]$, 
\end{center}
which indicates $\llrr{\psi_0}^{\mmodel_1}=\llrr{\psi_0[q/p]}^{\mmodel_2}$. Now, we just need to show that 
\begin{center}
$\mmodel_1|_{p:=\psi_0} \vDash_{\MISL} \sub{(p:=\psi_1)^*}\chi$ \quad  iff \quad $\mmodel_2|_{p:=\psi_0[q/p]} \vDash_{\MISL} \sub{(p:=\psi_1)^*}\chi$.    
\end{center}
Since $\mmodel_1|_{p:=\psi_0}$ and $\mmodel_2|_{p:=\psi_0[q/p]}$ may only differ on the truth set of $q$ that does not occur in $\sub{(p:=\psi_1)^*}\chi$, the equivalence above holds.
\vspace{1.5mm}

 (2.2). Propositional letter $r$ is different from $p$. By definition, $\varphi[q/p]$ is $\sub{r:=\psi_0[q/p];(r:=\psi_1[q/p])^*}(\chi[q/p])$. As the case above, we know that $\llrr{\psi_0}^{\mmodel_1}=\llrr{\psi_0[q/p]}^{\mmodel_2}$. Now, it holds that 
 \begin{center}
    $\mmodel_1|_{r:=\psi_0}\vDash_{\MISL} \sub{(r:=\psi_1)^*}\chi$\quad  iff\quad  $\mmodel_1|_{r:=\psi_0}\vDash_{\MISL} \sub{(r:=\psi_1)^n}\chi$ for some $n\in\mathbb{N}$.  
 \end{center}
 When $n$ is $0$, by induction hypothesis it holds immediately that   $\mmodel_1|_{r:=\psi_0}\vDash_{\MISL}  \chi$ iff $\mmodel_2|_{r:=\psi_0[q/p]}\vDash_{\MISL}  \chi[q/p]$. We now move to the case that $n>0$. Since $\mathsf{DOI}(\psi_1)<m$, for any $0\le j\le n-1$,  $\mmodel_1|_{r:=\psi_0;(r:=\psi_1)^j},w\vDash_{\MISL}  \psi_1$ iff  $\mmodel_2|_{r:=\psi_0[q/p];(r:=\psi_1[q/p])^j},w\vDash_{\MISL}  \psi_1[q/p]$. This indicates that $V_1|_{r:=\psi_0;(r:=\psi_1)^n}(r)=V_2|_{r:=\psi_0[q/p];(r:=\psi_1[q/p])^n}(r)$. By induction hypothesis, it holds that 
 \begin{center}
$\mmodel_1|_{r:=\psi_0;(r:=\psi_1)^n},w\vDash_{\MISL}  \chi$\quad  iff \quad  $\mmodel_2|_{r:=\psi_0[q/p];(r:=\psi_1[q/p])^n},w\vDash_{\MISL}  \chi[q/p]$.
\end{center}
This completes the proof of the claim.

 Now given the claim, one can prove the following statement: 
\[
\forall n\in \mathbb{N},\, \sub{p:=\psi_0;(p:=\psi_1)^n}\varphi\leftrightarrow \sub{q:=\psi_0;(q:=\psi_1[q/p])^n}\varphi[q/p], 
\]
It is enough to show that for all $n\in \mathbb{N}$, $V_1^{(n)}:= V|_{p:=\psi_0;(p:=\psi_1)^n}$ and $V_2^{(n)}:=V|_{q:=\psi_0;(q:=\psi_1[q/p])^n}$ satisfy the condition of the claim, then the statement follows since $\varphi$ is a normal formula. To show that $V_1^{(n)}$ and $V_2^{(n)}$ agree on $r\not\in\{p,q\}$ and $V_1^{(n)}(p)=V_2^{(n)}(q)$, one can prove by induction on $n$ - note that the base case follows directly and each induction step, can be proved by combining the claim with the induction hypothesis. Summing up, the above implies that $\sub{p:=\psi_0;(p:=\psi_1)^*}\varphi\leftrightarrow \sub{q:=\psi_0;(q:=\psi_1[q/p])^*}\psi_2[q/p]$, which completes the proof.
\end{proof}

\begin{remark}\label{remark:clean-MISL}
    Similar to the case in \Cref{remark:clean-MSL}, when  saying that an $\MISL$-formula  $\varphi$ is clean, we also assume that: 
    \begin{itemize}
        \item When $\varphi$ is  $\sub{p:=\psi}\chi$  and the main connective of $\chi$ is not an iterative substitution operator, formulas $\chi$ and $\psi$ are also clean.
        \item When  $\varphi$ is  $\sub{p:=\psi_1;(p:=\psi_2)^*}\chi$, all formulas $\psi_1$, $\psi_2$ and $\chi$ are clean.
    \end{itemize}
    This can also be achieved by repeated applications of (\ref{eqn:MISL_bound_variable_renaming}).
\end{remark}

Now,  we can show that under suitable condition, the principle $\sub{p:=\psi}\varphi\leftrightarrow \varphi[\psi/p]$ stated in \Cref{thm:sub_equal_replacement_MSL} also holds in the enlarged framework $\MISL$:

\begin{theorem}\label{thm:sub_equal_replacement_MISL}    $\sub{p:=\psi}\varphi\leftrightarrow \varphi[\psi/p]$ is a validity, if both $\sub{p:=\psi}\varphi$ and $\varphi$ are clean formulas.
\end{theorem}

We will prove it in the remainder of this part, following a similar method to the proof of \Cref{thm:sub_equal_replacement_MSL} in \Cref{sec:compare-two-styles-of-substitution}. Let us first  define the following order  on clean $\MISL$-formulas, which is an extension of the one given by \Cref{def:order_MSL}: 
\begin{definition}[Order  on clean formulas of $\MISL$]\label{def:order_MISL}
We define a binary relation $\ll$ on  clean $\mathcal{L}_\MISL$-formulas: 
\begin{center}
  $\varphi\ll\neg\varphi$, \quad   $\varphi\ll\varphi\land\psi$, \quad     $\varphi\ll\psi\land\varphi$,\quad $\varphi\ll\Diamond\varphi$, \vspace{2mm} \\
     $\varphi[\psi/p]\ll\sub{p:=\psi}\varphi$, \quad 
$\sub{p:=\psi_0; (p:=\psi_1)^n}\varphi\ll\sub{p:=\psi_0; (p:=\psi_1)^*}\varphi$,   
\end{center}
where $n\in\mathbb{N}$, the main connective of the $\varphi$ in the clause for $\varphi[\psi/p]$ is not an iterative substitution operator, and the replacement $\varphi[\psi/p]$ is defined in \Cref{def:variable_replacement_MISL}. Again, we use $\gg$ for the converse of $\ll$.
\end{definition} 

With the condition used in the clause for $\sub{p:=\psi}\varphi$, one can check that  whenever the formula to the right-hand-side of $\ll$ is clean, the formula to the left-hand-side of $\ll$ is also clean. 
Now, we can prove the following result that is similar to \Cref{prop:replacement_preseve_ordering_MSL}:
\begin{proposition}\label{prop:replacement_preserve_ordering_MISL}
Let $\sub{p:=\psi;(p:=\psi)^*}\varphi$ and $\sub{p:=\psi;(p:=\psi)^*}\chi$ be clean. If $\varphi\gg\chi$, then $\varphi[\psi/p]\gg\chi[\psi/p]$. 
\end{proposition}

\begin{proof}
Since $\sub{p:=\psi;(p:=\psi)^*}\varphi$ and $\sub{p:=\psi;(p:=\psi)^*}\chi$ are clean, based on (\ref{eqn:MISL_bound_variable_renaming}) we can assume that all formulas  $\varphi$, $\psi$, $\chi$, $\varphi[\psi/p]$ and $\chi[\psi/p]$ are clean (\Cref{remark:clean-MISL}). Now, assume that  $\varphi\gg\chi$. We only need to consider the case that $\varphi$ and $\chi$ are of the forms  $\sub{q:=\psi_0;(q:=\psi_1)^*}\theta$ and $\sub{q:=\psi_0;(q:=\psi_1)^n}\theta$ respectively, as other cases can be proved with the same argument as that in the proof for \Cref{prop:replacement_preseve_ordering_MSL}. We now consider two different situations.

    \vspace{1.5mm}

    (1). Proposition $q$ is different from $p$. Then, $\varphi[\psi/p]$ is $\sub{q:=\psi_0[\psi/p];(q:=\psi_1[\psi/p])^*}(\theta[\psi/p])$, and $\chi[\psi/p]$ is $\sub{q:=\psi_0[\psi/p];(q:=\psi_1[\psi/p])^n}(\theta[\psi/p])$. By the definition of $\gg$, $\varphi[\psi/p]\gg \chi[\psi/p]$. 

    \vspace{1.5mm}

    (2). Proposition $q$ is $p$. Then, formula $\varphi[\psi/p]$ is $\sub{q:=\psi_0[\psi/p];(q:=\psi_1)^*}(\theta[\psi/p])$, and formula $\chi[\psi/p]$ is $\sub{q:=\psi_0[\psi/p];(q:=\psi_1)^n}(\theta[\psi/p])$. By the definition of $\gg$, $\varphi[\psi/p]\gg \chi[\psi/p]$. 

    \vspace{1.5mm}
    
To sum up, $\varphi[\psi/p]\gg\chi[\psi/p]$ always holds. 
\end{proof}

\begin{lemma}\label{lem:wellfounded_ordering_MISL}
    The order $\ll$ defined in \Cref{def:order_MISL} is well-founded. 
\end{lemma}

\begin{proof}
We prove that any clean $\varphi\in\mathcal{L}_{\MISL}$ only spans finite descending $\ll$-chains,  by induction, first on  $\mathsf{DOI}(\varphi)=m$ and then on its  length $k$.


    \vspace{1.5mm}

    (1). For the base case that $m=0$, using the argument in the proof for \Cref{lem:wellfounded_ordering_MSL}, we can show that any $\varphi$ can only span finite descending $\ll$-chains.

        \vspace{1.5mm}

(2).  For the induction step, let $1\le m\in\mathbb{N}$ and suppose the statement is true for $(i)$ every formula $\psi$ with $\mathsf{DOI}(\psi)<m$  and $(ii)$ every formula $\chi$ such that $\mathsf{DOI}(\chi)=m$  and its length is smaller than $k$. We consider the case that $\mathsf{DOI}(\varphi)=m$ and the length is $k$.

\vspace{1.5mm}

 (2.1). When $\varphi$ is   $\neg \psi$, $\psi_1\land\psi_2$ or $\Box\psi$, it only descends along $\ll$ to strictly shorter formulas  without making the depth of iterative substitution operators bigger. By induction hypothesis, any such formula  only spans finite descending $\ll$-chains. Therefore $\varphi$ only spans finite descending $\ll$-chains. 

\vspace{1.5mm}

(2.2).  When $\varphi$ is  $\sub{p:=\psi}\chi$ and the main connective of $\chi$ is not an iterative substitution operator, the case can be proved with a similar way to that of \Cref{lem:wellfounded_ordering_MSL}. We skip the details. 



\vspace{1.5mm}

 (2.3).  When $\varphi$ is  $\sub{p:=\psi; (p:=\theta)^*}\chi$,  for every $n\in\mathbb{N}$, it descends to $\sub{p:=\psi; (p:=\theta)^n}\chi$. For each $\delta\in\{\psi,\theta,\chi\}$,  either $(i)$ $\mathsf{DOI}(\delta)<m$ or $(ii)$ $\mathsf{DOI}(\delta)=m$ and the length  of $\delta$ is smaller than $k$. By induction hypothesis, each of them  spans finite descending $\ll$-chains, and we write the following for three finite chains spanned by them:
\begin{center}
    \begin{tabular}{l@{\quad}l}
        $\varepsilon_{\psi}$: & $\psi\gg\psi_1\gg\psi_2\gg\dots\gg q_{\psi}$ \vspace{.8mm} \\
       $\varepsilon_{\theta}$:  &  $\theta\gg\theta_1\gg\theta_2\gg\dots\gg q_{\theta}$ \vspace{.8mm} \\
       $\varepsilon_{\chi}$:  &  $\chi\gg\chi_1\gg\chi_2\gg\dots\gg q_{\chi}$\\
    \end{tabular}
\end{center}
 
By definition, $\sub{p:=\psi; (p:=\theta)^n}\chi$ descends to $(\sub{(p:=\theta)^n}\chi)[\psi/p]$, i.e., $\sub{p:=\theta[\psi/p];(p:=\theta)^{n-1}}\chi$. It follows from Proposition \ref{prop:replacement_preserve_ordering_MISL} that:
\begin{center}
$\theta[\psi/p]\gg\theta_1[\psi/p]\gg\theta_2[\psi/p]\gg\dots\gg q_{\theta}[\psi/p]$
\end{center}
As the case in Lemma \ref{lem:wellfounded_ordering_MSL}, if $q_{\theta}\not=p$, then the above is a  finite  descending $\ll$-chain spanned by $\theta[\psi/p]$, otherwise by adding  $\varepsilon_{\psi}$ to the chain above we can obtain a finite chain spanned by $\theta[\psi/p]$. We write $\varepsilon_1$ for the resulting chain.

Next, $\sub{p:=\theta[\psi/p];(p:=\theta)^{n-1}}\chi$ descends to $\sub{p:=\theta[(\theta[\psi/p])/p];(p:=\theta)^{n-2}}\chi$. Again, by Proposition \ref{prop:replacement_preserve_ordering_MISL}, 
 \begin{center}
$\theta[(\theta[\psi/p])/p]\gg\theta_1[(\theta[\psi/p])/p]\gg\theta_2[(\theta[\psi/p])/p]\gg\dots\gg q_{\theta}[(\theta[\psi/p])/p]$
\end{center}  
As the case above, we can still get a finite $\ll$-chain spanned by $\theta[(\theta[\psi/p])/p]$, denoted $\varepsilon_2$. 
 
In what follows, we call $\theta[(\theta[\psi/p])/p]$ `{\em $2$-layer $\theta[\psi/p]$}', $\theta[(\theta[(\theta[\psi/p])/p])/p]$ `{\em $3$-layer $\theta[\psi/p]$}', and so on. Repeating the reasoning above, we can get a finite $\ll$-chain $\varepsilon_n$ spanned by the $n$-layer $\theta[\psi/p]$. Now we consider the following
\begin{center}
    $\chi [\dagger/p] \gg \chi_1[\dagger/p]\gg \chi_2[\dagger/p] \gg \dots\gg q_{\chi}[\dagger/p]$
\end{center}
where $\dagger$ is the $n$-layer $\theta[\psi/p]$. Like before, if $q_{\chi}$ is different $p$, then the above is a finite  chain spanned by  $\chi [\dagger/p]$, and if $q_{\chi}$ is $p$, then  by adding $\varepsilon_n$ to the sequence above we obtain a finite chain spanned by $\chi [\dagger/p]$.  Now one can see that for any $n\in\mathbb{N}$, $\sub{p:=\psi; (p:=\theta)^n}\chi$ only spans finite descending chains. Hence $\varphi$ only spans finite descending $\ll$-chains.  This completes the proof.
\end{proof}

We then define {\em the depth of a clean formula in $\mathcal{L}_{\MISL}$ $\dep(\varphi)$} to be the (ordinal) length of the longest descending $\ll$-chains spanned by $\varphi$. The details are as follows:

\begin{definition}[Depth of a clean $\mathcal{L}_{\MISL}$-formula]\label{def:formula_depth_MISL}
For a clean $\varphi\in\mathcal{L}_{\MISL}$, its {\em ordinal depth} $\dep(\varphi)$  is defined recursively by extending the clauses for $\mathcal{L}_{\MSL}$ in Definition \ref{def:formula_depth_MSL}  with the following:
    \begin{align*}
        \dep(\sub{p:=\psi;(p:=\chi)^*}\varphi)&:={\rm sup}_{n\in\mathbb{N}}\{\dep(\sub{p:=\psi;(p:=\chi)^n}\varphi)\}+1
    \end{align*}
\end{definition}

In the definition above, $\dep(\sub{p:=\psi;(p:=\chi)^*}\varphi)$ is always no less than $\omega+1$. Now we are ready to prove \Cref{thm:sub_equal_replacement_MISL} by induction along $\ll$, which is similar to the proof of \Cref{thm:sub_equal_replacement_MSL}. 

\begin{proof}[Proof of \Cref{thm:sub_equal_replacement_MISL}]
 It goes by transfinite induction on $\dep(\varphi[\psi/p])$. The base case, as well as the cases where the main connective of $\varphi$ is $\neg,\land$ or $\Diamond$ can be proved with the same argument as that in the proof of \Cref{thm:sub_equal_replacement_MSL}. Since $\varphi$ is clean,  the main connective of $\varphi$ cannot be an iterative substitution operator. Therefore we are left with the situation where $\varphi$ is $\sub{q:=\chi_1}\chi_2$. There are two different situations. 

    \vspace{1.5mm}

 (1). The main connective of $\chi_2$ is not an iterative substitution operator, in which case $\chi_2$ is a clean formula and so $\dep(\chi_2)$ is well defined. Then again we can prove the following equivalences through the same argument as that in the proof of \Cref{thm:sub_equal_replacement_MSL}, while changing $n$ for depth of clean formulas to be an ordinal number: 
    \begin{align*}
        \sub{p:=\psi}\sub{q:=\chi_1}\chi_2 &\leftrightarrow \sub{p:=\psi}(\chi_2[\chi_1/q])\\
        &\leftrightarrow (\chi_2[\chi_1/q])[\psi/p]\\
        &\leftrightarrow (\sub{q:=\chi_1}\chi_2)[\psi/p]
    \end{align*}
    This shows that $\sub{p:=\psi}\varphi\leftrightarrow\varphi[\psi/p]$.

    \vspace{1.5mm}

 (2). The main connective of $\chi_2$ is an iterative substitution operator. Since $\varphi=\sub{q:=\chi_1}\chi_2$ is clean,  $\chi_2$ can be written as  $\sub{(q:=\chi_3)^*}\chi_4$. Now, for any $n\in\mathbb{N}$, $\dep(\varphi)>\dep(\sub{q:=\chi_1;(q:=\chi_3)^n}\chi_4)$, and by induction hypothesis, the following is valid:
    \begin{center}
      $\sub{p:=\psi}\sub{q:=\chi_1;(q:=\chi_3)^n}\chi_4\leftrightarrow(\sub{q:=\chi_1;(q:=\chi_3)^n}\chi_4)[\psi/p]$.   
    \end{center}

 Let $(\mmodel,w)$ be a pointed model. Assume that $\mmodel,w\vDash_{\MISL}   \sub{p:=\psi}\sub{q:=\chi_1;(q:=\chi_3)^*}\chi_4$. Then, for some $m\in\mathbb{N}$, $\mmodel,w\vDash_{\MISL} \sub{p:=\psi}\sub{q:=\chi_1;(q:=\chi_3)^m}\chi_4$. So, $\mmodel,w\vDash_{\MISL} (\sub{q:=\chi_1;(q:=\chi_3)^m}\chi_4)[\psi/p]$, which then gives us  $\mmodel,w\vDash_{\MISL} (\sub{q:=\chi_1;(q:=\chi_3)^*}\chi_4)[\psi/p]$. The converse can be proved in a similar way. Therefore, $\sub{p:=\psi}\sub{q:=\chi_1;(q:=\chi_3)^*}\chi_4\leftrightarrow (\sub{q:=\chi_1;(q:=\chi_3)^*}\chi_4)[\psi/p]$ is valid.  The proof is completed.
\end{proof}

\subsection{Example revisited: backward induction}\label{sec:expressing-games}

Let us now get back to the applications of $\MISL$, and establish a connection  between backward induction as described in \Cref{sec:intro} and the corresponding iterative substitution operators. With respect to the board games defined in Scenario 2, we start by introducing several key notions, which are useful for both our proof and the illustration of the advantages of our approach:

\begin{definition}[Ranks of winning positions]\label{def:winning_position_rank}
Let $\mathcal{B} = (W, R)$ be a finite game board. For any $i\in \{0,1\}$ and any $k\in\mathbb{N}$, we define $\mathsf{Win}_i^k \subseteq W$ such that for any $s\in W$, $s\in \mathsf{ Win}_i^k$  when player $i$ has a strategy to win the game $\mathcal{G}_s=(\mathcal{B},s)$ within $k$ rounds. For all $0<k\in\mathbb{N}$, the positions in $\mathsf{Win}_i^k\setminus \mathsf{Win}_i^{k-1}$ are called {\em rank-$k$ winning positions of player $i$}, which means that player $i$ has a strategy to win the game $\mathcal{G}_t$ in exactly $k$ rounds. For the case when $k=0$, the {\em rank-$0$ winning positions of player $i$} are exactly the members of $\mathsf{Win}_i^0$.  In addition, when a node $u$ is not a finite rank winning position of any player, we call it {\em a rank-$\infty$ winning position of player $0$}, which means that player $0$ has a strategy to force  $\mathcal{G}_u$ to run infinitely. 
\end{definition}

For player $1$, we  define $\mathsf{Win}_1 := \bigcup_{k\in\mathbb{N}}\mathsf{Win}_1^k$, consisting of  the winning positions of player $1$.  Now, we are ready to show that the principle $\sub{p:=\Box\bot; (p:= p\lor \Box\Diamond p)^*}p$ involved in \Cref{sec:intro} defines the set:

\begin{proposition}
Let $\mathcal{B}=(W,R)$ be a finite board and $\mmodel$ be a model extending $\mathcal{B}$ with an arbitrary valuation function. Then, in the model, $\sub{p:=\Box\bot; (p:= p\lor \Box\Diamond p)^k}p$ is true exactly at the nodes in $\mathsf{Win}_1^k$. More generally,  
\[
\sub{p:=\Box\bot; (p:= p\lor \Box\Diamond p)^*}p
\]
captures the set $\mathsf{Win}_1$ of winning positions  of player $1$. 
\end{proposition}

\begin{proof} 
For convenience, let $R[s]=\{t:Rst\}$ denote the successors of $s\in W$. Define the function $[R]:\mathcal{P}(W)\rightarrow\mathcal{P}(W)$ as $[R]:X\mapsto \{s: R[s]\subseteq X\}$, and its dual function is $\sub{R}: X\mapsto \{s:R[s]\cap X\not=\emptyset\}$. Then, $\mathsf{Win}_1^0 = \{s: R[s] = \emptyset\}$. Using backward induction, one can see that 
\[
\mathsf{Win}_1^{k+1} = \bigcup_{i< k+1}\mathsf{Win}_1^i \cup [R](\sub{R}(\mathsf{Win}_1^k)),
\]
which implies that if the truth set of $p$ in the $k$-th stage  of the iterative substitution is $\mathsf{Win}_1^k$, then the truth set of $p$ in the $(k+1)$-th stage is $\mathsf{Win}_1^{k+1}$. Therefore, by induction, the truth set of $p$ in the $k$-th stage of the iterative substitution is exactly $\mathsf{Win}_1^k$. As a direct consequence, the truth set of $\sub{p:=\Box\bot; (p:= p\lor \Box\Diamond p)^*}p$ in $\mmodel$ is $\bigcup_{k}\mathsf{Win}_1^k = \mathsf{Win}_1$. 
\end{proof}

It is also instructive to point out that  on a finite board, $\mathsf{Win_1}$ can also be expressed by the formula 
\[
\mu p. (\Box\bot\lor\Box\Diamond p)
\]
of {\em the modal $\mu$-calculus} $\mu\ML$ \cite{10.1007/3-540-12896-4_370,Yde-mu-calculus}. We would like to argue that $\MISL$ has the capacity to analyze the nuances of winning positions' ranks in greater detail, though this discussion will be reserved for the following section.

\vspace{2mm}

{\bf Digression}\; Finally, let us briefly note that such the interplay between $\MISL$ and game-theoretic reasoning is widespread. Beyond the example provided, another instance is $\exists n\in \mathbb{N}\; (\Box\Diamond)^n p$, which is proposed and discussed in the context of the hide and seek game logic \cite{LHS-journal}.  However, the formula is not well-defined due to its infinite length in the cited work, but now it can be defined easily in $\MISL$ as $\sub{(p:= \Box\Diamond p)^*}p$. It is crucial for characterizing the existence of winning strategies in the games analyzed in \cite{LHS-journal}. For an in-depth exploration of  how the substitution operators capture the concept of $\exists n\in \mathbb{N}\; (\Box\Diamond)^n p$, we refer to \cite{Johan-fixedpoint2021}, and for more interactions between games and logic, we refer to, e.g., \cite{lig,graph-game-book}.

\section{On  relations between \texorpdfstring{$\MISL$}{} and other relevant logics}\label{sec:relation-with-other-logics}

In \Cref{sec:expressing-games}, we have shown that $\MISL$ can be used to reason about the winning positions for players, and in particular, given a finite game board, the formula $\sub{p:=\Box\bot; (p:= p\lor \Box\Diamond p)^*}p$
captures the winning positions $\mathsf{Win}_1$ of player $1$. However, as mentioned, this can also be defined with  $\mu\ML$. So, what is the advantage of proposing another logic, viz. $\MISL$? This question motivates us to explore the connections between $\MISL$ and other logical systems involving iterative computation, including $\mu\ML$, the {\em infinitary modal logic   $\ML^\infty$} \cite{infinity-modal-logic-axiomatization} extending the standard modal logic with countable conjunctions and disjunctions, and {\em the propositional dynamic logic} $\PDL$ \cite{pdl-bood}.\footnote{Moreover, the investigation into relation between $\MISL$ and {\em iterative modal relativization $\IMR$} \cite{IMR} is suggested in~\cite{Johan-fixedpoint2021}. In the next section, we will also provide results concerning their connection.}


 \subsection{Relation between \texorpdfstring{$\MISL$}{} and   \texorpdfstring{$\mu\ML$}{}}

We begin with the connections between $\MISL$ and $\mu\ML$. First of all, let us note the following:
\begin{remark}\label{remark:mu-calculus}
 Since $\MISL$ only allows finite steps of iteration, for simplicity, in what follows we only consider the fragment of $\mu\ML$ where the fixed point of every formula can be reached by a finite sequence of approximation, although in general $\mu\ML$ admits fixed point reached by ordinal sequences of approximation as well.  
\end{remark}

As mentioned, $\MISL$ is capable of reasoning about the ranks of winning positions (cf. \Cref{def:winning_position_rank}). For instance, consider the formula
\[
\sub{p:=\Box\bot; (p:= p\lor \Box\Diamond p)^*}(\Diamond p\land \Diamond(\neg p\land \Box\Diamond p))
\]
saying that the current state has different successors that are player $1$'s winning positions, and their ranks differ by one. In \Cref{fig:game_example}, node $d$ is of rank $0$ and node $c$ is of rank $1$. 
This  property  cannot be defined by $\mu\ML$: 

\begin{theorem}\label{thm:game_not_in_mu}
 The formula $\sub{p:=\Box\bot; (p:= p\lor \Box\Diamond p)^*}(\Diamond p\land\Diamond(\neg p\land \Box\Diamond p))$ is not definable by $\mu\ML$. 
\end{theorem}

\begin{proof}
 For simplicity, we write $\varphi$ for   $\sub{p:=\Box\bot; (p:= p\lor \Box\Diamond p)^*}(\Diamond p\land\Diamond(\neg p\land \Box\Diamond p))$. We  try to arrive at a contradiction by assuming that  there exists a formula $\varphi_{\mu}$ of $\mu\ML$ that is equivalent to $\varphi$. Since every formula of $\mu\ML$ can be translated into a formula of {\em monadic second order logic} $\MSO$ (see e.g., \cite{CompletenessMuMso}),  our assumption shows that there is an $\MSO$-formula $\alpha(x)$ such that $\mmodel,s\vDash_{\MISL} \varphi$ iff $\mmodel, [x\mapsto s]\vDash_{\MSO} \alpha(x)$ for any pointed model $(\mmodel,s)$.\footnote{By $\mmodel, [x\mapsto s]\vDash_{\MSO} \alpha(x)$ we mean that when we assign the value $s$ to the variable $x$, $\alpha(x)$ is true in $\mmodel$.} 
    
     We will make use of the following structure. Set alphabet $\Sigma = \{a,b,c\}$. Any word $w\in\Sigma^*$ can be viewed as a model $\mmodel_w = (\{1,2,\dots, |w|\}, P_a,P_b,P_c,<)$ such that:
     \begin{itemize}
         \item $<$ the usual ‘less than’ relation on natural numbers,
         \item $P_a,P_b$ and $P_c$ are unary predicates such that a natural number $i$ of the domain only satisfies $P_{w_i}$ where $w_i$ is the $i$-th symbol of $w$.
     \end{itemize}
     Also, for any $i,j\in\mathbb{N}$, we will use $i\le j$ for $i<j\lor i=j$. {\em A formula $\psi$ of $\MSO$ is said to define a language $\mathbb{L}$} when it is the case that  $w\in\mathbb{L}$ iff $\mmodel_w\vDash_{\MSO}\psi$. From  B\"uchi's theorem \cite{regular-language}, we know that $\MSO$-formulas define exactly the {\em regular languages}. Now, consider the following set of $\MSO$-formulas:
\begin{align*}
    \psi_1&:=\exists x(P_c x\land \forall y(P_c y\rightarrow y=x))\\
    \psi_2&:=\exists x(P_cx\land \forall y(P_ay\rightarrow y<x)\land \forall z(P_bz\rightarrow x<z)\\
 \psi_3&:=\forall x\forall y(Rxy\leftrightarrow (P_ay\land y<x\land \forall z(z<x\rightarrow z\le y))\lor (P_by\land x<y\land \forall z(z<y\rightarrow z\le x)))\\
 \psi_4&:=\forall x (P_cx\rightarrow \alpha(x))
\end{align*}
    Let $\psi:=\psi_1\land\psi_2\land\exists R(\psi_3\land\psi_4)$. It follows that if $\mmodel_w\vDash_{\MSO} \psi_1\land\psi_2$, then $w\in\{a^*cb^*\}$. When $\mmodel_w\vDash_{\MSO} \exists R(\psi_3\land \psi_4)$, $R$ satisfying $\psi_3$ should form a two-branch tree rooted at the state satisfying $P_c$:  one branch consisting of predecessor relation in $P_a$, and the other branch consisting of successor relation in $P_b$. From $\psi_4$ one can see that $\psi$ defines the non-regular language $\{a^{2m}cb^{2n}: m=n+1$ or $n=m+1\}$, which should not be defined by $\MSO$-formulas. 
\Cref{fig:game_nonregular_language} provides such a model $\mmodel_w$.

\begin{figure}[htbp]
    \centering
\begin{tikzpicture}
\node(2m+1)[circle,draw,inner sep=0pt,minimum size=1mm,fill=black] at (0,0) [label=above:$2m+1$][label=below:$P_c$]{};

\node(2m)  [circle,draw,inner sep=0pt,minimum size=1mm,fill=black] at (-1.5,0) [label=above:$2m$][label=below:$P_a$]{};

\node(2m-1)   [circle,draw,inner sep=0pt,minimum size=1mm,fill=black] at (-3,0) [label=above:$2m-1$][label=below:$P_a$]{};

\node(2m-2)   at (-4.5,0){$\dots$};

\node(2)   [circle,draw,inner sep=0pt,minimum size=1mm,fill=black] at (-5,0) [label=above:$2$][label=below:$P_a$]{};

\node(1)   [circle,draw,inner sep=0pt,minimum size=1mm,fill=black] at (-6.5,0) [label=above:$1$][label=below:$P_a$]{};

\node(2m+2)  [circle,draw,inner sep=0pt,minimum size=1mm,fill=black] at (1.5,0) [label=above:$2m+2$][label=below:$P_b$]{};

\node(2m+3)   [circle,draw,inner sep=0pt,minimum size=1mm,fill=black] at (3,0) [label=above:$2m+3$][label=below:$P_b$]{};

\node(2m+4)   at (4.5,0){$\dots$};

\node(4m+2)   [circle,draw,inner sep=0pt,minimum size=1mm,fill=black] at (5,0) [label=above:$4m+2$][label=below:$P_b$]{};

\node(4m+3)   [circle,draw,inner sep=0pt,minimum size=1mm,fill=black] at (6.5,0) [label=above:$4m+3$][label=below:$P_b$]{};

\draw[->](2m+1) to node [below] {$R$} (2m);
\draw[->](2m) to node [below] {$R$} (2m-1);
\draw[->](2m-1) to node [below] {$R$} (2m-2);
\draw[->](2) to node [below] {$R$} (1);

\draw[->](2m+1) to node [below] {$R$} (2m+2);

\draw[->](2m+2) to node [below] {$R$} (2m+3);

\draw[->](2m+3) to node [below] {$R$} (2m+4);

\draw[->](4m+2) to node [below] {$R$} (4m+3);

\draw [decorate,decoration={brace,amplitude=10pt},xshift=-4pt,yshift=0pt] (-1.5,-.6) -- (-6.5,-.6) node [black,midway,yshift=-.5cm] {\footnotesize $2m$};

\draw [decorate,decoration={brace,amplitude=10pt},xshift=-4pt,yshift=0pt] (6.5,-.6)  -- (1.5,-.6) node [black,midway,yshift=-.5cm] {\footnotesize $2m+2$};
\end{tikzpicture}
    \caption{A model $\mmodel_w$ such that $\mmodel_w\vDash_{\MSO} \psi_1\land \psi_2\land\exists R(\psi_3\land \psi_4)$, where $w=a^{2m}cb^{2m+2}$}
\label{fig:game_nonregular_language}
\end{figure}
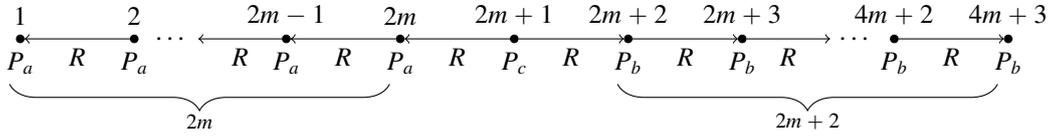

    Therefore the original $\MISL$-formula $\varphi$ is not equivalent to any $\mu\ML$-formula. 
\end{proof}

On the other hand, recall that to calculate the least fixed-point of a monotone function in $\mu\ML$, one starts with $\bot$ and computes the approximation sequence. As mentioned in~\cite{Johan-fixedpoint2021}, if the least fixed point of $\varphi$ can be reached by a finite sequence of approximation, then the  $\mu\ML$-formula $\mu p.\varphi$ is equivalent to
\begin{center}
 $\sub{p:=\bot;(p:=\varphi)^*}p $.\footnote{But it is important to emphasize that we have this equivalence with the convention that in $\mu\ML$ only  finite sequences of approximation are allowed (Remark \ref{remark:mu-calculus}). And, we do {\em not} have the equivalence when ordinal sequences are allowed in $\mu\ML$: for instance, $\mu p.\Box p$ is true $(\mmodel,s)$ depicted in \Cref{fig:infinite_branching} while $\sub{p:=\bot;(p:=\Box p)^*}p$ is false there.}   
\end{center}
Now, combining this and \Cref{thm:game_not_in_mu}, we have the following as a corollary:

\begin{corollary}\label{coro:mu-contained-in-MISL}
  When only finite sequences of approximation are allowed, $\mu\ML$ is strictly contained in $\MISL$.  
\end{corollary}

 For the general case without the restriction imposed in the corollary above, it remains to be determined whether or not $\mu\ML$ is still contained in $\MISL$. 

\subsection{Relation between \texorpdfstring{$\MISL$}{} and \texorpdfstring{$\ML^\infty$}{}} 

Let us  now proceed to discuss $\ML^\infty$. As stated in \cite{open-minds},  $\mu\ML$ under the restriction imposed in Remark \ref{remark:mu-calculus} can be embedded into $\ML^\infty$. Now, we provide an analogous result stating  that  $\MISL$ is also a fragment of $\ML^\infty$:

\begin{theorem}\label{thm:frag_ML_infty}
    There is a translation $T:\mathcal{L}_\MISL\rightarrow \mathcal{L}_{\ML^\infty}$ such that for any formula $\varphi\in \mathcal{L}_\MISL$, it holds that:
    \[ \mmodel,s\vDash_{\MISL}\varphi\quad\text{ iff }\quad \mmodel,s\vDash_{\ML^\infty}T(\varphi).
    \]
\end{theorem}

\begin{proof}
We are going to show explicitly the construction of  $T$. Let $T = T_0\circ T_1$, where $T_0:\mathcal{L}_{\mathsf{clean}}\rightarrow \mathcal{L}_{\ML^\infty}$ is a translation from  clean $\MISL$-formulas to $\ML^\infty$-formulas and $T_1:\mathcal{L}_\MISL\rightarrow \mathcal{L}_{\mathsf{ clean}}$ is a translation from arbitrary $\MISL$-formulas to the clean ones. The translation $T_1$ can be obtained by bound variable renaming (\ref{eqn:MISL_bound_variable_renaming}), and $T_0$ keeps  atoms $p$  the same,  permutes with Boolean connectives and $\Diamond$, and  
\begin{align*}
 T_0(\sub{p:=\psi}\varphi)&:=T_0(\varphi)[T_0(\psi)/p]\\
 T_0(\sub{p:=\psi_0;(p:=\psi_1)^*}\varphi)&:=\textstyle \bigvee_{n\in \mathbb{N}}T_0(\sub{p:=\psi_0;(p:=\psi_1)^n}\varphi)
\end{align*}
where formula $\varphi$ in the clause for $T_0(\sub{p:=\psi}\varphi)$ does not have its  main connective being an iterative substitution operator.\footnote{Also, note that the replacement $[T_0(\psi)/p]$ of the variable $p$ is carried out in an $\ML^\infty$-formula, which substitutes all occurrences of propositional letter $p$ with $T_0(\psi)$. } It can be verified that the translation $T_0$ and thus $T$ translate formulas to their equivalent forms. Therefore $\MISL$ can be translated to $\ML^\infty$. 
\end{proof}

W.r.t. the relation between $\MISL$ and $\ML^\infty$, it remains to determine the following:

\vspace{1.5mm}

\noindent{\bf Open problem}\; Is $\MISL$ strictly contained in $\ML^\infty$? Are there some classes of models over which they are equivalent?

\subsection{Relation between \texorpdfstring{$\MISL$}{} and \texorpdfstring{$\PDL$}{}} 

 Finally, let us turn to another logical framework, $\PDL$, involving finite steps of iteration.  Since the language $\mathcal{L}_\PDL$ may contain more than one primitive programs (e.g., $\Delta = \{a,b,\dots\}$), we also add the corresponding modalities $\sub{a},\sub{b},\dots$ to the language of $\MISL$, and write $\mathcal{L}_{\MISL(\Delta)}$ and $\MISL(\Delta)$ for the resulting language and logic respectively. We now show that $\PDL$ can be embedded into $\MISL(\Delta)$:  

\begin{theorem}
There exists a translation $T:\mathcal{L}_\PDL\rightarrow \mathcal{L}_{\MISL(\Delta)}$ 
such that for any $\PDL$-formula $\varphi$, it holds that
\begin{center}
$\mmodel,s\vDash_\PDL\varphi$\quad iff \quad $\mmodel,s\vDash_{\MISL(\Delta)}  T(\varphi)$.
\end{center}
where $\mmodel$ is a model containing relations $R_{c}$ with $c\in\Delta$.
\end{theorem}

\begin{proof} 
A desired translation $T$ keeps propositional variables  the same,  permutes with Boolean connectives and the modalities $\sub{a}$ for primitive programs, and
\begin{align*}
 T(\sub{\pi_1;\pi_2}\varphi)&:=T(\sub{\pi_1}\sub{\pi_2}\varphi)\\
    T(\sub{\pi_1\cup \pi_2}\varphi)&:=T(\sub{\pi_1}\varphi) \lor T(\sub{\pi_2}\varphi)\\
T(\sub{\pi^*}\varphi)&:= \sub{q:=T(\varphi);(q:=T(\sub{\pi}q))^*}q
\end{align*}
 where $\pi_1;\pi_2$, $\pi_1\cup \pi_2$, and $\pi^*$ are complex programs, and  $q$  is a fresh variable. We leave the proof for the correctness of the translation to the reader. 
\end{proof}

Having seen a number of expressive languages that can be encoded within $\MISL$,\footnote{Moreover, the undecidable framework  iterative modal relativization $\IMR$ \cite{IMR} can also be encoded within $\MISL$ (see \Cref{def:IMR_to_MISL}). } we guess that the complexity of the framework should be high. This is confirmed by the findings presented in the next section, which studies various properties of $\MISL$. 

\section{Expressiveness and undecidability of \texorpdfstring{$\MISL$}{}}\label{sec:properties-misl}

In this section, we explore further properties of $\MISL$, including its expressive power and the decidability of its satisfiability problem. We will show that  even though $\MISL$ is much more powerful than $\ML$, truth of $\MISL$-formulas is still invariant under the standard notion of bisimulation for $\ML$ (\Cref{sec:bisimulation}); meanwhile,  the logic will be proved to be highly undecidable: its satisfiability problem is $\Sigma_1^1$-complete (\Cref{sec:undecidable}). Moreover, when we restrict our attention to the very simple class of finite tree models, the logic is still shown to be undecidable (\Cref{sec:undecidable-tree}).


\subsection{Bisimulation for \texorpdfstring{$\MISL$}{}} \label{sec:bisimulation}

Let us start by considering the expressive power of $\MISL$. Specifically, we will first introduce the details for the notion of bisimulation for $\ML$, and then show that although $\MISL$ is equipped with additional substitution operators, it is still invariant under the  established notion of bisimulation.

\begin{definition}[Bisimulation \cite{blackburn2004modal}]\label{def:bisimulation}
Let $\mmodel=(W,R,V)$ and $\mmodel'=(W',R',V')$  be two models and let $s \in W$ and  $s' \in W'$. A non-empty relation $Z\subseteq W\times W'$ is called a {\em bisimulation  between $\mmodel$ and $\mmodel'$}, if the following are satisfied:
\begin{itemize}[align=left,itemindent=-2em]
\item[\bf{Atom:}] If $sZs'$, then $(\mmodel, s)$  and $(\mmodel', s')$ satisfy the same propositional letters.
\item[\bf{Zig:}]  If $sZs'$ and there exists $t\in W$ such that $Rst$, then there exists $t'\in W'$ such that $R's't'$ and $tZt'$.
\item[\bf{Zag:}] If $sZs'$ and there exists $t'\in W'$ such that $R's't'$, then there exists $t\in W$ such that $Rst$ and $tZt'$. 
\end{itemize}
\end{definition}

When there is a bisimulation $Z$ linking states $s\in \mmodel$ and $s'\in \mmodel'$, we say that $(\mmodel,s)$ and $(\mmodel',s')$ are {\em bisimilar}, and write  $(\mmodel,s) \bis (\mmodel',s')$, or just $s \bis s'$ when the models are clear from the context. 

\begin{theorem}[Invariance under bisimulation]\label{thm:bis_inv}
   If $(\mmodel,s) \bis (\mmodel', s')$, then for any $\varphi\in\mathcal{L}_{\MISL}$,
    \[
    \mmodel,s\vDash_{\MISL} \varphi\; \text{ iff }\; \mmodel',s'\vDash_{\MISL} \varphi.
    \]
\end{theorem}
\begin{proof}
We prove by induction on formulas $\varphi$. When $\varphi$ is an $\ML$-formula, it is easy to see that the equivalence holds.  We now consider other cases. In what follows, we write $Z$ for a bisimulation that links $s$ and $s'$.

  \vspace{1.5mm}

 (1). Formula $\varphi$ is $\sub{p:=\psi}\chi$. Assume that $(t,t')\in Z$. We are going to show that $(\mmodel|_{p:=\psi},t)\bis (\mmodel'|_{p:=\psi},t')$.  Models $\mmodel|_{p:=\psi}$ and $\mmodel$ may only differ on  the truth set of $p$, and similarly for $\mmodel'|_{p:=\psi}$ and $\mmodel'$. So, we only need to prove that  $\mmodel|_{p:=\psi},t$ and $\mmodel'|_{p:=\psi},t'$ agree on the propositional variable $p$. By induction hypothesis,  $\mmodel,t\vDash_{\MISL} \psi$ iff $\mmodel',t'\vDash_{\MISL} \psi$. Hence, $\mmodel|_{p:=\psi},t\vDash_{\MISL} p$ iff $\mmodel'|_{p:=\psi},t'\vDash_{\MISL} p$. Therefore,   $(\mmodel|_{p:=\psi},s)\bis (\mmodel'|_{p:=\psi},s')$. Again, by induction hypothesis, $\mmodel|_{p:=\psi},s\vDash_{\MISL} \chi$ iff $\mmodel'|_{p:=\psi},s'\vDash_{\MISL} \chi$. Thus, based on the semantics,  $\mmodel,s\vDash_{\MISL} \sub{p:=\psi}\chi$ iff $\mmodel',s'\vDash_{\MISL} \sub{p:=\psi}\chi$.
    
\vspace{1.5mm}

(2). Formula $\varphi$ is $\sub{(p:=\psi)^*}\chi$. Assume that $(t,t')\in Z$. We show by induction that for all $n\in\mathbb{N}$, it holds that $(\mmodel|_{(p:=\psi)^n},t)\bis (\mmodel'|_{(p:=\psi)^n},t')$. The base case $n=0$ is direct. Also, for the case that $n=1$,  we have known from the reasoning in case (1) that $(\mmodel|_{p:=\psi},t)\bis (\mmodel'|_{p:=\psi},t')$. For any $n\ge 2$, by induction hypothesis, we have $(\mmodel|_{(p:=\psi)^n},t)\bis (\mmodel'|_{(p:=\psi)^n},t')$. Now, using the reasoning similar to that in case (1), we can show that $(\mmodel|_{(p:=\psi)^{n+1}},t)\bis (\mmodel'|_{(p:=\psi)^{n+1}},t')$.

 Now assume that $\mmodel,s\vDash_{\MISL}\sub{(p:=\psi)^*}\chi$. Then, $\mmodel,s\vDash_{\MISL}\sub{(p:=\psi)^m}\chi$ for some $m\in\mathbb{N}$, which indicates that $\mmodel|_{(p:=\psi)^m},s\vDash_{\MISL} \chi$. It follows from  the reasoning above that $(\mmodel|_{(p:=\psi)^{m}},s)\bis (\mmodel'|_{(p:=\psi)^{m}},s')$. By induction hypothesis,  $\mmodel|_{(p:=\psi)^m},s\vDash_{\MISL} \chi$ implies $\mmodel'|_{(p:=\psi)^m},s'\vDash_{\MISL} \chi$. So, $\mmodel',s'\vDash_{\MISL}\sub{(p:=\psi)^m}\chi$, which gives us $\mmodel',s'\vDash_{\MISL}\sub{(p:=\psi)^*}\chi$. The converse direction can be proved similarly. This completes the proof.
\end{proof}

\subsection{ Satisfiability problem of \texorpdfstring{$\MISL$}{}: \texorpdfstring{$\Sigma_1^1$}{}-complete}\label{sec:undecidable} 

In this part, we prove that the satisfiability problem for $\MISL$ is undecidable, by encoding the same for the {\em iterative modal relativization} $\IMR$, whose complexity is known to be $\Sigma_1^1$-complete \cite{IMR}. Moreover, as we shall see, the logic is  undecidable even when we consider certain simple classes of models, for example, finite tree models. As a warm-up, we first show the following:

\begin{proposition}\label{prop:non-fmp}
    $\MISL$ lacks the finite model property. 
\end{proposition}
\begin{proof}
It suffices to construct an $\MISL$-formula that can only have infinite models. Consider the following:
\[
[p:=\Box\bot;(p:= \neg p \land \Box p)^*]\Diamond p.
\]
  
First, the formula is satisfiable. One can check that it is true at the state $s$ in model  $\mmodel$ depicted in \Cref{fig:infinite_branching}.

\vspace{1.5mm}

Next, given a model in which the formula is true, one can see that the truth set of $p$, obtained at the $n$-stage of the  iteration, consists of the states of height $n$, and the formula says that there is no finite upper bound on the height of the current state. Therefore, any model satisfying the formula must be infinite, as needed. \qedhere
\end{proof}

\begin{figure}[htbp]
    \centering
\begin{tikzpicture}
\node(1)[circle,draw,inner sep=0pt,minimum size=1mm,fill=black] at (0,0)  [label=left:$s$] {};
\node(2)[circle,draw,inner sep=0pt,minimum size=1mm,fill=black] at (1,.7) {};

\node(3)[circle,draw,inner sep=0pt,minimum size=1mm,fill=black] at (1,0) {};

\node(4)[circle,draw,inner sep=0pt,minimum size=1mm,fill=black] at (2,0) {};

\node(5)[circle,draw,inner sep=0pt,minimum size=1mm,fill=black] at (1,-.7) {};

\node(6)[circle,draw,inner sep=0pt,minimum size=1mm,fill=black] at (2,-1.4) {};

\node(7)[circle,draw,inner sep=0pt,minimum size=1mm,fill=black] at (3,-2.1) {};

\node(8) at (1,-1.4)  {$\cdots$};

\draw[->](1) to (2);

\draw[->](1) to (3);

\draw[->](3) to (4);
\draw[->](1) to (5);

\draw[->](5) to (6);

\draw[->](6) to (7);

\end{tikzpicture}
    \caption{An infinite model $\mmodel$ such that $\mmodel,s\vDash_{\MISL} [p:=\Box\bot;(p:= \neg p \land \Box p)^*]\Diamond p$}
    \label{fig:infinite_branching}
\end{figure} 

Next we figure out the exact complexity of the satisfiability problem for $\MISL$:

\begin{theorem}[Undecidability of $\MISL$]\label{thm:undecidability}
The satisfiability problem for $\MISL$ is undecidable and is $\Sigma_1^1$-complete. 
\end{theorem}

To prove \Cref{thm:undecidability}, we first show the satisfiability problem for $\MISL$ is undecidable and $\Sigma_1^1$-hard, by embedding iterative modal relativization $\IMR$ \cite{IMR} into $\MISL$. Here is a quick introduction to $\IMR$:

\vspace{2mm}

\noindent{\bf Basics of $\IMR$}\;  The language $\mathcal{L}_{\IMR}$ of $\IMR$ extends $\mathcal{L}_{\ML}$ with the public announcement operators $[\psi]$ and their finitely iterative generalizations $[\psi^*]$, and the common knowledge operator $\Box^*$ (i.e., the reflexive-transitive closure of $\Box$). 

\vspace{2mm}

\noindent To show that $\IMR$ can be translated into $\MISL$, let us first prove that the simplest relativization operator $[p]$ can be mimicked with $\MISL$. To facilitate discussion, let us define a function $(\cdot)^p$ on $\MISL$-formulas such that $(\varphi)^p$ intuitively captures the relativization $[p]\varphi$. Details are as follows: 

\begin{definition}[Intermediate relativization]\label{def:simple_relativization}
Let $p\in {\bf P}$ be a propositional variable. For any  $\varphi\in\mathcal{L}_{\MISL}$ such that $p\not\in bv(\varphi)$,   we define $(\varphi)^p$ recursively as follows: 
\begin{center}
    \begin{tabular}{rl}
   $(q)^p:= p\rightarrow q$ &  $(\neg\varphi)^p:= p\to \neg (\varphi)^p$ \vspace{1mm} \\  $(\varphi_1\land\varphi_2)^p:= (\varphi_1)^p\land (\varphi_2)^p$& $(\Box\varphi)^p:= p\rightarrow \Box(\varphi)^p$ \vspace{1mm} \\ 
   $(\sub{q:=\varphi_1}\varphi_2)^p:= \sub{q:=(\varphi_1)^p}(\varphi_2)^p$& $(\sub{(q:=\varphi_1)^*}\varphi_2)^p:= \sub{(q:=(\varphi_1)^p)^*}(\varphi_2)^p$
\end{tabular}
\end{center}
\end{definition}

In the definition above, the restriction $p\not\in bv(\varphi)$ is used to exclude the situation that $p$ is bounded by some substitution operators in $\varphi$. 
Let us now show the following, which intuitively states that  $(\varphi)^p$ is equivalent to $[p]\varphi$: 

\begin{lemma}\label{lem:recursion_PAL}
Let $\mmodel=(W,R,V)$ be a model, $s\in W$ and  $p\in{\bf P}$. For any $\varphi\in\mathcal{L}_\MISL$ such that $p\not\in bv(\varphi)$, it holds that
\begin{center}
  $ \mmodel, s\vDash_{\MISL} (\varphi)^p$ \quad iff \quad  $s\not\in W^p$ or $\mmodel^p, s\vDash_{\MISL} \varphi$,
\end{center}
where $\mmodel^p=(W^p, R^p, V^p)$ is the model obtained by relativizing $p$ in $\mmodel$, i.e.,
\begin{itemize}
    \item $W^p:= \{s\in W: \mmodel,s\vDash_{\MISL} p\}$, 
    \item $R^p:= R\cap (W^p\times W^p)$,
    and 
    \item  for any $q\in{\bf P}$, $V^p(q):= V(q)\cap W^p$. 
\end{itemize}
\end{lemma}

\begin{proof}
It is by induction on  $\varphi$. The cases for the Boolean connectives and $\Box$ are straightforward. We now consider the cases for substitution operators.

\vspace{1.5mm}

(1). Formula  $\varphi$ is $\sub{q:=\psi}\chi$. Then, by definition, $(\varphi)^p=\sub{q:=(\psi)^p}(\chi)^p$. Since $p\not\in bv(\varphi)$, $p$ is  different from $q$. We are going to show that $\mmodel^p|_{q:=\psi}$ is the same model as $(\mmodel|_{q:=(\psi)^p})^p$. Notice that if they can be different, then their difference can only be the truth sets of $q$.  By induction hypothesis, 
    \begin{align*}
        V^p|_{q:=\psi}(q) &= \{s\in W^p:\mmodel^p,s\vDash_{\MISL} \psi\} \\
        &= \{s\in W^p:\mmodel,s\vDash_{\MISL} (\psi)^p\} \\
        &= (V|_{q:=(\psi)^p})^p(q)
    \end{align*}
    Hence the two models are the same. Therefore, 
    \begin{align*}
        \mmodel,s\vDash_{\MISL} \sub{q:=(\psi)^p}\chi^p &\; 
        \Leftrightarrow \; s\not\in W^p\text{ or }(\mmodel|_{q:=(\psi)^p})^p,s\vDash_{\MISL}\chi\\
         &\; 
        \Leftrightarrow \; s\not\in W^p\text{ or }\mmodel^p|_{q:=\psi},s\vDash_{\MISL} \chi\\
        &\; 
        \Leftrightarrow \; s\not\in W^p\text{ or }\mmodel^p,s\vDash_{\MISL} \sub{q:=\psi}\chi
    \end{align*}

(2). Formula $\varphi$ is $\sub{(q:=\psi)^*}\chi$. Then, by definition, $(\varphi)^p=\sub{(q:=(\psi)^p)^*}(\chi)^p$. Since $p\not\in bv(\varphi)$, $p$ is not $q$. 
Let us first prove that for any $n\in\mathbb{N}$, it holds that
\begin{center}
$\mmodel,s\vDash_{\MISL} \sub{(q:=(\psi)^p)^n}\chi^p \; \Leftrightarrow  \; s\not\in W^p\text{ or }\mmodel^p,s\vDash_{\MISL} \sub{(q:=\psi)^n}\chi.$    
\end{center}
The case for $n=0$ holds immediately by induction hypothesis. Let us consider the case that $n=k+1$.
  \begin{align*}
     \mmodel,s\vDash_{\MISL} \sub{(q:=(\psi)^p)^{k+1}}(\chi)^p
       &\; 
        \Leftrightarrow \;  \mmodel,s\vDash_{\MISL} \sub{q:=(\psi)^p}\sub{(q:=(\psi)^p)^{k}}(\chi)^p\\
        &\; 
        \Leftrightarrow \;  \mmodel|_{q:=(\psi)^p},s\vDash_{\MISL}  \sub{(q:=(\psi)^p)^{k}}(\chi)^p\\
     &\; 
        \Leftrightarrow \;  s\not\in W^p\text{ or }(\mmodel|_{q:=(\psi)^p})^p,s\vDash_{\MISL} \sub{(q:=\psi)^k} \chi  \\
      &\; 
        \Leftrightarrow \; s\not\in W^p\text{ or } \mmodel^p|_{q:=\psi},s\vDash_{\MISL} \sub{(q:=\psi)^k} \chi  \\
      &\; 
        \Leftrightarrow \;  s\not\in W^p\text{ or }\mmodel^p,s\vDash_{\MISL} \sub{(q:=\psi)^{k+1}}\chi    
  \end{align*}
  The third equivalence holds by induction hypothesis. The fourth can be obtained by the same reasoning as that in  (1). 

  Now, assume that $\mmodel,s\vDash_{\MISL} (\sub{(q:=\psi)^*}\chi)^p$. Then, $\mmodel,s\vDash_{\MISL}\sub{(q:=(\psi)^p)^m}(\chi)^p$ for some $m\in\mathbb{N}$. By the reasoning above, $s\not\in W^p$ or $\mmodel^p,s\vDash_{\MISL} \sub{(q:=\psi)^m}\chi$. The latter implies that $\mmodel^p,s\vDash_{\MISL} \sub{(q:=\psi)^*}\chi$, as needed. The other direction can be proved in a similar way.
 This completes the proof.
\end{proof} 

Now we are ready to show the details of the translation:
\begin{definition}[Translation from $\IMR$ to $\MISL$]\label{def:IMR_to_MISL}
  The translation $T:\mathcal{L}_{\IMR}\rightarrow \mathcal{L}_\MISL$ is  given by the following:
\begin{center}
 $T(p):=p$ \quad  $T(\neg \varphi):=\neg T(\varphi)$\quad  $T(\varphi_1\land\varphi_2):=T(\varphi_1)\land T(\varphi_2)$ \vspace{1mm} \\
 $T(\Box\varphi):=\Box T(\varphi)$\quad 
        $T(\Box^*\varphi):=[q:=T(\varphi); (q:=\Box q)^*]q$ \vspace{1mm} \\
$T([\psi]\varphi):=\sub{q:=T(\psi)}(T(\varphi))^q$ \quad  $T(\sub{\psi^*}\varphi):=\sub{q:=\top; (q:=q\land (T(\psi))^q)^*}(T(\varphi))^q$
\end{center}
    where $q$ is a fresh variable.
\end{definition}

\begin{theorem}\label{thm:translation_IMR}
For any pointed model $(\mmodel,s)$ and any formula  $\varphi\in\mathcal{L}_{\IMR}$, we have the following:
\begin{center}
$\mmodel,s\vDash_\IMR \varphi$\quad iff\quad $\mmodel,s\vDash_{\MISL}  T(\varphi)$
\end{center}
where  $T$ is the translation given in  \Cref{def:IMR_to_MISL}.
\end{theorem}

\begin{proof}
It is by  induction on $\varphi\in\mathcal{L}_{\IMR}$.  The cases for Boolean formulas and $\Box\varphi$  are easy to prove. We consider others.

 \vspace{1.5mm}

  (1). Formula $\varphi$ is $\Box^*\psi$. By definition, $T(\varphi) = [q:=T(\psi);(q:=\Box q)^*]q$. Then, we have the following sequence: 
  \begin{align*}
      \mmodel,s\vDash_\IMR \varphi& \;\Leftrightarrow\;  \text{~for all~} n\in\mathbb{N},\, \mmodel,s\vDash_\IMR \Box^n\psi\\      
      &\;\Leftrightarrow\;   \text{~for all~} n\in\mathbb{N}, \, \mmodel,s\vDash_{\MISL} \Box^n T(\psi)\\
      &\;\Leftrightarrow\;  \text{~for all~} n\in\mathbb{N}, \, \mmodel,s\vDash_{\MISL}  [q:=T(\psi)]\Box^n q\\ 
      &\;\Leftrightarrow\;  \text{~for all~} n\in\mathbb{N}, \, \mmodel,s\vDash_{\MISL}  [q:=T(\psi);(q:=\Box q)^n]q\\ 
      &\;\Leftrightarrow\;   \mmodel,s\vDash_{\MISL} [q:=T(\psi);(q:=\Box q)^*]q\\
      &\;\Leftrightarrow\;   \mmodel,s\vDash_{\MISL}  T(\varphi)
  \end{align*}
  The second equivalence holds by induction hypothesis.

  \vspace{1.5mm}

  (2).  Formula $\varphi$ is $[\psi]\chi$. Take a fresh variable $q$, and $T(\varphi) = \sub{q:=T(\psi)}T(\chi)^q$. We write $\mmodel^\psi = (W^\psi, R^\psi, V^\psi)$ for the model obtained by relativizing $\psi\in \mathcal{L}_{\IMR}$.  By  induction hypothesis, $\mmodel,s\vDash_\IMR \psi$ iff $\mmodel,s\vDash_{\MISL} T(\psi)$. So,   $W^\psi$ is the same as the domain of the model $(\mmodel|_{q:=T(\psi)})^q$ that is obtained by relativizing $q$ in $\mmodel|_{q:=T(\psi)}$. Models  $\mmodel^\psi$ and $\left(\mmodel|_{q:=T(\psi)}\right)^q$ may only differ on the valuation of $q$. Since $q$ does not appear in $\chi$, it holds that:
  \[
  \mmodel^\psi,s\vDash_\IMR \chi \;\Leftrightarrow\; \left(\mmodel|_{q:=T(\psi)}\right)^q,s\vDash_\IMR \chi.
  \]
Then, the following sequence of equivalences holds:
  \begin{align*}
      \mmodel,s\vDash_\IMR \varphi& \;\Leftrightarrow\;  \mmodel,s\not\vDash_\IMR \psi \text{ or }\mmodel^\psi,s\vDash_\IMR \chi\\
      & \;\Leftrightarrow\; 
      \mmodel,s\not\vDash_\IMR \psi
      \text{ or }      \left(\mmodel|_{q:=T(\psi)}\right)^q,s\vDash_\IMR \chi\\
      & \;\Leftrightarrow\; \mmodel|_{q:=T(\psi)},s\not\vDash_{\MISL}  q
      \text{ or }      \left(\mmodel|_{q:=T(\psi)}\right)^q,s\vDash_{\MISL}   T(\chi)\\
      & \;\Leftrightarrow\; \mmodel|_{q:=T(\psi)},s\vDash_{\MISL}   (T(\chi))^q \\
      & \;\Leftrightarrow\; \mmodel,s\vDash_{\MISL}  \sub{q:=T(\psi)}(T(\chi))^q\\
      & \;\Leftrightarrow\; \mmodel,s\vDash_{\MISL}  T(\varphi)
  \end{align*}
  The third equivalence holds by induction hypothesis, and the fourth one is obtained by \Cref{lem:recursion_PAL}.

\vspace{1.5mm}
    
   (3).  Formula $\varphi$ is $\sub{\psi^*}\chi$.   Pick a fresh variable $q$, and $T(\varphi)=\sub{q:=\top;(q:=q\land (T(\psi))^q)^*}(T(\chi))^q$. For any $n\in\mathbb{N}$, we write $\mmodel^{\psi^n}$ for the model obtained by repeatedly relativizing $\psi\in \mathcal{L}_{\IMR}$ for $n$ times. By induction hypothesis, for any  $(\mmodel,s)$ and $\varphi_0\in\{\psi,\chi\}$, we have the following: 
   \begin{center}
     $\mmodel,s\vDash_\IMR  \varphi_0 \;\Leftrightarrow\; \mmodel,s\vDash_{\MISL}  T(\varphi_0)$. \qquad ($\ddagger$)
   \end{center}
   
   Now, let us first show by induction on $n\in\mathbb{N}$ that for any $n\in \mathbb{N}$, $\mmodel^{\psi^n}$ and $\left(\mmodel|_{q:=\top;(q:=q\land (T(\psi))^q)^n}\right)^q$ may only disagree on the valuation of $q$.  

   \vspace{1mm}
   
   For the base case that  $n=0$, $\mmodel^{\psi^n} = \mmodel$ and $\left(\mmodel|_{q:=\top;(q:=q\land (T(\psi))^q)^n}\right)^q = \mmodel|_{q:=\top}$. Therefore the two models can only disagree on the valuation of $q$. 

      \vspace{1mm}

   For induction step, suppose that all the  cases for $n-1$ are proved, where $n\ge 1$. Then $\left(\mmodel|_{q:=\top;(q:=q\land (T(\psi))^q)^{n-1}}\right)^q$ and $\mmodel^{\psi^{n-1}}$ can only disagree on the valuation of $q$. Then, for the domain of $\mmodel^{\psi^n} = \left(\mmodel^{\psi^{n-1}}\right)^\psi$, we have the following:  
   \begin{align*}
        &\{s\in \mmodel^{\psi^{n-1}}: \mmodel^{\psi^{n-1}},s\vDash_\IMR \psi\}\\
       =& \{s\in \left(\mmodel|_{q:=\top;(q:=q\land (T(\psi))^q)^{n-1}}\right)^q: \left(\mmodel|_{q:=\top;(q:=q\land (T(\psi))^q)^{n-1}}\right)^q,s\vDash_\IMR \psi\}\\
       =& \{s\in \left(\mmodel|_{q:=\top;(q:=q\land (T(\psi))^q)^{n-1}}\right)^q: \left(\mmodel|_{q:=\top;(q:=q\land (T(\psi))^q)^{n-1}}\right)^q,s\vDash_{\MISL}  T(\psi)\} \\
       =& \{s\in \mmodel|_{q:=\top;(q:=q\land (T(\psi))^q)^{n-1}}: \mmodel|_{q:=\top;(q:=q\land (T(\psi))^q)^{n-1}},s\vDash_{\MISL}  q\textit{ and }\left(\mmodel|_{q:=\top;(q:=q\land (T(\psi))^q)^{n-1}}\right)^q,s\vDash_{\MISL}  T(\psi)\}\\
       =& \{s\in \mmodel|_{q:=\top;(q:=q\land (T(\psi))^q)^{n-1}}: \mmodel|_{q:=\top;(q:=q\land (T(\psi))^q)^{n-1}},s\vDash_{\MISL}  q\land (T(\psi))^q\}\\
       =& \{s\in \mmodel|_{q:=\top;(q:=q\land (T(\psi))^q)^{n-1}}: \mmodel|_{q:=\top;(q:=q\land (T(\psi))^q)^n},s\vDash_{\MISL}  q\}
   \end{align*}
The first equation holds by induction hypothesis and the fact that $q$ does not appear in $\psi$, while the 
second equation holds by $(\ddagger)$. Also, notice that the last set in the sequence above is exactly the domain  of $\left(\mmodel|_{q:=\top;(q:=q\land (T(\psi))^q)^n}\right)^q$. Now it is easy to see that  $\mmodel^{\psi^n}$ and $\left(\mmodel|_{q:=\top;(q:=q\land (T(\psi))^q)^n}\right)^q$ can only disagree on the valuation of $q$. Since $q$ does not appear in $\chi$, we have
    \[
    \mmodel^{\psi^n},s\vDash_\IMR \chi \;\Leftrightarrow\; \left(\mmodel|_{q:=\top;(q:=q\land (T(\psi))^q)^n}\right)^q,s\vDash_\IMR \chi.
    \]
    Therefore, the following sequence holds:
    \begin{align*}
        \mmodel,s\vDash_\IMR [\psi^n]\chi &\;\Leftrightarrow\;  s\not\in \mmodel^{\psi^n}\text{ or }\mmodel^{\psi^n},s\vDash_\IMR \chi\\
        &\;\Leftrightarrow\;  s\not\in\left(\mmodel|_{q:=\top;(q:=q\land (T(\psi))^q)^n}\right)^q\text{ or }\left(\mmodel|_{q:=\top;(q:=q\land (T(\psi))^q)^n}\right)^q,s\vDash_\IMR \chi\\
        &\;\Leftrightarrow\;  s\not\in\left(\mmodel|_{q:=\top;(q:=q\land (T(\psi))^q)^n}\right)^q\text{ or }\left(\mmodel|_{q:=\top;(q:=q\land (T(\psi))^q)^n}\right)^q,s\vDash_{\MISL}  T(\chi) \\
        &\;\Leftrightarrow\;  \mmodel|_{q:=\top;(q:=q\land (T(\psi))^q)^n},s\vDash_{\MISL}  (T(\chi))^q  \\
        &\;\Leftrightarrow\;  \mmodel,s\vDash_{\MISL}  \sub{q:=\top;(q:=q\land (T(\psi))^q)^n}(T(\chi))^q \\
        &\;\Leftrightarrow\; \mmodel,s\vDash_{\MISL}  T(\varphi)
    \end{align*}
    The third equivalence holds by ($\ddagger$), and the fourth one holds by \Cref{lem:recursion_PAL}. Now the proof is completed.
\end{proof}

So far, we have shown that the satisfiability problem for $\MISL$ is undecidable and $\Sigma_1^1$-hard. In what follows, we will show that the problem is in $\Sigma_1^1$, which concludes that $\MISL$ is $\Sigma_1^1$-complete. To do so, we define the {\em evaluation relation}  that intuitively expresses whether or not  a formula $\varphi$ is true at a pointed model: 

\begin{definition}\label{def:evaluation-relation}
Let $\mmodel = (W,R,V)$ be a model and $\mathcal{L}_{\mathsf{clean}}$ be the set of clean $\MISL$-formulas. We define a binary relation $E_\mmodel\subseteq \mathcal{L}_{\mathsf{ clean}}\times W$ such that: 
\begin{center}
  $E_\mmodel(\varphi,s)$\; iff\; $\mmodel,s\vDash_{\MISL}\varphi$.  
\end{center}
\end{definition}

Now, w.r.t. a model $\mmodel$ and its associated $E_\mmodel$, we can show the following: 

\begin{lemma}\label{lemma:sat_relation}
Given a model $\mmodel=(W,R,V)$, for all $s\in W$, it holds that $V(p) = \{s:E_\mmodel(p,s)\}$, and also, all the following properties hold for $E_\mmodel$:
\begin{align*}
 E_\mmodel(\bot,s)  &\leftrightarrow \bot\\
 E_\mmodel(\neg\varphi,s)  & \leftrightarrow \neg E_\mmodel(\varphi,s) \\
E_\mmodel(\varphi_1\land\varphi_2,s)  &\leftrightarrow E_\mmodel(\varphi_1,s)\land E_\mmodel(\varphi_2,s)\\
E_\mmodel(\Box\varphi,s) &\leftrightarrow \forall t\in W\,(Rst\rightarrow E_\mmodel(\varphi,t))\\
E_\mmodel(\sub{p:=\psi}\varphi,s) &\leftrightarrow E_\mmodel(\varphi[\psi/p],s)\\
E_\mmodel(\sub{p:=\psi_0;(p:=\psi_1)^*}\varphi,s) &\leftrightarrow  \exists n\in \mathbb{N}\,E_\mmodel(\sub{p:=\psi_0; (p:=\psi_1)^n}\varphi,s)
\end{align*}
where the main connective of the formula $\varphi$ in the clause  for $E_\mmodel(\sub{p:=\psi}\varphi,s)$ is not an iterative substitution operator.
\end{lemma}

This lemma can be proved  by induction on formulas in $\mathcal{L}_{\mathsf{clean}}$ along $\ll$ (with the help of \Cref{thm:sub_equal_replacement_MISL}). Notice that a formula occurring on the right-hand-side of a condition is `less than' ($\ll$) the formula occurring on the left-hand-side of this condition.  Now we are ready to prove that $\MISL$ is in $\Sigma_1^1$.

\begin{theorem}\label{thm:Sigma-1-1}
The satisfiability problem for $\MISL$ is in $\Sigma_1^1$. 
\end{theorem}

\begin{proof}
Since $\MISL$ can be treated as a fragment of  $\ML^\infty$  (\Cref{thm:frag_ML_infty}) and the latter has the downward L\"owenheim-Skolem property,  
an $\mathcal{L}_\MISL$-formula is satisfiable iff it is satisfiable w.r.t. countable models $\mmodel = (\omega, R, V)$. 

In the sequel, we  only need to consider the clean $\MISL$-formulas of $\mathcal{L}_{\mathsf{clean}}$, since for an arbitrary $\MISL$-formula, we can always recursively transform it into an equivalent clean formula via  (\ref{eqn:MISL_bound_variable_renaming}). Our method is motivated by \cite{IMR}. We show that for any $\varphi\in \mathcal{L}_{\mathsf{ clean}}$, $\varphi$ is satisfiable if and only if there exist relations $R\subseteq \omega\times \omega$ and $E\subseteq \mathcal{L}_{\mathsf{clean}} \times\omega$ such that the equivalences in \Cref{lemma:sat_relation} hold and that $E(\varphi,m)$ for some $m\in\omega$. More specifically, we show that the satisfiability of $\varphi$ is equivalent to the $\Sigma_1^1$-formula~\ref{eqn:sigma_1_1_formula} that is defined in the following:
\begin{equation}\label{eqn:sigma_1_1_formula}\tag{$\varphi$-$\mathsf{SAT}$}
        \begin{split}
            \exists R\exists E\,\forall s\in \omega \, &\\
        &\neg E(\bot,s) \\
        &\wedge\forall \chi\in\mathcal{L}_{\mathsf{clean}}\,(E(\neg\chi,s)\leftrightarrow \neg E(\chi,s))\\
        &\wedge\forall \chi_1, \chi_2\in\mathcal{L}_{\mathsf{clean}}\, (E(\chi_1\land\chi_2,s)\leftrightarrow E(\chi_1,s)\land E(\chi_2,s))\\
        &\wedge\forall \chi\in\mathcal{L}_{\mathsf{clean}}\,\left(E(\Box\chi,s)\leftrightarrow \forall t\in \omega\,(Rst\rightarrow E(\chi,t))\right)\\
        &\wedge\forall \sub{p:=\psi}\chi\in\mathcal{L}_{\mathsf{clean}}\, (E(\sub{p:=\psi}\chi,s)\leftrightarrow E(\chi[\psi/p],s)),\\
        &\quad \text{ if the main connective of $\chi$ is not an iterative substitution operator}\\
        &\wedge\forall\sub{p:=\psi_0;(p:=\psi_1)^*}\chi\in\mathcal{L}_{\mathsf{clean}}\, \\
        &\qquad E(\sub{p:=\psi_0;(p:=\psi_1)^*}\chi,s)\leftrightarrow\exists n\in \mathbb{N}\,E(\sub{p:=\psi_0; (p:=\psi_1)^n}\chi,s)\\
        &\wedge\exists m\in\omega\, E(\varphi,m)
        \end{split}
    \end{equation}
We assume the existence of an encoding of ${\bf P}$ with $\omega$ (e.g., $p_i$ is encoded by $i$), and then the formulas $\neg\chi$, $\chi_1\land\chi_2$, $\Box\chi$, $\sub{p:=\psi}\chi$, $\chi[\psi/p]$ and $\sub{p:=\psi_0; (p:=\psi_1)^*}\chi$ 
involved in  $\Sigma_1^1$-formula~\ref{eqn:sigma_1_1_formula} are actually abbreviations for the recursive functions respectively:
\begin{center}
$f_1:\; \chi\mapsto\neg\chi$ \qquad  $ f_2: \;  \chi_1,\chi_2\mapsto \chi_1\land\chi_2$  \qquad $f_3: \;  \chi\mapsto\Box\chi$ \qquad  $f_4:\; p,\psi,\chi\mapsto\sub{p:=\psi}\chi$ \vspace{2mm}\\
      \qquad $f_5:\; p,\psi,\chi\mapsto \chi[\psi/p]$\qquad $f_6:\; p,\psi_0,\psi_1,\chi\mapsto\sub{p:=\psi_0; (p:=\psi_1)^*}\chi$
\end{center}

Now we proceed to prove that  the satisfiability of $\varphi$ is  equivalent to the $\Sigma_1^1$-formula~\ref{eqn:sigma_1_1_formula}.    
    
    First, assume that $\varphi$ is satisfiable. Then, there is a countable model $\mmodel = (\omega, R,V)$ such that $\mmodel,m\vDash_{\MISL} \varphi$ for some $m\in\omega$. Let $E$ be a relation such that $E(p,s)$ iff $s\in V(p)$, and for any $\chi\in\mathcal{L}_{\mathsf{clean}}$, $E(\chi,s)$ is defined recursively according to the equivalences in~\Cref{lemma:sat_relation}. More specifically, define $E$ to be a relation satisfying the following: 
    \begin{align*}
        E(\bot,s) &:=\bot\\
        E(p,s) &\text{ iff }s\in V(p)\\
        E(\neg\chi,s)&:=\neg E(\chi,s)\\
        E(\chi_1\land\chi_2,s)&:=E(\chi_1,s)\land E(\chi_2,s)\\
        E(\Box\chi,s)&:=\forall t\in W\,(Rst\rightarrow E(\chi,t)) \\    
        E(\sub{p:=\psi}\chi,s)&:= E(\chi[\psi/p],s)\\
        E(\sub{p:=\psi_0;(p:=\psi_1)^*}\chi,s)&:=\exists n\in \mathbb{N} \,E(\sub{p:=\psi_0; (p:=\psi_1)^n}\chi,s)
    \end{align*}
    where the main connective of $\chi$ used in the clause for $E(\sub{p:=\psi}\chi,s)$ is not an iterative substitution operator with  $p$ as its pivot. Note that $E$ is well-defined since the recursion follows a descending chain of the well-founded order $\ll$. It can be proved by induction on formulas along relation $\ll$ that $E$ agrees with $E_\mmodel$ (the evaluation relation defined in \Cref{lemma:sat_relation}), and hence $E(\varphi,m)$ holds. Therefore we can find $R$, $E$ and $m$ that make the $\Sigma_1^1$-formula~\ref{eqn:sigma_1_1_formula} true. 

Moreover, for the other direction, when the $\Sigma_1^1$-formula~\ref{eqn:sigma_1_1_formula} is true for some $R$, $E$ and $m$, we just need to set $\mmodel = (\omega,R,V)$ where $V(p) = \{s:E(p,s)\;\textit{holds}\}$. By induction on formulas along the relation $\ll$, we can show that $\mmodel,m\vDash_{\MISL}\varphi$ is the case. 

To sum up, the satisfiability problem of $\varphi$ is in $\Sigma_1^1$. Therefore $\MISL$  is in $\Sigma_1^1$.
 \end{proof}

With \Cref{thm:translation_IMR,thm:Sigma-1-1}, we have \Cref{thm:undecidability} immediately. Let us end this part with the following:

 \paragraph{The relation between $\MISL$ and $\IMR$} We have already shown that the iterative modal relativization $\IMR$ can be translated to $\MISL$ (\Cref{thm:translation_IMR}). Also, since both  $\MISL$ and $\IMR$ are $\Sigma_1^1$-complete, there is a converse reduction in principle, which we leave for further study:

\vspace{2mm}

\noindent{\bf Open problem}\; Find a sound and faithful translation from $\MISL$ to $\IMR$.

\subsection{Undecidability of \texorpdfstring{$\MISL$}{} on finite tree models}\label{sec:undecidable-tree}

Having seen that the satisfiability problem for $\MISL$ is highly undecidable, in this part we confine ourselves to the class of finite tree models and show that even on such a simple class of models, the logic is undecidable. To achieve this, we will make use of the  {\em  post correspondence problem} (PCP), whose details are as follows:

\begin{definition}[Post correspondence problem]
    Let $\Sigma$ be an alphabet. An {\em instance} of the {post correspondence problem} over alphabet $\Sigma$ is a pair of sequences $(U,V)$ such that $U = (u_1,u_2,\dots, u_m)\in (\Sigma^*)^m$ for some $m\in\mathbb{N}$, $V = (v_1,v_2,\dots,v_m)\in(\Sigma^*)^m$ and the pairs $(u_1, v_1),(u_2, v_2),\dots,(u_m, v_m)$ are distinct. The {\em post correspondence problem} is to find whether there is a finite sequence of pairs $(u_{i_1}, v_{i_1}),(u_{i_2}, v_{i_2}),\dots, (u_{i_l},v_{i_l})$ such that for any $1\le j\le l$, $i_j\in \{1,\dots,m\}$ and $u_{i_1}u_{i_2}\dots u_{i_l} = v_{i_1}v_{i_2}\dots v_{i_l}$, which is called a {\em solution} to the problem. 
\end{definition}

We will work with the PCP over the alphabet $\Sigma = \{a,b\}$ with two elements, which is undecidable \cite{post}. We prove the undecidability of $\MISL$ over finite tree models by showing that  for any instance of PCP over $\Sigma$, there exists an $\MISL$-formula $\varphi$ such that $\varphi$ is satisfied in a finite tree model if and only if there is a solution to the PCP.

\begin{theorem}
    For any instance $(U,V)$ of the post correspondence problem over alphabet $\Sigma=\{a,b\}$, there is an $\MISL$-formula $\varphi(U,V)$ such that 
    \begin{center}
        $\mmodel,w\vDash_{\MISL} \varphi(U,V)$ for some finite tree model $\mmodel$ \quad iff \quad $(U,V)$ has a solution.
    \end{center}
\end{theorem}

\begin{proof}
Before going into the proof details, let us note that in the proof we will assume that models $\mmodel=(W,R_c,R_d,V)$ contain two relations $R_c$ and $R_d$, and corresponding to this, we will use two different modalities $\Diamond_c$ and $\Diamond$ in our language: $\Diamond_c$ moves only along $R_c$ and $\Diamond$ moves only along $R_d$. With the help of propositional letters, the two modalities and the two relations can be encoded by our standard case with only one modality and one relation (cf. e.g., \cite{LHS-journal,Chenqian2023}), but the usage of the two relations would make the proof easier to be understood. Let us now begin.
    
    Suppose that $U = \{u_1,\dots, u_m\}$ and $V = \{v_1,\dots,v_m\}$.  In the sequel we first construct the formula $\varphi(U,V)$ such that if $\mmodel,w\vDash_{\MISL} \varphi(U,V)$ then a solution to the PCP instance $(U,V)$ can be `read' immediately from $(\mmodel,w)$. Before showing the  details of the construction, let us first  explain the key ideas:
    \begin{itemize}
     \item The construction will involve  three kinds of states, including the witness state, the configuration states, and the states on the branches associated to the configuration states. 
        \item   The state $w$ denotes the  `{\em witness state}' in  $\mmodel$, and its  $R_d$-successors are  `{\em configuration states}', whose intuition will be made clear in the item below. Also, we will use $R^*_d$ for the {\em transitive closure of $R_d$},  and use $\Box^*$ for the corresponding modality that is incorporated from $\IMR$: since we have proved that $\IMR$ can be translated into $\MISL$ (\Cref{thm:translation_IMR}), we can use $\Box^*$ directly whenever needed.
        \item Along the relation $R_c$,  configuration states can only reach configuration states, while via $R_d$, each   configuration state is associated with two $R_d$-branches, a `left' branch and a `right' branch, which intuitively stand for a pair of  strings in $(\Sigma^*)^*\times (\Sigma^*)^*$. Given a configuration state, the associated pair of branches is called a `{\em configuration}' of the configuration state .\footnote{It is important to keep in mind that  a configuration state is different from a configuration: a configuration state is a state that leads to a pair of branches, and such a pair is a configuration.}

        \item The formula $\varphi(U,V)$ constructed will apply iterative substitutions to search for configuration states  whose configurations are pairs of the form $(u_{i_1}u_{i_2}\dots u_{i_l}, v_{i_1}v_{i_2}\dots v_{i_l})$, until the configuration of a configuration state is found to be the solution. In what follows, the configuration states whose configurations are of this form will also be called `{\em candidate states}', and corresponding to this, those configurations will also be called `{\em candidates}' (of the solution of PCP).\footnote{In a candidate $(u_{i_1}u_{i_2}\dots u_{i_l}, v_{i_1}v_{i_2}\dots v_{i_l})$, both   the `$u$'-part and   the `$v$'-part contain $l$ strings, but a configuration does not have to be so.}
    \end{itemize}

    We first define the configuration states, and each of them is associated with a left $R_d$-branch and a right $R_d$-branch. Each branch corresponds to  strings $s_1s_2\dots s_l$ with $l\in\mathbb{N}$ and $s_i\in\Sigma^*$. Every node on the left branch is tagged by the proposition $\mathsf{left}$, one of the propositional letters $a$ and $b$, and one of the propositional letters $\mathsf{odd}$ and $\mathsf{even}$ for parities of natural numbers. Similarly for the nodes on the right branch, except that each node on the right branch is tagged by the proposition $\mathsf{right}$. In this way, any path residing in a  branch can be viewed as a string by reading the propositions $a,b$ along the relation $R_d$. The parity propositions $\mathsf{odd}$ and $\mathsf{even}$ denote the separation of strings, which helps to read the strings encoded by a branch. For instance, when the first half $s_1$ of a branch consists of some $\mathsf{odd}$-states and the second half $s_2$ consists of some $\mathsf{even}$-states, we read the branch as $s_1s_2$. 

    Let $\oplus$ denote XOR (i.e., {\em exclusive or}). Let $\mathsf{left}$, $\mathsf{right}$, $\mathsf{odd}$, 
 $\mathsf{even}$,  $a$ and $b$ be propositions that will function as mentioned. Define the formula $\mathsf{Configuration}$ characterizing the configuration states as follows: 
    \begin{align*} \mathsf{Configuration}:=
        &(\neg \mathsf{odd}\land \neg\mathsf{ even}\land \neg \mathsf{left}\land \neg \mathsf{right})\land &  (1)\\
        \big(&\langle p:=\Box\bot\land \mathsf{left}\land (\mathsf{odd}\oplus \mathsf{even})\land (a\oplus b);(p:=\mathsf{left} \land (\mathsf{odd}\oplus \mathsf{even})\land (a\oplus b)\land \Diamond p)^*; & (2)\\
        &q:=\Box\bot\land \mathsf{right}\land (\mathsf{odd}\oplus \mathsf{even})\land (a\oplus b);(q:=\mathsf{right} \land (\mathsf{odd}\oplus \mathsf{even})\land (a\oplus b)\land \Diamond q)^*\rangle &\\
        &\left(\Diamond\top\rightarrow \left(\Diamond p\land \Diamond q\land \Box(p\lor q)\right)\right)\big)\land &\\
        & \bigwedge_{\begin{subarray}{l}
            X\in\{\mathsf{left},\mathsf{right}\},\, Y\in\{\mathsf{odd},\,\mathsf{even}\},\, s\in\{a,b\}
        \end{subarray}}[q:=X\land Y\land s;(q:=\Diamond q)^*](\Diamond q\rightarrow \Box(X\rightarrow q)) &(3)
    \end{align*}
    
The conjunct (2) says that each branch associated to a configuration state is either a $\mathsf{left}$-branch or a $\mathsf{right}$-branch, on which each node is tagged with exactly one of $\{\mathsf{odd}, \mathsf{even}\}$ and with exactly one of $\{a,b\}$. The last conjunct (3) says that all $\mathsf{left}$-branches of the configuration state represent the same sequence of strings, and the same for right branches. Note that for any $n\in\mathbb{N}$, after the substitution $q:=X\land Y\land s;(q:=\Diamond q)^n$, the truth set of $q$ consists of states that can reach a state tagged by $X$, $Y$ and $s$ via an $R_d$-path of length $n$. Moreover, by construction, it might be the case that a configuration state has the empty configuration, and we use $\mathsf{start}$ for this situation, i.e., 
    \[\mathsf{start}:=\neg \mathsf{odd}\land \neg\mathsf{ even}\land \neg \mathsf{left}\land \neg \mathsf{right}\land \Box\bot,\]
which implies $\sf Configuration$. 
    
The witness state only reaches configuration states and the states on the configurations associated to those configuration states  in finitely many steps along $R_d$, and those configuration states only reach configuration states via finitely many steps along $R_c$.  The formula $\mathsf{Witness}$ that characterizes the witness state is defined as follows: 
\[
\mathsf{Witness}:=\Diamond\top\land  \Box( \mathsf{Configuration}\land \Box_c\mathsf{Configuration})
\]
    
 \noindent  A form of the witness state is  depicted in \Cref{fig:witness_state}. 

\begin{figure}
    \centering
\begin{tikzpicture}
\node(w)[circle,draw,inner sep=0pt,minimum size=1mm,fill=black] at (0,0) [label=above:$w$] {};

\node(l)[circle,draw,inner sep=0pt,minimum size=1mm,fill=black] at (-3,-1)[label=above:$c_1$]  {};

\node(r)[circle,draw,inner sep=0pt,minimum size=1mm,fill=black] at (3,-1) [label=above:$c_2$]  {};

\node(ll)[circle,draw,inner sep=0pt,minimum size=1mm,fill=black] at (-4,-2)[label=left:$c_3$] {};

\node(lr)[circle,draw,inner sep=0pt,minimum size=1mm,fill=black] at (-2,-2)[label=right:$c_4$] {};

\node(rl)[circle,draw,inner sep=0pt,minimum size=1mm,fill=black] at (2,-2) [label=left:$c_5$] { };

\node(rr)[circle,draw,inner sep=0pt,minimum size=1mm,fill=black] at (4,-2)[label=right:$c_6$] {};

\draw[->](w) to node [above] {$R_d$} (l);
\draw[->](w) to node [above] {$R_d$}  (r);
\draw[->](l) to node [left] {$R_c$} (ll);
\draw[->](l) to node [right] {$R_c$} (lr);

\draw[->](r) to node [left] {$R_c$} (rl);
\draw[->](r) to node [right] {$R_c$} (rr);

\node(rll)[circle,draw,inner sep=0pt,minimum size=1mm,fill=black] at (1,-3) [label=left:$\mathsf{odd,left}$] {};

\node(rlr)[circle,draw,inner sep=0pt,minimum size=1mm,fill=black] at (3,-3) [label=right:$\mathsf{even,right}$] {};

\node(rll1)[circle,draw,inner sep=0pt,minimum size=1mm,fill=black] at (1,-4) [label=left:$\mathsf{odd,left}$] {};

\node(rlr1)[circle,draw,inner sep=0pt,minimum size=1mm,fill=black] at (3,-4) [label=right:$\mathsf{odd,right}$] {};

\node(rll2)  at (1,-4.3)  {$\vdots$};

\node(rlr2)  at (3,-4.3)  {$\vdots$};

\node(rll3)[circle,draw,inner sep=0pt,minimum size=1mm,fill=black] at (1,-5.5) [label=left:$\mathsf{even,left}$] {};

\node(rlr3)[circle,draw,inner sep=0pt,minimum size=1mm,fill=black] at (3,-5.5) [label=right:$\mathsf{odd,right}$] {};

\draw[->](rl) to  node [left] {$R_d$}  (rll);
\draw[->](rl) to  node [right] {$R_d$} (rlr);
\draw[->](rll) to node [left] {$R_d$}  (rll1);
\draw[->](rlr) to node [right] {$R_d$} (rlr1);

 \draw[->](rll2) to node [left] {$R_d$} (rll3);
\draw[->](rlr2) to node [right] {$R_d$} (rlr3);

\end{tikzpicture}
    \caption{The state $w$ is a witness state, and along  $R_d$, it can only reach configuration states $c_1$ and $c_2$; also, all states $c_3$-$c_6$ that can be reached from $c_1$ and $c_2$ along  $R_c$  are configuration states. For each $1\le i\le 6$, the configuration state $c_i$ is associated with a pair of $R_d$-branches that give the configuration of $c_i$, and in the picture, only the configuration of $c_5$ is depicted, and those for other configuration states are omitted. As we can see, each state on the left branch is labelled with $\mathsf{left}$, and each one on the right is labelled with $\mathsf{right}$. However, the information for $a$ and $b$ is omitted.}
    \label{fig:witness_state}
\end{figure}
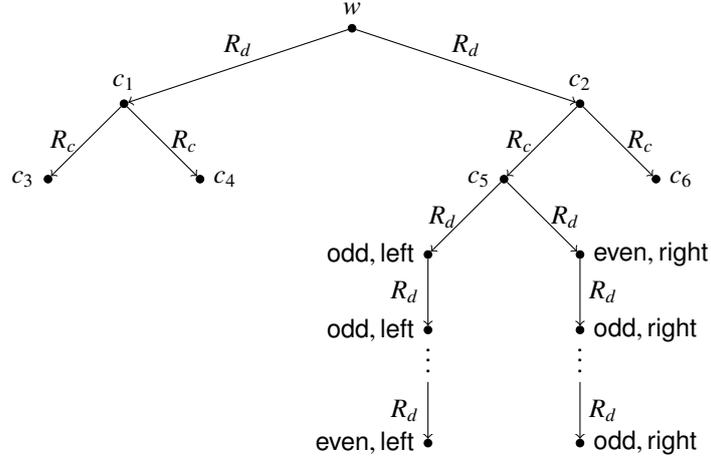

 Now we will construct formulas that are useful in defining candidate states, a special kind of configuration states.
 For $X\in\{\mathsf{left},\mathsf{right}\}$, $Y\in\{\mathsf{odd},\mathsf{even}\}$, and a non-empty $s\in\Sigma^*$, we define $\mathsf{String}(X,Y,s)$ recursively as follows:
    \begin{align*}
        &\text{For }s\in\{a,b\},\,\mathsf{String}(X,Y,s):=X\land Y\land s\land \Box\bot. \\
        &\text{For }k\ge 2,s^i\in\{a,b\},\,\mathsf{String}(X,Y,s^1s^2\dots s^k):=X\land Y\land s^1\land \Diamond \top\land \Box\mathsf{String}(X,Y,s^2\dots s^k). 
    \end{align*}
So,  $\mathsf{String}(X,Y,s)$ is true at states $u$ such that all states on an `$s$-path' (which is an $R_d$-path) starting from $u$ and ending with a dead end are tagged with $X$ and $Y$.  
    

    Next, let us proceed to define candidate states. As mentioned, we call a pair   of strings of form 
    \[
    (u_{i_1}u_{i_2}\dots u_{i_l}, v_{i_1}v_{i_2}\dots v_{i_l})
    \]
     a candidate (of the solution). For any  $l\in\mathbb{N}$, we say that a candidate is of order $l$ if it is a pair   of strings of form $(u_{i_1}u_{i_2}\dots u_{i_l}, v_{i_1}v_{i_2}\dots v_{i_l})$. Meanwhile we say that a candidate state is of order $l$ if its candidate is of order $l$. 

 Now, we are going to recursively construct a class $\{\mathsf{ cand}^{l}: l\in \mathbb{N}\}$ of formulas: intuitively, for each $l\in\mathbb{N}$, $\mathsf{cand}^{l}$ represents a set of candidate states of order $l$.  It is useful to keep in mind that our construction below will ensure that the configuration states in ${\sf cand}^l$ for each $l$ are all candidate states that have the same candidate $(u_{i_1}u_{i_2}\dots u_{i_l},v_{i_1}v_{i_2}\dots v_{i_l})$ for some sequence of $i_1i_2\dots i_l$. 
    For the basic case $l=0$, we define that:
    \[
    \mathsf{cand}^0:=\mathsf{start}
    \]
referring to the  candidate states with the empty string as their candidates. Now, for any $l\in\mathbb{N}$, we present a method to construct $\mathsf{cand}^{l+1}$ from $\mathsf{cand}^l$. 
    
    

Suppose for now that the candidates of all the candidate states in ${\sf cand}^l$ are the same, e.g., the pair of sequences of strings $(U',V')$.\footnote{This property is not yet true but will be stated in the formula ${\sf Append}(i,{\sf cand}^l)$, i.e., if ${\sf Append}(i,{\sf cand}^l)$ is satisfied on any configuration state $\alpha$, it must be the case that the candidates of all candidate states in ${\sf cand}^l$ are the same, and this candidate pair of strings is a prefix of the pair of string encoded by $\alpha$.}  Intuitively, we are going to make the elements in $\mathsf{cand}^{l+1}$ be candidate states with the candidate extending $(U',V')$ with $(u_i,v_i)$ for a fixed $i\in\{1,\dots,m\}$.\footnote{Notice that we have assumed that the instance $(U,V)$ of the post correspondence problem we are working on is given by $U=\{u_1,\dots,u_m\}$ and $V=\{v_1,\dots,v_m\}$, which give us the natural number $m$.} This is characterized by the formula $\mathsf{ Append}(i, \mathsf{cand}^l)$: 
    \begin{align*}
        &\mathsf{Append}(i, \mathsf{ cand}^l):=\Diamond_c\mathsf{cand}^l\land \\
        &\qquad \left(\sub{q:=\mathsf{even}\land\Box\bot;(q:=\Diamond q)^*}\Diamond_c q\rightarrow (\Diamond^*\mathsf{String(left,odd,}u_i)\land \Diamond^*\mathsf{String(right,odd,}v_i))\right)\land\\
        &\qquad \left(\sub{q:=\mathsf{odd}\land \Box\bot;(q:=\Diamond q)^*} \Diamond_c q\rightarrow (\Diamond^*\mathsf{String(left,even,}u_i)\land \Diamond^*\mathsf{String(right,even,}v_i))\right)\land\\
         &\qquad \bigwedge_{\begin{subarray}{l}
        X\in\{\mathsf{left},\mathsf{right}\}\\Y\in\{\mathsf{odd},\mathsf{even}\}\\ s\in\{a,b\}
        \end{subarray}}[q:=X\land Y\land s;(q:=\Diamond q)^*]\left(\Diamond_c q\rightarrow q\right)\land\\
        &\qquad \bigwedge_{\begin{subarray}{l}  X\in\{\mathsf{left},\mathsf{right}\}\\Y,Y'\in\{\mathsf{odd},\mathsf{even}\}\\ s\in\{a,b\}
        \end{subarray}}                
        [q:=X\land Y'\land s\land\Diamond\mathsf{String}(X,Y,w_i^X);(q:=\Diamond q)^*]\left(q\rightarrow \Box_c q\right)
    \end{align*}
    where $w^X_i:=\begin{cases}
        u_i, & \text{ if } X\text{ is }\mathsf{left},\\
        v_i, &\text{ if }  X\text{ is }\mathsf{right}.\end{cases}$

        \vspace{2mm}

The first conjunct in $\mathsf{Append}(i,\mathsf{cand}^l)$ expresses that the current state has at least an $R_c$-successor that is a candidate state in $\mathsf{cand}^l$. 
By the second and the third conjuncts, the configuration of the current state has the form $(U_1u_i,V_1v_i)$ ending with $(u_i,v_i)$ and tagged with the proper parity (the parity of $l+1$).  Since the current state can reach a candidate state in $\mathsf{cand}^l$ via $R_c$ and we have assumed that  the candidate states in $\mathsf{cand}^l$ have the same candidate $(U',V')$,  the fourth conjunct ensures that $(U_1u_i,V_1v_i)$ has $(U',V')$ as its prefix. But by the second and the third conjuncts, the parity of $(u_i,v_i)$ is different from the final segment of $(U',V')$, and so it must be the case that $(U_1,V_1)$ has $(U',V')$ as its prefix. 
By the fifth conjunct, $(U_1,V_1)$ is a prefix of $(U',V')$, and combining the fourth conjunct, we have $(U_1,V_1) = (U',V')$. 
Putting the five conjuncts together, we know that the configuration of the current state is obtained by appending $(u_i,v_i)$ to $(U',V')$. A form of a candidate state satisfying $\mathsf{Append}(i, \mathsf{cand}^l)$ is described in \Cref{fig:candidate_state_l}.

\begin{figure} 
    \centering
\begin{tikzpicture}
\node(A)[circle,draw,inner sep=0pt,minimum size=1mm,fill=black] at (0,0) [label=above:$A$] {};

\node(cl)[circle,draw,inner sep=0pt,minimum size=1mm,fill=black] at (-3,-1) [label=above:$B$] {};

\node(ll1)[circle,draw,inner sep=0pt,minimum size=1mm,fill=black] at (-3.5,-2)  {};

\node(lr1)[circle,draw,inner sep=0pt,minimum size=1mm,fill=black] at (-2.5,-2)  {};

\node(ll2)[circle,draw,inner sep=0pt,minimum size=1mm,fill=black] at (-3.5,-3)  {};

\node(lr2)[circle,draw,inner sep=0pt,minimum size=1mm,fill=black] at (-2.5,-3)  {};

\node(ll3)  at (-3.5,-3.3)  {$\vdots$};

\node(lr3)  at (-2.5,-3.3)  {$\vdots$};

\node(ll4)  [circle,draw,inner sep=0pt,minimum size=1mm,fill=black] at (-3.5,-4.5)  {};

\node(lr4) [circle,draw,inner sep=0pt,minimum size=1mm,fill=black] at (-2.5,-4.5)  {};

\draw[->](A) to node [above] {$R_c$} (cl);

\draw[->](cl) to node [left] {$u_{i_1}$} (ll1);

\draw[->](cl) to node [right] {$v_{i_1}$} (lr1);

\draw[->](ll1) to node [left] {$u_{i_2}$} (ll2);

\draw[->](lr1) to node [right] {$v_{i_2}$} (lr2);

\draw[->](ll3) to node [left] {$u_{i_l}$} (ll4);

\draw[->](lr3) to node [right] {$v_{i_l}$} (lr4);

\node(l1)[circle,draw,inner sep=0pt,minimum size=1mm,fill=black] at (-.5,-1)  {};

\node(r1)[circle,draw,inner sep=0pt,minimum size=1mm,fill=black] at (.5,-1)  {};

\node(l2)[circle,draw,inner sep=0pt,minimum size=1mm,fill=black] at (-.5,-2)  {};

\node(r2)[circle,draw,inner sep=0pt,minimum size=1mm,fill=black] at (.5,-2)  {};

\node(l3)  at (-.5,-2.3)  {$\vdots$};

\node(r3)  at (.5,-2.3)  {$\vdots$};

\node(l4)  [circle,draw,inner sep=0pt,minimum size=1mm,fill=black] at (-.5,-3.5)  {};

\node(r4) [circle,draw,inner sep=0pt,minimum size=1mm,fill=black] at (.5,-3.5)  {};

\node(l5)  [circle,draw,inner sep=0pt,minimum size=1mm,fill=black] at (-.5,-4.5)  {};

\node(r5) [circle,draw,inner sep=0pt,minimum size=1mm,fill=black] at (.5,-4.5)  {};

\draw[->](A) to node [left] {$u_{i_1}$} (l1);

\draw[->](A) to node [right] {$v_{i_1}$} (r1);

\draw[->](l1) to node [left] {$u_{i_2}$} (l2);

\draw[->](r1) to node [right] {$v_{i_2}$} (r2);

\draw[->](l3) to node [left] {$u_{i_l}$} (l4);

\draw[->](r3) to node [right] {$v_{i_l}$} (r4);

\draw[->](l4) to node [left] {$u_{i}$} (l5);

\draw[->](r4) to node [right] {$v_{i}$} (r5);

\node(cf)[circle,draw,inner sep=0pt,minimum size=1mm,fill=black] at (3,-1) [label=above:$C$]  {};

\draw[->](A) to node [above] {$R_c$} (cf);

\node(rl1)[circle,draw,inner sep=0pt,minimum size=1mm,fill=black] at (2.5,-2)  {};

\node(rr1)[circle,draw,inner sep=0pt,minimum size=1mm,fill=black] at (3.5,-2)  {};

\node(rl2)[circle,draw,inner sep=0pt,minimum size=1mm,fill=black] at (2.5,-3)  {};

\node(rr2)[circle,draw,inner sep=0pt,minimum size=1mm,fill=black] at (3.5,-3)  {};

\node(rl3)  at (2.5,-3.3)  {$\vdots$};

\node(rr3)  at (3.5,-3.3)  {$\vdots$};

\node(rl4)  [circle,draw,inner sep=0pt,minimum size=1mm,fill=black] at (2.5,-4.5)  {};

\node(rr4) [circle,draw,inner sep=0pt,minimum size=1mm,fill=black] at (3.5,-4.5)  {};

\draw[->](cf) to node [left] {$u_{i_1}$} (rl1);

\draw[->](cf) to node [right] {$v_{i_1}$} (rr1);

\draw[->](rl1) to node [left] {$u_{i_2}$} (rl2);

\draw[->](rr1) to node [right] {$v_{i_2}$} (rr2);

\draw[->](rl3) to node [left] {$u_{i_l}$} (rl4);

\draw[->](rr3) to node [right] {$v_{i_l}$} (rr4);
\end{tikzpicture}
    \caption{In the picture, links labeled with $R_c$ are edges in $R_c$, links without any relation-label are edges of $R_d$, and the labels $u_{i_j}$ and $v_{i_k}$ of the $R_d$-edges mean that the states on the edges give us strings $u_{i_j}$ and $v_{i_k}$. The state $A$ is a candidate state satisfying $\mathsf{Append}(i, \mathsf{cand}^l)$, and it has an $R_c$-successor $B$ that is a candidate state in $\mathsf{cand}^l$, in which all candidate states have the same candidate. Also, $A$ may have other $R_c$-successors $C$ that are configuration states not in $\mathsf{cand}^l$, but those $C$s have the candidate of $B$ as their candidate. Moreover, for each pair of branches depicted, states on the left branch satisfy the property $\mathsf{left}$, and those on the right branch satisfy the property $\mathsf{right}$.}  
    \label{fig:candidate_state_l}
\end{figure}
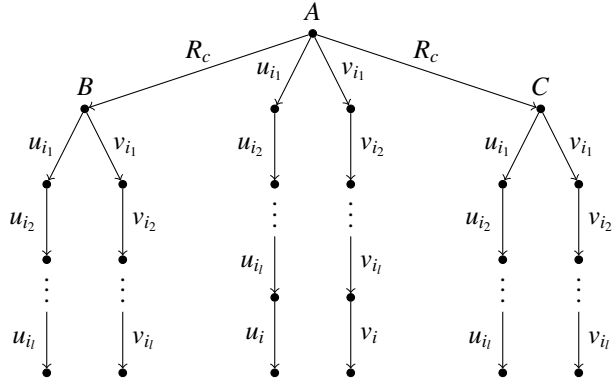

Now we can define a set of candidate states of order $(l+1)$  to be the truth set of $\mathsf{cand}^{l+1}$: 
\[
\mathsf{cand}^{l+1}:= \bigvee_{1\le i\le m}\mathsf{Append}(i,\mathsf{cand}^l) 
\]
Note that $\mathsf{cand}^{l+1}$ only accounts for the candidate states of order $(l+1)$ that have those candidate states in $\mathsf{cand}^l$  as their  $R_c$-successors, but not for arbitrary  candidate states of order $(l+1)$.
    
In the end we need to define when a configuration state can be treated as a solution to the PCP instance. We justify the acceptance of a configuration state by testing whether its `$\mathsf{left}$-string' and `$\mathsf{right}$-string' are identical, using the $\mathsf{Accept}$ formula as follows: 
\[
\mathsf{Accept}:=\Diamond\top \land \left([q:=a;(q:=\Diamond q)^*](\Diamond q\rightarrow \Box q)\right)\land\left([q:=b;(q:=\Diamond q)^*](\Diamond q\rightarrow \Box q)\right) 
\]
The second and the third conjuncts are analogous, and let us explain the second one. It intuitively states that for a given candidate state, if it can reach $a$ in some fixed $n$ steps for an arbitrary number $n$, then it must reach $a$ in $n$ steps on both the $\mathsf{left}$-branch and the $\mathsf{right}$-branch. So, the formula ensures that the $\mathsf{left}$-branch and the $\mathsf{right}$-branch of a candidate state are the same.   
Finally, define the formula $\varphi(U,V)$  as follows: 
\[
\varphi(U,V):= \mathsf{Witness}\land
        \sub{\mathsf{cand}:=\mathsf{cand}^0;(\mathsf{cand}:=\bigvee_{1\le i\le m}\mathsf{Append}(i, \mathsf{cand}))^*}\Diamond(\mathsf{cand}\land \mathsf{Accept})
\]

If $\mmodel,w\vDash_{\MISL} \varphi(U,V)$, then $w$ is the witness state that can reach some candidate states via $R_d$. Also, recall that $\mathsf{cand}^0$ means the set of candidate states having the empty string as their candidate, and $\mathsf{cand}^l$ is the truth set of the variable $\mathsf{cand}$ obtained by  the $l$-th iterative substitution. For all $l\in\mathbb{N}$, as explained for the construction of $\mathsf{Append}(i, \mathsf{ cand}^l)$, the truth set of $\mathsf{cand}^l$ is a subset of the set of all candidate states of order $l$.  Therefore, when $\mmodel,w\vDash_{\MISL} \varphi(U,V)$, $w$ has an $R_d$-successor that is a candidate state of suitable order corresponding to the solution of the PCP instance $(U,V)$. 

On the other hand, assume that $(U,V)$ has a solution 
$(u_{i_1}u_{i_2}\dots u_{i_l}, v_{i_1}v_{i_2}\dots v_{i_l})$.  Consider the model  $\mmodel$ depicted in \Cref{fig:PCP_model}.
In the model, the  witness state  is $w$, which has an $R_d$-successor $c_n$ that is a configuration state,  and there is an $R_c$-sequence $(c_n,\dots,c_0)$ starting from $c_n$ (and so all these states $c_0, \dots,c_n$ are configuration states). According to the construction of $\varphi(U,V)$, it can be verified that the truth set of $\mathsf{ cand}^i$ is $\{c_i\}$ for $0\le i\le l$, and that $\mmodel,c_l\vDash_{\MISL} \mathsf{Accept}$. Therefore $\mmodel,w\vDash_{\MISL} \varphi(U,V)$. 

\begin{figure}
    \centering
\begin{tikzpicture}
\node(w)[circle,draw,inner sep=0pt,minimum size=1mm,fill=black] at (1.5,1.5) [label=above:$w$] {};

\node(cn)[circle,draw,inner sep=0pt,minimum size=1mm,fill=black] at (0,0) [label=above:$c_n$] { };

\node(cn-1)[circle,draw,inner sep=0pt,minimum size=1mm,fill=black] at (-1.5,0) [label=above:$c_{n-1}$] { };

\node(cn-2)  at (-2,0)  {$\dots$};

\node(c1)[circle,draw,inner sep=0pt,minimum size=1mm,fill=black] at (-3.5,0) [label=above:$c_{1}$] { };

\node(c0)[circle,draw,inner sep=0pt,minimum size=1mm,fill=black] at (-5,0) [label=above:$c_{0}$] { };

\node(c1l)[circle,draw,inner sep=0pt,minimum size=1mm,fill=black] at (-4,-1.5)  { };

\node(c1r)[circle,draw,inner sep=0pt,minimum size=1mm,fill=black] at (-3,-1.5)  { };

\node(dots)  at (-1.75,-1.3)  {$\dots$};

\node(cnl)[circle,draw,inner sep=0pt,minimum size=1mm,fill=black] at (-.5,-1.5)  { };

\node(cnr)[circle,draw,inner sep=0pt,minimum size=1mm,fill=black] at (.5,-1.5)  { };

\node(cnl1)[circle,draw,inner sep=0pt,minimum size=1mm,fill=black] at (-.5,-2.5)  { };

\node(cnr1)[circle,draw,inner sep=0pt,minimum size=1mm,fill=black] at (.5,-2.5)  { };

\node(dotsl)  at (-.5,-2.8)  {$\vdots$};

\node(dotsr)  at (.5,-2.8)  {$\vdots$};

\node(cnl2)[circle,draw,inner sep=0pt,minimum size=1mm,fill=black] at (-.5,-4)  { };

\node(cnr2)[circle,draw,inner sep=0pt,minimum size=1mm,fill=black] at (.5,-4)  { };

\draw[->](w) to node [right] {$R_d$}  (cn);
\draw[->](cn) to node [above] {$R_c$}  (cn-1);
\draw[->](cn-2) to node [above] {$R_c$}  (c1);
\draw[->](c1) to node [above] {$R_c$}  (c0);
\draw[->](c1) to node [left] {$\mathsf{odd},u_{i_1}$} (c1l);
\draw[->](c1) to node [right] {$\mathsf{odd},v_{i_1}$}  (c1r);

\draw[->](cn) to node [left] {$\mathsf{odd},u_{i_1}$} (cnl);
\draw[->](cn) to node [right] {$\mathsf{odd},v_{i_1}$}  (cnr);

\draw[->](cnl) to node [left] {$\mathsf{even},u_{i_1}$} (cnl1);
\draw[->](cnr) to node [right] {$\mathsf{even},v_{i_1}$}  (cnr1);

\draw[->](dotsl) to node [left] {$\mathsf{parity}(n),u_{i_n}$}  (cnl2);

\draw[->](dotsr) to node [right] {$\mathsf{parity}(n),v_{i_n}$}  (cnr2);

\end{tikzpicture}
    \caption{A model satisfying $\varphi(U,V)$, given that $(u_{i_1}u_{i_2}\dots u_{i_l}, v_{i_1}v_{i_2}\dots v_{i_l})$ is a solution. In the model, links with the label $R_d$ are edges of $R_d$; links with the label $R_c$ are edges of $R_c$; and links without any relation-label are edges in the transitive closure $R^*_d$ of $R_d$. The labels $\mathsf{even}$ and $\mathsf{odd}$ for links mean that the states on the edges of $R^*_d$ are $\mathsf{even}$ and $\mathsf{odd}$ respectively, and  $\mathsf{parity}(n)$ gives us the parity of the number $n$. Also,  labels $u_{i_j}$ and $u_{v_k}$ for links mean that the states on the edges of $R^*_d$ give us the strings $u_{i_j}$ and $u_{v_k}$ respectively. In the picture, all $c_0,\dots,c_n$ are configurations states, but we only draw the configurations for states $c_0,c_1$ and $c_n$ (notice that $c_0$ has the empty string as its configuration), and the configurations of other configuration states are omitted.}
    \label{fig:PCP_model}
\end{figure}
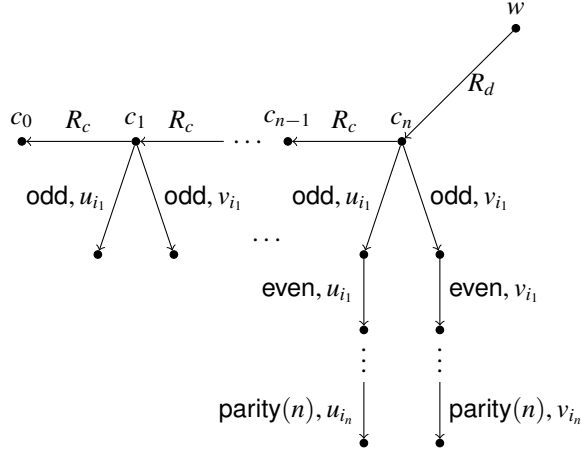

To sum up, given a PCP instance $(U,V)$, $\varphi(U,V)$ is satisfiable in a finite tree model if and only if $(U,V)$ has a solution. Therefore, we conclude that the satisfiability problem for $\MISL$ w.r.t. finite tree models is undecidable.
\end{proof}



\section{Conclusion}\label{sec:conclusion}

 {\bf Summary}\;  Motivated by the ubiquitous applications of substitutions, we develop logical frameworks containing single-step substitutions and iterative substitutions as modal operators that update valuation of propositional variables. 
 
 Our starting point is a simple setting $\MSL$ with only single-step substitution operators. For the logic, we provide a complete proof system and show its decidability. Also, we clarify the differences between those operators with ordinary syntactic replacements.
 
 Then, for the more intricate proposal $\MISL$ containing iterative substitution operators, we develop various validities to show how those operators work and analyze its applications to crucial notions in games. Besides, we compare $\MISL$ with many other important logics that share a similar iterative feature, including the modal $\mu$-calculus, the infinitary modal logic $\ML^{\infty}$, the propositional dynamic logic $\PDL$ and the iterative modal relativization $\IMR$, which also illustrates advantages of our logic. Moreover, we explore suitable criteria to measure the expressiveness of the logics and study the computational behavior of both $\MISL$ and its  restriction to the class of finite tree models.


\vspace{3mm}

 {\bf Further directions}\;  Let us end by a few directions that are worth pursuing in future. Several open problems have been identified along the way, including relations between $\MISL$ and other relevant logics. In addition to the logic frameworks mentioned here, it is crucial to recognize that there are many other important logics that have the iterative reasoning concepts, including   modal logics with inflationary fixed-point $\MIC$ \cite{MIC} and  the modal logic of oscillations \cite{Benthem2015OscillationsLA}. Some of them can also be translated into $\MISL$ or its further extensions,\footnote{For instance,  the framework $1\MIC$ (i.e., $\MIC$ without simultaneous inductions) can be embedded into $\MISL$ (when only finite sequences of approximation are allowed) and the oscillation operators can be defined by  the extension of $\MISL$ with the universal modality \cite{open-minds}, but we leave the details to other occasions.} but the precise connections remain to be determined. Also, the explorations in the article have  a model-theoretic approach, and it is important to study the logical proposals from a proof-theoretic perspective as well, e.g., sequent calculi. Finally, going beyond the current framework, it is   meaningful to allow ordinal sequences of iteration, and based on the simultaneous one-step substitutions \cite{public-assignment,public-assignment-journal,kooi-substitution-formulas}, one can also study the iterative generalization  like $\sub{(p:=\psi;q=\chi)^*}\varphi$, to make the logics applicable in broader scenarios.

\bigskip

\bibliographystyle{elsarticle-num}
\bibliography{mybib}

\end{document}